\documentclass[11pt]{article}

\usepackage[normalem]{ulem}
\usepackage[justification=justified,singlelinecheck=false]{caption}

\usepackage[OT1]{fontenc}
\usepackage{amsthm,amsmath,amssymb}
\usepackage{natbib}
\usepackage{multirow}
\usepackage[pdftex]{graphicx}
\usepackage{subfigure}
\usepackage{makecell}
\usepackage{booktabs}
\usepackage{array}
\usepackage{fullpage}
\usepackage{url}
\usepackage{algorithm}
\usepackage{algorithmic}
\usepackage{bm}
\usepackage{nonfloat}
\usepackage{smile}

\usepackage[colorlinks,
            linkcolor=red,
            anchorcolor=blue,
            citecolor=blue
            ]{hyperref}

\usepackage{natbib}
\usepackage{multirow}

\DeclareMathOperator*{\argzero}{argzero}
\DeclareMathOperator*{\supp}{\text{supp}}

\DeclareMathOperator{\gOj}{\widehat{\bgamma}^{\lambda}(\cD_{j})}

\renewcommand{\hat}{\widehat}
\DeclareMathOperator{\nab}{\nabla}
\renewcommand{\epsilon}{\varepsilon}
\newcommand{\sj}{^{(j)}}

\begin{document}

%\title{\Huge Split Sample Asymptotic Theory for Distributed Estimation and Hypothesis Testing}

\title{\Huge Distributed Estimation and Inference with Statistical Guarantees}

%\title{\Huge Oracle Efficiency in Distributed Estimation and Hypothesis Testing}

\author{Heather Battey\thanks{Department of Operations Research and Financial Engineering, Princeton University, Princeton, NJ 08540; Email: $\{$\texttt{hbattey},\texttt{jqfan},\texttt{hanliu},\texttt{junweil},\texttt{ziweiz}$\}$ \texttt{@princeton.edu}}\;\thanks{Department of Mathematics, Imperial College London, London, SW7 2AZ; Email: \texttt{h.battey@imperial.ac.uk}\;
Heather Battey was supported in part by the NIH grant 2R01-GM072611-11.
Jianqing Fan was supported in part by NSF Grants DMS-1206464 and  DMS-1406266, and
NIH grants 2R01-GM072611-11.   }
  \and Jianqing Fan$^{*}$ \and Han Liu$^{*}$ \and Junwei Lu$^{*}$ \and Ziwei Zhu$^{*}$}

\normalsize{ }

\maketitle

\begin{abstract}
This paper studies hypothesis testing and parameter estimation in the context of the divide and conquer algorithm. In a unified likelihood based framework, we propose new test statistics and point estimators obtained by aggregating various statistics from $k$ subsamples of size $n/k$, where $n$ is the sample size. In both low dimensional and high dimensional settings, we address the important question of how to choose $k$ as $n$ grows large, providing a theoretical upper bound on $k$ such that the information loss due to the divide and conquer algorithm is negligible. In other words, the resulting estimators have the same inferential efficiencies and estimation rates as a practically infeasible oracle with access to the full sample. Thorough numerical results are provided to back up the theory.
\end{abstract}

\section{Introduction}

In recent years, the field of statistics has developed apace in response to the opportunities and challenges spawned from the `data revolution', which marked the dawn of an era characterized by the availability of enormous datasets. An extensive toolkit of methodology is now in place for addressing a wide range of high dimensional problems, whereby the number of unknown parameters, $d$, is much larger than the number of observations, $n$. However, many modern datasets are instead characterized by $n$ and $d$ both large. The latter presents intimidating practical challenges resulting from storage and computational limitations, as well as numerous statistical challenges \citep{FHL14}. It is important that statistical methodology targeting modern application areas does not lose sight of the practical burdens associated with manipulating such large scale datasets. In this vein, incisive new algorithms have been developed for exploiting modern computing architectures and recent advances in distributed computing. These algorithms enjoy computational efficiency and facilitate data handling and storage, but come with a statistical overhead if inappropriately tuned.

% computational challenge, for which a new theoretical objective has emerged: to quantify the statistical cost associated with algorithms that enjoy computational efficiency in these large scale problems.

%One class of algorithms enjoying these computational efficiencies is the class of divide and conquer algorithms.
% potential synergies from conjoint development of statistical methodology with algorithm
%With a reckognition of the need to be mindful of algorithmic difficulties

With increased mindfulness of the algorithmic difficulties associated with large datasets, the statistical community has witnessed a surge in recent activity in the statistical analysis of various divide and conquer (DC) algorithms, which randomly partition the $n$ observations into $k$ subsamples of size $n_{k}=n/k$, construct statistics based on each subsample, and aggregate them in a suitable way. In splitting the dataset, a single, very large scale estimation or inference problem with computational complexity $O(\gamma(n))$, for a given function $\gamma(\cdot)$ that depends on the underlying problem, is transformed into $k$ high dimensional (large $d$ smaller $n_{k}$) problems each with computational complexity $O\bigl(\gamma(n/k)\bigr)$ on each machine.    What gets lost in this process is the interactions of split subsamples in each machine.  They are not recoverable.  However, the information got lost is not much statistically, as the spilt subsamples are supposed to be independent.
It is thus of significant practical interest to derive a theoretical upper bound on the number of subsamples $k$ that delivers the same statistical performance as the computationally infeasible ``oracle'' procedure based on the full sample. We develop communication efficient generalizations of the Wald and Rao score tests for the high dimensional scheme, as well as communication efficient estimators for the parameters of the high dimensional and low dimensional linear and generalized linear models. In all cases we give the upper bound on $k$ for preserving the statistical error of the analogous full sample procedure.

While hypothesis testing in a low dimensional context is straightforward, in the high dimensional setting, nuisance parameters introduce a non-negligable bias, causing classical low dimensional theory to break down. In our high dimensional Wald construction, the phenomenon is remedied through a debiasing of the estimator, which gives rise to a test statistic with tractable limiting distribution, as documented in the $k=1$ (no sample split) setting in \citet{ZhangZhang2014} and \citet{vdGB2014}. For the high dimensional analogue of Rao's score statistic, the incorporation of a correction factor increases the convergence rate of higher order terms, thereby vanquishing the effect of the nuisance parameters. The approach is introduced in the $k=1$ setting in \citet{NingLiu2014b}, where the test statistic is shown to possess a tractable limit distribution. 
However, the computation complexity for the debiased estimators increases by an order of magnitude, due to solving $d$ high-dimensional regularization problems.  This motivates us to appeal to the divide and conquer strategy.

We develop the theory and methodology for DC versions of these tests. In the  case $k=1$, each of the above test statistics can be decomposed into a dominant term with tractable limit distribution and a negligible remainder term. The DC extension requires delicate control of these remainder terms to ensure the error accumulation remains sufficiently small so as not to materially contaminate the leading term. In obtaining the upper bound on the number of permitted subsamples, $k$, we provide an upper bound on $k$ subject to a statistical guarantee. We find that the theoretical upper bound on the number of subsamples guaranteeing the same inferential or estimation efficiency as the procedure without DC is $k=o((s\log d)^{-1}\sqrt{n})$ in the linear model, where $s$ is the sparsity of the parameter vector. In the generalized linear model the scaling is $k=o(((s\vee s_{1})\log d)^{-1}\sqrt{n})$, where $s_{1}$ is the sparsity of the inverse information matrix.

For high dimensional estimation problems, we use the same debiasing trick introduced in the high dimensional testing problems to obtain a thresholded divide and conquer estimator that achieves the full sample minimax rate. The appropriate scaling is found to be $k=O(\sqrt{n/(s^{2}\log d)})$ for the linear model and $k=O(\sqrt{n/((s\vee s_{1})^2\log d)})$ for the generalized linear model. Moreover, we find that the loss incurred by the divide and conquer strategy, as quantified by the distance between the DC estimator and the full sample estimator, is negligible in comparison to the statistical error of the full sample estimator provided that $k$ is not too large. In the context of estimation, the optimal scaling of $k$ with $n$ and $d$ is also developed for the low dimensional linear and generalized linear model. This theory is of independent interest. It also allows us to study a refitted estimation procedure under a minimal signal strength assumption.

A partial list of references covering DC algorithms from a statistical perspective is \citet{Xie2012}, \citet{Duchi2014}, \citet{Kleiner2014}, and \citet{TianqiGuangHan2014}. For the high dimensional estimation setting, the same debiasing approach of \cite{vdGB2014} is proposed independently by \citet{LeeSunLiuTaylor2015} for divide-and-conquer estimation.
%, who obtain similar conclusions regarding the scaling of $k$ with $n$ and $d$.
Our paper differs from that of \citet{LeeSunLiuTaylor2015} in that we additionally cover high dimensional hypothesis testing and refitted estimation in the DC setting. Our results on hypothesis testing reveals a different phenomenon to that found in estimation, as we observe through the different requirements on the scaling of $k$. On the estimation side, our results also differ  from those of \citet{LeeSunLiuTaylor2015} in that the additional refitting step allows us to achieve the oracle  rate under the same scaling of $k$.

%construction jointly exploits a score based debiased estimator of the parameter of interest, and a particular plug-in functional of the nuisance parameter estimator. The form of the plug-in likelihood ratio induces the same convergence rate acceleration to allow the leading term to dominate the term involving the nuisance parameters

The rest of the paper is organized as follows. Section \ref{sectionBackground} collects notation and details of a generic likelihood based framework. Section \ref{sectionTesting} covers testing, providing high dimensional DC analogues of the Wald test (Section \ref{sectionWald}) and Rao score test (Section \ref{sectionScore}), in each case deriving a tractable limit distribution for the corresponding test statistic under standard assumptions. Section \ref{sectionEstimation} covers distributed estimation, proposing an aggregated estimator of $\bbeta^{*}$ for low dimensional and high dimensional linear and generalized linear models, as well as a refitting procedure that improves the estimation rate, with the same scaling, under a minimal signal strength assumption. Section \ref{sectionNumerical} provides numerical experiments to back up the developed theory. In Section \ref{sectionDiscussion} we discuss our results together with remaining future challenges. Proofs of our main results are collected in Section \ref{sectionProofs}, while the statement and proofs of a number of technical lemmas are deferred to the appendix.

\section{Background and Notation}\label{sectionBackground}

We first collect the general notation, before providing a formal statement of our statistical problems. More specialized notation is introduced in context.

\subsection{Generic Notation}

We adopt the common convention of using boldface letters for vectors only, while regular font is used for both matrices and scalars, with the context ensuring no ambiguity. $|\cdot|$ denotes both absolute value and cardinality of a set, with the context ensuring no ambiguity. For $\bv = (v_1,\ldots,v_d)^{T} \in \RR^{d}$, and $1 \leq q \leq \infty$, we define $\|\bv\|_{q} = \bigl(\sum_{j=1}^{d}|v_{j}|^{q}\bigr)^{1/q}$, $\|\bv\|_{0}=|\supp(\bv)|$, where $\supp(\bv) = \{j : v_{j}\neq 0\}$ and $|A|$ is the cardinality of the set $A$. Write $\|\bv\|_{\infty} = \max_{1\leq j\leq d} |v_{j}|$, while for a matrix $M = [M_{jk}]$, let $\|M\|_{\max} = \max_{j,k} |M_{jk}|$, $\|M\|_{1}=\sum_{j,k}|M_{jk}|$. For any matrix $M$ we use $\bM_{\ell}$ to index the transposed $\ell^{th}$ row of $M$ and $[\bM]_{\ell}$ to index the $\ell^{th}$ column. The sub-Gaussian norm of a scalar random variable $X$ is defined as $\|X\|_{\psi_{2}} = \sup_{q\geq1}q^{-1/2}(\EE|X|^{q})^{1/q}$. For a random vector $\bX \in \RR^{d}$, its sub-Gaussian norm is defined as $\|\bX\|_{\psi_{2}} = \sup_{x\in \mathbb{S}^{d-1}} \|\langle \bX, \bx \rangle\|_{\psi_{2}}$, where $\mathbb{S}^{d-1}$ denotes the unit sphere in $\RR^{d}$. Let $I_{d}$ denote the $d\times d$ identity matrix; when the dimension is clear from the context, we omit the subscript. We also denote the Hadamard product of two matrices $A, B$ as $A \circ B$ and  $(A \circ B)_{jk}= A_{jk} B_{jk}$ for any $j,k$. $\{\be_{1},\ldots,\be_{d}\}$ denotes the canonical basis for $\RR^{d}$. For a vector $\bv\in\RR^{d}$ and a set of indices $\mathcal{S}\subseteq \{1,\ldots, d\}$, $\bv_{\mathcal{S}}$ is the vector of length $|\mathcal{S}|$ whose components are $\{v_{j}:j\in\mathcal{S}\}$. Additionally, for a vector $\bv$ with $j^{\text{th}}$ element $v_{j}$, we use the notation $\bv_{-j}$ to denote the remaining vector when the $j^{\text{th}}$ element is removed. With slight abuse of notation, we write $\bv=(v_{j},\bv_{-j})$ when we wish to emphasize the dependence of $\bv$ on $v_{j}$ and $\bv_{-j}$ individually. The gradient of a function $f(\bx)$ is denoted by $\nab f(\bx)$, while $\nab_{\bx}f\bigl((\bx,\by)\bigr)$ denotes the gradient of $f\bigl((\bx,\by)\bigr)$ with respect to $\bx$, and $\nab_{\bx\by}^{2}f\bigl((\bx,\by)\bigr)$ denotes the matrix of cross partial derivatives with respect to the elements of $\bx$ and $\by$. For a scalar $\eta$, we simply write $f'(\eta):=\nab_{\eta}f(\eta)$ and $f''(\eta):=\nab^{2}_{\eta\eta}f(\eta)$. %For two positive sequences $a_n$ and $b_n$, we write $a_n\asymp b_n$ if $C\leq a_{n}/b_{n} \leq C'$ for some $C, C' >0$ and we write $a_{n}\lesssim b_{n}$ if $a_{n}\leq C b_{n}$, i.e. if $a_{n}=O(b_{n})$.
For a random variable $X$ and a sequence of random variables, $X_{n}$, we write $X_{n} \leadsto X$ when $X_{n}$ converges weakly to $X$. If $X$ is a random variable with standard distribution, say $F_{X}$, we simply write $X_{n}\leadsto F_{X}$. Given $a,b \in \RR$, let $a\vee b$ and $a\wedge b$ denote the maximum and minimum of $a$ and $b$.
%Let $a_{n}$ be a positive nonstochastic sequence. We write $X_{n}=O_{\PP}(a_{n})$ if for any $\epsilon>0$ there exists a $C<\infty$ such that $\PP\bigl(|X_{n}|> C a_{n}\bigr)\le \epsilon$ for all $n$, and $X_{n}=o_{\PP}(a_{n})$ if $\PP\bigl(|X_{n}|/a_{n}>\epsilon)\rightarrow 0$ as $n\rightarrow \infty$ for any given $\epsilon>0$.
We also make use of the notation $a_{n}\lesssim b_{n}$ ($a_{n}\gtrsim b_{n}$) if $a_{n}$ is less than (greater than) $b_{n}$ up to a constant, and $a_{n} \asymp b_{n}$ if $a_{n}$ is the same order as $b_{n}$.  Finally, for an arbitrary function $f$, we use $\argzero_{\theta}f(\theta)$ to denote the solution to $f(\theta)=0$.

\subsection{General Likelihood based Framework}\label{sec2.2}

Let $(\bX_{1}^{T},Y_{1})^{T}, \ldots, (\bX_{n}^{T},Y_{n})^{T}$ be $n$ i.i.d.~copies of the random vector $(\bX^{T},Y)^{T}$, whose realizations take values in $\RR^{d}\times \mathcal{Y}$. Write the collection of these $n$ i.i.d.~random couples as $\mathcal{D}=\{(\bX_{1}^{T},Y_{1})^{T},\ldots,(\bX_{n}^{T},Y_{n})^{T}\}$ with $\bY=(Y_{1},\ldots,Y_{n})^{T}$ and $X=(\bX_{1},\ldots,\bX_{n})^{T}\in\RR^{n\times d}$. Conditional on $\bX_{i}$, we assume $Y_{i}$ is distributed as $F_{\bbeta^{*}}$ for all $i\in\{1,\ldots,n\}$, where $F_{\bbeta^{*}}$ has a density or mass function $f_{\bbeta^{*}}$. We thus define the negative log-likelihood function, $\ell_{n}(\bbeta)$, as
\begin{equation}\label{LLike}
\ell_{n}(\bbeta)=\frac{1}{n}\sum_{i=1}^{n}\ell_{i}(\bbeta)=-\frac{1}{n}\sum_{i=1}^{n}\log f_{\bbeta}(Y_{i}|\bX_{i}).
\end{equation}
We use $J^{*}=J(\bbeta^{*})$ to denote the information matrix and $\Theta^{*}$ to denote $(J^{*})^{-1}$, where $J(\bbeta)=\EE \bigl[\nab_{\bbeta\bbeta}^{2}\ell_{n}(\bbeta)\bigr]$.

For testing problems, our goal is to test $H_{0}:\beta_{v}^{*}=\beta_{v}^{H}$ for any $v\in\{1,\ldots, d\}$. We partition $\bbeta^{*}$ as $\bbeta^{*}=(\beta_{v}^{*},\bbeta_{-v}^{*T})^{T}\in \RR^{d}$, where $\bbeta_{-v}^{*}\in \RR^{d-1}$ is a vector of nuisance parameters and $\beta_{v}^{*}$ is the parameter of interest. To handle the curse of dimensionality, we exploit a  penalized M-estimator defined as,
\begin{align}\label{eqMEstimator}
\hat{\bbeta}^{\lambda} = \argmin_{\bbeta} \left\{ \ell_n(\bbeta) + \mathcal{P}_{\lambda}(\bbeta)\right\},
\end{align}
with $\mathcal{P}_\lambda(\bbeta)$ a sparsity inducing penalty function with a regularization parameter $\lambda$. Examples of $\mathcal{P}_{\lambda}(\bbeta)$ include the convex $\ell_1$ penalty, $\mathcal{P}_{\lambda}(\bbeta) = \lambda \|\bbeta\|_1 =  \lambda \sum_{v=1}^{d}|\beta_v|$ which, in the context of the linear model, gives rise to the LASSO estimator \citep{TibshiraniLasso1996},
\begin{equation}\label{eqDefinitionLASSO}
\hat{\bbeta}^{\lambda}_{\text{LASSO}} = \argmin_{\bbeta}\left\{ \frac{1}{2n} \|\bY - X \bbeta\|_2^2 + \lambda \|\bbeta\|_1\right\}.
\end{equation}
Other penalties include folded concave penalties such as the smoothly clipped absolute deviation (SCAD) penalty~\citep{Fan01} and minimax concave MCP penalty \citep{ZhangCH2010}, which eliminate the estimation bias and attain the oracle rates of convergence~\citep{LohWW2013,WangLiuZhang2014}. The SCAD penalty is defined as
\begin{align}
\mathcal{P}_{\lambda}(\bbeta)=\sum_{v=1}^{d}p_{\lambda}(\beta_{v}), \text{ where } p_{\lambda}(t) = \int_{0}^{|t|} \biggl\{\lambda \one(z\leq \lambda) + \frac{a\lambda -z}{a-1} \one(z > \lambda)\biggr\}dz,
\end{align}
for a given parameter $a > 0$ and MCP penalty is given by
\begin{align}\label{eqMCP}
\mathcal{P}_{\lambda}(\bbeta)=\sum_{v=1}^{d}p_{\lambda}(\beta_{v}), \text{ where } p_{\lambda}(t)=\lambda \int_0^{|t|} \left(1-\frac{z}{\lambda b}\right)_+dz %=\left(\lambda |t| - \frac{t^2}{2b}\right)\mathbf{1}(|t|\leq b\lambda) + \frac{b\lambda^2}{2}\mathbf{1}(|t|>b\lambda),
\end{align}
where $b>0$ is a fixed parameter.
The only requirement we have on $\mathcal{P}_{\lambda}(\bbeta)$ is that it induces an estimator satisfying the following condition.

\begin{condition}\label{con:estimation}
For any $\delta\in (0,1)$, if $\lambda \asymp \sqrt{\log (d/\delta)/n}$,
\begin{equation}\label{con:lasso}
\PP\Bigl(\|\widehat{\bbeta}^{\lambda} - \bbeta^{*}\|_{1} > C s n^{-1/2}\sqrt{\log (d/\delta)}\Bigr)\le \delta,
\end{equation}
where $s$ is the sparsity of $\bbeta^{*}$, i.e.,~$s=\|\bbeta^{*}\|_{0}$.
\end{condition}
Condition \ref{con:estimation} holds for the LASSO, SCAD and MCP.  See \cite{vdGBBook2011, Fan01, ZhangCH2010} respectively and \cite{ZZh12}.

The DC algorithm randomly and evenly partitions $\mathcal{D}$ into $k$ disjoint subsets $\mathcal{D}_{1},\ldots,\mathcal{D}_{k}$, so that $\mathcal{D}=\cup_{j=1}^{k}\mathcal{D}_{j}$, $\mathcal{D}_{j}\cap \mathcal{D}_{\ell}=\emptyset$ for all $j,\ell\in\{1,\ldots,k\}$, and $|\mathcal{D}_{1}|=|\mathcal{D}_{2}|=\cdots=|\mathcal{D}_{k}|=n_{k}=n/k$, where it is implicitly assumed that $n$ can be divided evenly. Let $\mathcal{I}_{j}\subset\{1,\ldots,n\}$ be the index set corresponding to the elements of $\mathcal{D}_{j}$. Then for an arbitrary $n\times d$ matrix $A$, $A^{(j)}=[A_{i\ell}]_{i\in\mathcal{I}_{j}, 1\le\ell\le d}$.  For an arbitrary estimator $\hat{\tau}$, we write $\hat{\tau}(\mathcal{D}_{j})$ when the estimator is constructed based only on $\mathcal{D}_{j}$. What information gets lost in this process is the interactions of data across subsamples $\{\mathcal{D}_j\}_{j=1}^{n/k}$.  Taking the oridinary least-squares regression, for example, the cross-covariances of the subsamples will not be able to get recovered.   However, they do not contain much information about the unknown parameters, as the subsamples are nearly independent.  Finally, we write $\ell_{n_{k}}\sj(\bbeta)=\sum_{i\in \mathcal{I}_{j}}\ell_{i}(\bbeta)$ to denote the negative log-likelihood function of equation \eqref{LLike} based on $\cD_{j}$.

While the results of this paper hold in a general likelihood based framework, for simplicity we state conditions at the population level for the generalized linear model (GLM) with canonical link. A much more general set of statements appear in the auxiliary lemmas upon which our main results are based. Under GLM with the canonical link, the response follows the distribution,
	\begin{equation}
		\label{eq:glm}
		f_n(\bY; X, {\bbeta^*})=\prod\limits_{i=1}^n f(Y_i; \eta^*_i)=\prod\limits_{i=1}^n\left\{c(Y_i)\exp\left[\frac{Y_i\eta^*_i-b(\eta^*_i)}{\phi}\right]\right\},
	\end{equation}
	where $\eta_i^*=\bX_i^T\bbeta^*$. The negative log-likelihood corresponding to (\ref{eq:glm}) is given, up to an affine transformation, by
	\begin{equation}
		\label{eqGLMLikelihood}
		\ell_n(\bbeta)=\frac{1}{n}\sum\limits_{i=1}^n -Y_i\bX_i^T\bbeta+b(\bX_i^T\bbeta)=\frac{1}{n}\sum\limits_{i=1}^n -Y_i\eta_i+b(\eta_i)=\frac{1}{n}\sum\limits_{i=1}^n \ell_i(\bbeta),
	\end{equation}
	and the gradient and Hessian of $\ell_n(\bbeta)$ are respectively
	\[
			\nabla\ell_n(\bbeta)=-\frac{1}{n} X^T(\bY-\bmu(X \bbeta))\quad\text{and}\quad
			\nabla^{2}\ell_n(\bbeta)=\frac{1}{n} X^TD(X \bbeta)X,
	\]
	where $\bmu(\bbeta)=(b'(\eta_1), \ldots,b'(\eta_n))^T$ and $D(\bbeta)=\diag\{b''(\eta_1),\ldots,b''(\eta_n)\}$. In this setting, $J(\bbeta)=\EE[b''(\bX_1^T\bbeta)\bX_1\bX_1^T]$ and $J^*=\EE[b''(\bX_1^T\bbeta^*)\bX_1\bX_1^T]$.

\section{Divide and Conquer Hypothesis Tests}\label{sectionTesting}

In the context of the two classical testing frameworks, the Wald and Rao score tests, our objective is to construct a test statistic $\overline{S}_{n}$ with low communication cost and a tractable limiting distribution $F$. Let $\beta_{v}^{*}$ be the $v^{th}$ component of $\bbeta^{*}$. From this statistic we define a test of size $\alpha$ of the null hypothesis, $H_{0}:\beta_{v}^{*}=\beta_{v}^{H}$, against the alternative, $H_{1}:\beta_{v}^{*}\neq \beta_{v}^{H}$, as a partition of the sample space described by
\begin{equation}\label{sizeAlphaTest}
T_{n}^{\alpha}=\left\{\begin{array}{rl}0 & \text{if } |\overline{S}_{n}|\leq F^{-1}(1-\alpha/2) \\
															1 & \text{if } |\overline{S}_{n}|> F^{-1}(1-\alpha/2)	\end{array}\right.
\end{equation}
for a two sided test.

\subsection{Two Divide and Conquer Wald Type Constructions}\label{sectionWald}

For the high dimensional linear model, \citet{ZhangZhang2014}, \citet{vdGB2014} and \citet{JMJMLR2014} propose methods for debiasing the LASSO estimator with a view to constructing high dimensional analogues of Wald statistics and confidence intervals for low-dimensional coordinates.  As pointed out by \citet{ZhangZhang2014}, the debiased estimator does not impose the minimum signal condition used in establishing oracle properties of regularized estimators \citep{Fan01, FLv11,
PoLing2015, Wang2014, ZZh12} and hence has wider applicability than those inferences based on the oracle properties.
The method of \citet{vdGB2014} is appealing in that it accommodates a general penalized likelihood based framework, while the \citet{JMJMLR2014} approach is appealing in that it optimizes asymptotic variance and % is designed for conditional inference and thus
requires a weaker condition  than  \citet{vdGB2014} in the specific case of the linear model. We consider the DC analogues of \citet{JMJMLR2014} and \citet{vdGB2014} in Sections \ref{SectionWaldJM} and \ref{sectionWaldvdG} respectively.

\subsubsection{LASSO based Wald Test for the Linear Model}\label{WaldLinearModel}\label{SectionWaldJM}

The linear model assumes
\begin{equation}
	\label{eq:lm}
	Y_{i} = \bX_{i}^{T}\bbeta^{*} + \varepsilon_{i},
\end{equation}
where $\{\varepsilon_{i}\}_{i=1}^{n}$ are i.i.d.~with mean zero and variance $\sigma^{2}$. For concreteness, we focus on a LASSO based method, but our procedure is also valid when other pilot estimators are used. We describe a modification of the bias correction method introduced in \citet{JMJMLR2014} as a means to testing hypotheses on low dimensional coordinates of $\bbeta^{*}$ via pivotal test statistics.

On each subset $\mathcal{D}_j$, we compute the debiased estimator of $\bbeta^{*}$ as in \citet{JMJMLR2014} as
\begin{equation}\label{splitwiseDebiased}
\hat{\bbeta}^{d}(\cD_{j})=\hat{\bbeta}^{\lambda}_{\text{LASSO}}(\cD_{j})+\frac{1}{n_{k}}M^{(j)}\bigl(X^{(j)}\bigr)^{T}\bigl(Y^{(j)}-X^{(j)}\hat{\bbeta}^{\lambda}_{\text{LASSO}}(\cD_{j})\bigr),
\end{equation}
where the superscript $d$ is used to indicate the debiased version of the estimator, $M^{(j)}= (\mb_1^{(j)}, \ldots, \mb_d^{(j)})^T$ and $\mb_{v}$ is the solution of
\begin{eqnarray} \label{eq:m-est}
\mathbf{m}_{v}\sj = \argmin_{\mathbf{m}} \mathbf{m}^{T} \hat{\Sigma}^{(j)} \mathbf{m} \;\; &\text{s.t.} & \|\hat{\Sigma}\sj\mathbf{m} - \be_v\|_\infty \leq \vartheta_{1}, \\
\nonumber 			  &               & \|X^{(j)}\mb\|_{\infty}\leq \vartheta_{2},	
\end{eqnarray}
where the choice of tuning parameters $\vartheta_{1}$ and $\vartheta_{2}$ is discussed in \citet{JMJMLR2014} and \citet{TianqiGuangHan2014}. Above, $\hat{\Sigma}^{(j)} = {n_k}^{-1} \sum_{i\in\mathcal{I}_{j}}\bX_{i}^{(j)}\bX_{i}^{(j)T}$ is the sample covariance based on $\cD_{j}$, whose population counterpart is $\Sigma=\EE(\bX_{1}\bX_{1}^{T})$ and $M^{(j)}$ is its regurlized inverse. The second term in \eqref{splitwiseDebiased} is a bias correction term, while $\sigma^{2}\mathbf{m}_{v}^{(j)T}\widehat{\Sigma}^{(j)}\mathbf{m}^{(j)}_{v}/n_{k}$ is shown in \citet{JMJMLR2014} to be the variance of the $v^{th}$ component of $\widehat{\bbeta}^{d}(\cD_{j})$.  The parameter $\vartheta_1$, which tends to zero, controls the bias of the debiased estimator \eqref{splitwiseDebiased} and the optimization in \eqref{eq:m-est} minimizes the variance of the resulting estimator.  

Solving $d$ optimization problems in \eqref{eq:m-est} increase an order of magnitude of computation complexity even for $k =1$.  Thus, it is necessary to appeal to the divide and conquer strategy to reduce the computation burden.  This gives rise to the question how large $k$ can be in order to maintain the same statistical properties as the whole sample one ($k=1$).
%, thus the debiased estimator \eqref{splitwiseDebiased} simultaneously keeps bias and variance low.

%\begin{remark}
%\textcolor{red}{Easy to make it tuning free. Reference.}.
%\end{remark}

Because our DC procedure gives rise to smaller samples, we overcome the singularity in $\widehat{\Sigma}$ through a change of variables. More specifically, noting that $M^{(j)}$ is not required explicitly, but rather the product $M^{(j)}(X^{(j)})^{T}$, we propose
\begin{eqnarray} \label{eq:m-estv2}
\nonumber \bb_{v}\sj = \argmin_{\bb} \frac{\bb^{(j)T} \bb^{(j)}}{n_{k}} \;\; &\text{s.t.} & \Bigl\|\frac{X^{(j)T}\bb^{(j)}}{n_{k}} - \be_v\Bigr\|_\infty \leq \vartheta_{1}, \\
\nonumber 			  &               & \|\bb^{(j)}\|_{\infty}\leq \vartheta_{2},			
\end{eqnarray}
from which we construct $M^{(j)}(X^{(j)})^{T}=B^{T}$, where $B=(\bb_{1},\ldots,\bb_{d})$.

The following conditions on the data generating process and the tail behavior of the design vectors are imposed in \citet{JMJMLR2014}. Both conditions are used to derive the theoretical properties of the DC Wald test statistic based on the aggregated debiased estimator, $\overline{\bbeta}^{d}=k^{-1}\sum_{j=1}^{k}\hat{\bbeta}^{d}(\cD_{j})$.

\begin{condition}\label{con:sigma}
$\{(Y_{i},\bX_{i})\}_{i=1}^{n}$ are i.i.d.~and $\Sigma$ satisfies $0<C_{\min}\leq \lambda_{\min}(\Sigma)\leq \lambda_{\max}(\Sigma) \leq C_{\max}$.
\end{condition}

\begin{condition}\label{con:subg}
The rows of $X$ are sub-Gaussian with $\|\bX_i\|_{\psi_{2}}\le \kappa$, $i=1,\ldots,n$.
\end{condition}

Note that under the two conditions above, there exists a constant $\kappa_1>0$ such that $\|\bX_1\Sigma^{-\frac{1}{2}}\|_{\psi_2}\le \kappa_1$. Without loss of generality, we can set $\kappa_1=\kappa$. Our first main theorem provides the relative scaling of the various tuning parameters involved in the construction of $\overline{\bbeta}^{d}$.
\begin{theorem}\label{thmWaldDistribution}
Suppose Conditions \ref{con:estimation}, \ref{con:sigma} and \ref{con:subg} are fulfilled. Suppose $\EE[\epsilon_{1}^{4}]<\infty$ and choose $\vartheta_{1}$, $\vartheta_{2}$ and $k$ such that  $\vartheta_{1} \asymp\sqrt{k\log d/n}$, $\vartheta_{2}n^{-1/2}=o(1)$ and $k=o((s\log d)^{-1}\sqrt{n})$. For any $v\in\{1,\ldots, d\}$,
\begin{equation}\label{eqWaldConvWeakly}
\sqrt{n}\frac{1}{k}\sum_{j=1}^{k} \frac{\widehat{\beta}^{d}_{v}(\cD_{j}) - \beta^{*}_{v}}{\hat{Q}_{v}^{(j)}} \leadsto N(0,\sigma^{2}),
\end{equation}
where $\hat{Q}_{v}^{(j)}=\bigl(\mb_{v}^{(j)T}\hat{\Sigma}\sj \mb_{v}\sj\bigr)^{1/2}$.
\end{theorem}
Theorem \ref{thmWaldDistribution} entertains the prospect of a divide and conquer Wald statistic of the form
\begin{equation}\label{eqWaldStatistic}
\overline{S}_{n}=\sqrt{n}\frac{1}{k}\sum_{j=1}^{k} \frac{\hat{\beta}_{v}^{d}(\cD_{j})-\beta_{v}^{H}}{\overline{\sigma}\bigl(\mb_{v}^{(j)T}\hat{\Sigma}\sj \mb_{v}\sj\bigr)^{1/2}}
\end{equation}
for $\beta_{v}^{*}$, where $\overline{\sigma}$ is an estimator for $\sigma$ based on the $k$ subsamples. On the left hand side of equation \eqref{eqWaldStatistic} we suppress the dependence on $v$  to simplify notation. As an estimator for $\sigma$, a simple suggestion with the same computational complexity is $\overline{\sigma}$ where
\begin{equation}\label{eqSigmaEstimator}
\overline{\sigma}^{2} = \frac{1}{k}\sum_{j=1}^{k} \hat{\sigma}^{2}(\cD_{j}) \quad \text{ and } \quad \hat{\sigma}^{2}(\cD_{j})=\frac{1}{n_{k}}\sum_{i\in\mathcal{I}_{j}}(Y_{i}\sj - \bX_{i}^{(j)T}\hat{\bbeta}^{\lambda}_{\text{LASSO}}(\cD_{j}))^{2}.
\end{equation}
One can use the refitted cross-validation procedure of \citet{FanGuoHao2012} to reduce the bias of the estimate. In Lemma \ref{sigmaEstConsistent} we show that with the scaling of $k$ and $\lambda$ required for the weak convergence results of Theorem \ref{thmWaldDistribution}, consistency of $\overline{\sigma}^{2}$ is also achieved.
\begin{lemma}\label{sigmaEstConsistent}
Suppose $\EE[\epsilon_{i}|\bX_{i}]=0$ for all $i\in\{1,\ldots, n\}$. Then with $\lambda\asymp\sqrt{k \log d/n}$ and $k=o\bigl( \sqrt{n}(s\log d)^{-1}\bigr)$, $|\overline{\sigma}^{2}-\sigma^{2}|=o_{\PP}(1)$.
\end{lemma}
With Lemma \ref{sigmaEstConsistent} and Theorem \ref{thmWaldDistribution} at hand, we establish in Corollary \ref{corollWaldTestNull} the asymptotic distribution of $\overline{S}_{n}$ under the null hypothesis $H_{0}:\beta_{v}^{*}=\beta_{v}^{H}$. This holds for each component $v\in\{1,\ldots,d\}$.

\begin{corollary}\label{corollWaldTestNull}
Suppose Conditions \ref{con:sigma} and \ref{con:subg} are fulfilled, $\EE[\epsilon_{1}^{4}]<\infty$, and $\lambda$, $\vartheta_{1}$ and $\vartheta_{2}$ are chosen as $\lambda  \asymp \sqrt{k\log d/n}$, $\vartheta_{1} \asymp\sqrt{k\log d/n}$ and $\vartheta_{2}n^{-1/2}=o(1)$. Then provided $k=o((s\log d)^{-1}\sqrt{n})$, under $H_{0}:\beta_{v}^{*}=\beta_{v}^{H}$, we have
\begin{equation}
\lim_{n\rightarrow \infty} \sup_{t\in \RR} \bigl|\PP\bigl(\overline{S}_{n}\leq t\bigr) - \Phi(t)\bigr| = 0,
\end{equation}
where $\Phi(\cdot)$ is the cdf of a standard normal distribution.
\end{corollary}		

\subsubsection{Wald Test in the Likelihood Based Framework}\label{sectionWaldvdG}

An alternative route to debiasing the LASSO estimator of $\bbeta^{*}$ is the one proposed in \citet{vdGB2014}. Their so called desparsified estimator of $\bbeta^{*}$ is more general than the debiased estimator of \citet{JMJMLR2014} in that it accommodates generic estimators of the form \eqref{eqMEstimator} as pilot estimators, but the latter optimizes the variance of the resulting estimator. The desparsified estimator for subsample $\cD_j$ is
\begin{align}\label{eq:desparsified estimator}
\hat{\bbeta}^{d}(\cD_{j}) = \hat{\bbeta}^{\lambda}(\cD_{j}) - \hat{\Theta}\sj \nab\ell_{n_{k}}\sj(\hat{\bbeta}^{\lambda}(\cD_{j})),
\end{align}
%
%where $\hat{\Theta}\sj$ will be defined in due course. %It is insightful to consider the case of the Gaussian generalised linear model with identity link and LASSO pilot estimator, for which a geometric interpretation of equation \eqref{eq:desparsified estimator} is readily accessible. In this case equation \eqref{eq:desparsified estimator} reduces to
%
%\begin{equation}\label{eqDesparsifiedSpecialCase}
%\hat{\bbeta}^{d}(\cD_{j}) = \hat{\bbeta}^{\lambda}(\cD_{j}) + \hat{\Theta}\sj X^{(j)T}\bigl(Y\sj - X\sj\hat{\bbeta}^{\lambda}(\cD_{j})\bigr)/n_{k}.
%\end{equation}
%
%As first pointed out in \citet{JMJMLR2014}, the term $\hat{\Theta}\sj X^{(j)T}\bigl(Y\sj - X\sj\hat{\bbeta}^{\lambda}(\cD_{j})\bigr)/n_{k}$ corrects the bias induced by the $\ell_{1}$ penalty in the LASSO construction by adding a term proportional to the subgradient of the $\ell_{1}$ norm at the LASSO solution.
where $\hat{\Theta}\sj$ is a regularized inverse of the Hessian matrix of second order derivatives of $\ell_{n_{k}}\sj(\bbeta)$ at $\hat{\bbeta}^{\lambda}(\cD_{j})$, denoted by $\hat{J}\sj= \nab^{2} \ell_{n_{k}}\sj\bigl(\hat{\bbeta}^{\lambda}(\cD_{j})\bigr)$.
We will make this explicit in due course.  The estimator resembles the classical one-step estimator \citep{Bic75}, but now in the high-dimensional setting via regularized inverse of the Hessian matrix $\hat{J}\sj$,
%The high level motivation for construction \eqref{eq:desparsified estimator} is similar in spirit to that of \citet{JMJMLR2014}.
%The approach involves approximately inverting $\hat{H}\sj$ by regularization,
which reduces to the empirical covariance of the design matrix in the case of the linear model. From equation \eqref{eq:desparsified estimator}, the aggregated debiased estimator over the $k$ subsamples is defined as $\overline{\bbeta}^d=k^{-1}\sum_{j=1}^{k} \widehat{\bbeta}^{d}(\cD_{j})$.

%Since $\hat{J}\sj$ in general depends on $\hat{\bbeta}^{\lambda}(\cD_{j})$,
We now use the nodewise LASSO \citep{MB2006} to approximately invert $\hat{J}\sj$ via $L_1$-regularization.  The basic idea is to find the regularized invert row by row via a penalized $L_1$-regression, which is the same as regressing the variable $X_v$ on $\bX_{-v}$ but expressed in the sample covariance form. For each row $v\in{1,\ldots,d}$, consider the optimization
\begin{equation}\label{eq:nodewise lasso}
\hat{\bkappa}_{v}(\cD_{j}) = \argmin_{\bkappa \in \RR^{d-1}}~\bigl(\hat{J}_{vv}\sj - 2\hat{J}_{v,-v}\sj\bkappa + \bkappa^{T} \hat{J}_{-v,-v}\sj\bkappa +2\lambda_{v} \|\bkappa\|_1\bigr),
\end{equation}
where $\hat{J}_{v,-v}\sj$ denotes the $v^{th}$ row of $\hat{J}\sj$ without the $(v,v)^{\text{th}}$ diagonal element, and $\hat{J}_{-v, -v}\sj$ is the principal submatrix without the $v^{\text{th}}$ row and $v^{\text{th}}$ column. Introduce

\begin{equation}\label{NWMatrix}
 \hat{C}\sj := \left( \begin{array}{cccc}
 1 & -\hat{\kappa}_{1,2}(\cD_{j}) & \ldots & -\hat{\kappa}_{1,d}(\cD_{j}) \\
 -\hat{\kappa}_{2,1}(\cD_{j}) & 1 & \ldots  & -\hat{\kappa}_{2,d}(\cD_{j})\\
 \vdots & \vdots & \ddots & \vdots\\
 -\hat{\kappa}_{d,1} (\cD_{j})& -\hat{\kappa}_{d,2}(\cD_{j}) & \ldots & 1 \end{array} \right)
 \end{equation}
 and $\hat{\Xi}^{(j)} = \diag\big(\hat{\tau}_{1}(\cD_{j}),\ldots, \hat{\tau}_{d}(\cD_{j})\big)$, where $\hat{\tau}_{v}(\cD_{j})^{2} = \hat{J}\sj_{vv}-\hat{J}\sj_{v,-v} \hat{\bkappa}_{v}(\cD_{j})$. $\hat{\Theta}\sj$ in equation \eqref{eq:desparsified estimator} is given by
\begin{equation}\label{eq:estimator of Theta}
\hat{\Theta}\sj = (\hat{\Xi}^{(j)})^{-2}\hat{C}\sj,
\end{equation}
and we define $\hat{\bTheta}\sj_{v}$ as the transposed $v^{th}$ row of $\hat{\Theta}\sj$.

Theorem \ref{thmWaldDistVdG} establishes the limit distribution of the term,
\begin{equation}\label{vdGWaldStat}
\overline{S}_{n} = \sqrt{n} \frac{1}{k}\sum_{j=1}^{k} \frac{\hat{\beta}_{v}^{d}(\cD_{j})-\beta_{v}^{H}}{\sqrt{\Theta^{*}_{vv}}}
\end{equation}
for any $v\in\{1,\ldots,d\}$ under the null hypothesis $H_{0}: \beta_{v}^{*}=\beta_{v}^{H}$.  This provides the basis for the statistical inference based on divide-and-conquer. We need the following condition.
Recall that $J^{*}=\EE\bigl[\nabla_{\bbeta\bbeta}\ell_{n}(\bbeta^{*})\bigr]$ and  consider the generalized linear model \eqref{eq:glm}.

\begin{condition}
	\label{con:glm}
(i) $\{(Y_{i},\bX_{i})\}_{i=1}^{n}$ are i.i.d., $0<C_{\min}\leq\lambda_{\min}(\Sigma)\le \lambda_{\max}(\Sigma) \le C_{\max}$, $\lambda_{\min}(J^{*})\geq L_{\min} >0$, $\|J^{*}\|_{\max}<U_1<\infty$. (ii) For some constant $M<\infty$, $\max_{1\leq i\leq n}\bigl|\bX_{i}^{T}\bbeta^{*}\bigr|\le M$ and $\max_{1\leq i \leq n}\|\bX_{i}\|_{\infty}\leq M$.  (iii) There exist finite constants $U_2, U_3>0$ such that $b''(\eta)<U_2$  and $b'''(\eta)<U_3$ for all $\eta\in\RR$.	
\end{condition}

The same assumptions appear in \citet{vdGB2014}. In the case of the Gaussian GLM, the condition on $\lambda_{\min}(J^{*})$ reduces to the requirement that the covariance of the design has  minimal eigenvalue bounded away from zero, which is a standard assumption. We require $\|J^{*}\|_{\max} <\infty$ to control the estimation error of different functionals of $J^*$. The restriction in (ii) on the covariates and the projection of the covariates are imposed for technical simplicity; it can be extended to the case of exponential tails \citep[see][]{FanSong2010}. Note that $\Var(Y_{i})=\phi b''(\bX_{i}^{T}\bbeta^{*})$ where $\phi$ is the dispersion parameter in \eqref{eq:glm}, so $b''(\eta)<U_2$ essentially implies an upper bound on the variance of the response. In fact, Lemma \ref{lem:mgf} shows that $b''(\eta)<U_2$ can guarantee that the response is sub-gaussian. $b'''(\eta)<U_3$ is used to derive the Lipschitz property of $b''(\bX_i^T\bbeta)$ with respect to $\bbeta$ as shown in Lemma \ref{lem:lipshitz}. We emphasize that no requirement in Condition \ref{con:glm} is specific to the divide and conquer framework.

The assumption of bounded design in (ii) can be relaxed to the sub-gaussian design. However, the price to pay is that the allowable number of subsets $k$ is smaller than the bounded case, which means we need a larger sub-sample size. To be more precise, the order of maximum $k$ for the sub-gaussian design has an extra factor, which is a polynomial of $\sqrt{\log d}$, compared to the order for the bounded design. This logrithmic factor comes from different Lipschitz properties of $b''(\bX_i^T\bbeta)$ in the two designs, which is fully explained in Lemma \ref{lem:lipshitz} of the appendix. In the following theorems, we only present results for the case of bounded design for technical simplicity.

In addition, recalling that $\Theta^{*}=(J^{*})^{-1}$, where $J^{*}:=J(\bbeta^{*})=\EE \bigl[\nab_{\bbeta\bbeta}^{2}\ell_{n}(\bbeta^{*})\bigr]$, we impose Condition \ref{con:theta} on $\Theta^{*}$ and its estimator $\widehat{\Theta}$.

\begin{condition}\label{con:theta}
(i) $\min_{1\leq v \leq d}\Theta_{vv}^{*}>\theta_{\min}>0$. (ii) $\max_{1\le i \le n} \|\bX_i ^T \bTheta^*\|_{\infty} \le M$. (iii) For $v=1,\dots,d$, whenever $\lambda_v\asymp\sqrt{k\log d/n}$ in (\ref{eq:nodewise lasso}), we have
	\[
		\PP\left(\|\hat\bTheta_v-\bTheta^*_v\|_1 \ge Cs_{1}\sqrt{\log d/n}\right) \le d^{-1},
	\]
where $C$ is a constant and $s_{1}$ is such that $\|\bTheta^{*}_{v}\|_{0}\lesssim s_{1}$ for all $v\in\{1,\ldots,d\}$.
\end{condition}

Part (i) of Corollary \ref{con:theta} ensures that the variances of each component of the debiased estimator exist, guaranteeing the existence of the Wald statistic. Parts (ii) and (iii) are imposed directly for technical simplicity. Results of this nature have been established under a similar set of assumptions in \citet{vdGB2014} and \citet{negahban2009unified} for convex penalties and in \citet{WangLiuZhang2014} and \citet{PoLing2015} for folded concave penalties.

As a step towards deriving the limit distribution of the proposed divide and conquer Wald statistic in the GLM framework, we establish the asymptotic behavior of the aggregated debiased estimator $\overline{\beta}_{v}^{d}$ for every given $v \in [d]$.

\begin{theorem}\label{thmWaldDistVdG}
Under Conditions \ref{con:estimation}, \ref{con:glm} and \ref{con:theta},  with $\lambda \asymp \sqrt{k \log d / n}$, we have
\begin{equation}\label{waldDesparsifed}
\overline{\beta}_{v}^{d} - \beta_{v}^{*} = -\frac{1}{k} \sum_{j=1}^{k} \hat{\bTheta}^{\sj T}_{v}\nab\ell_{n_{k}}\sj (\bbeta^{*}) + o_{\PP}(n^{-1/2})
\end{equation}
for any $k \ll d$ satisfying $k=o\bigl(((s\vee s_{1})\log d)^{-1}\sqrt{n}\bigr)$, where $\hat{\bTheta}^{\sj}_{v}$ is the transposed $v^{th}$ row of $\hat{\Theta}^{(j)}$.
\end{theorem}

A corollary of Theorem \ref{thmWaldDistVdG} provides the asymptotic distribution of the Wald statistic in equation \eqref{vdGWaldStat} under the null hypothesis.

\begin{corollary}\label{corollWaldvdGDist}
Let $\overline{S}_{n}$ be as in equation \eqref{vdGWaldStat}, with $\Theta_{vv}^*$ replaced with an estimator $\widetilde{\Theta}_{vv}$. Then under the conditions of Theorem \ref{thmWaldDistVdG} and $H_{0}:\beta_{v}^{*}=\beta_{v}^{H}$, provided $|\widetilde{\Theta}_{vv}-\Theta_{vv}|=o_{\PP}(1)$ under the scaling $k=o\bigl(((s\vee s_{1})\log d)^{-1}\sqrt{n}\bigr)$, we have
\[
\lim_{n\rightarrow \infty} \sup_{t\in \RR} \bigl|\PP\bigl(\overline{S}_{n}\leq t\bigr) - \Phi(t)\bigr| = 0.
\]
\end{corollary}

\begin{remark}
Although Theorem \ref{thmWaldDistVdG} and Corollary \ref{corollWaldvdGDist} are stated only for the GLM, their proofs are in fact an application of two more general results. Further details are available in Lemmas \ref{lemmaWaldDistVdGWeakerConditions} and \ref{lemmaWaldDistVdG2WeakerConditions} of the appendix.
\end{remark}

We return to the issue of estimating $\Theta_{vv}^{*}$ in Section \ref{sectionEstimation}, where we introduce an consistent estimator of $\Theta^{*}_{vv}$ that preserves the scaling of Theorem \ref{thmWaldDistVdG} and Corollary \ref{corollWaldvdGDist}.

%\textcolor{red}{Comparing Corollary \ref{corollWaldvdGDist} to Theorem \ref{thmWaldDistribution}, we observe that the price to pay for increased generality is a $\min\{1,s/s_{1}\}$ factor in the permitted number of sample splits. By contrast with the linear model, the population analogue of the Hessian matrix in the generalized linear model depends on the unknown $\bbeta^{*}$. The same factor arises in the likelihood based score test of \citet{NingLiu2014b}. We emphasize that this factor is not a consequence of the use of the divide and conquer algorithm.}

\subsection{Divide and Conquer Score Test}\label{sectionScore}

In this section, we use $\nab_{v}f(\bbeta)$ and $\nab_{-v}f(\bbeta)$ to denote, respectively, the partial derivative of $f$ with respect to $\beta_{v}$ and the partial derivative vector of $f$ with respect to $\bbeta_{-v}$.  $\nab^{2}_{vv}f(\bbeta)$, $\nab^{2}_{v,-v}f(\bbeta)$, $\nab^{2}_{-v, v}f(\bbeta)$ and $\nab^{2}_{-v,-v}f(\bbeta)$ are analogously defined.

In the low dimensional setting (where $d$ is fixed), Rao's score test of $H_{0}:\beta_{v}^{*}=\beta_{v}^{H}$ against $H_{1}:\beta_{v}^{*}\neq \beta_{v}^{H}$ is based on $\nab_{v}\ell_{n}(\beta_{v}^{H},\widetilde{\bbeta}_{-v})$, where $\widetilde{\bbeta}_{-v}$ is a constrained maximum likelihood estimator of $\bbeta^{*}_{-v}$, constructed as $\widetilde{\bbeta}_{-v}=\argmin_{\bbeta_{-v}}\ell_{n}(\beta_{v}^{H},\bbeta_{-v})=\argmax_{\bbeta_{-v}}\{-\ell_{n}(\beta_{v}^{H},\bbeta_{-v})\}$. If $H_{0}$ is false, imposing the constraint postulated by $H_{0}$ significantly violates the first order conditions from M-estimation with high probability; this is the principle underpinning the classical score test. Under regularity conditions, it can be shown \citep[e.g.][]{CoxHinkley1974} that
\[
\sqrt{n} \bigl(\nab_{v}\ell_{n}(\beta_{v}^{H}, \widetilde{\bbeta}_{-v})\bigr)J_{v|-v}^{*-1/2}\leadsto N(0,1),
\]
where $J_{v|-v}^{*}$ is given by $J^{*}_{v|-v}=J^{*}_{v,v}-\bJ^{*}_{v,-v}J^{*-1}_{-v,-v}\bJ^{*}_{-v,v}$, with $J^{*}_{v,v}$, $\bJ^{*}_{v,-v}$, $J^{*}_{-v, -v}$ and $\bJ^{*}_{-v, v}$ the partitions of the information matrix $J^{*}=J(\bbeta^{*})$,
\begin{equation}\label{infoMatrix}
J(\bbeta) = \left( \begin{array}{cc}
		J_{v,v} & \bJ_{v,-v} \\
		\bJ_{-v,v} & J_{-v,-v}
\end{array}\right)  = \left( \begin{array}{cc}
		\EE\nabla_{v,v}^{2}\ell_{n}(\bbeta) & \EE\nabla_{v,-v}^{2}\ell_{n}(\bbeta) \\
		\EE\nabla_{-v,v}^{2}\ell_{n}(\bbeta) & \EE\nabla_{-v,-v}^{2}\ell_{n}(\bbeta)
\end{array}\right).
\end{equation}

The problems associated with the use of the classical score statistic in the presence of a high dimensional nuisance parameter are brought to light by \citet{NingLiu2014b}, who propose a remedy via the decorrelated score.  The problem stems from the inversion of the matrix $J^{*}_{-v,-v}$ in high dimensions.  The decorrelated score is defined as
\begin{equation}\label{deCorrScore}
S(\beta_{v}^{*}, \bbeta_{-v}^{*}) = \nabla_{v} \ell_{n}(\beta_{v}^{*}, \bbeta_{-v}^{*}) - \bw^{*T} \nabla_{-v} \ell_{n}(\beta_{v}^{*}, \bbeta_{-v}^{*}), \text{ where } \bw^{*T} = \bJ_{v,-v}^{*} J^{*-1}_{-v,-v}.
\end{equation}
For a regularized estimator $\hat{\bw}$ of $\bw^{*}$, to be defined below, a mean value expansion of
\begin{equation}\label{eqDeCorrScoreStat}
\hat{S}(\beta_{v}^{*}, \hat{\bbeta}_{-v}^{\lambda}):=\nabla_{v} \ell_{n}(\beta_{v}^{*}, \hat{\bbeta}_{-v}^{\lambda}) - \hat{\bw}^T \nabla_{-v} \ell_{n}(\beta_{v}^{*}, \hat{\bbeta}_{-v}^{\lambda})
\end{equation}
around $\bbeta_{-v}^{*}$ gives
\begin{eqnarray}\label{eqKKey}
\nonumber \hat{S}(\beta_{v}^{*}, \hat{\bbeta}_{-v}^{\lambda})&=&\nabla_{v} \ell_{n}(\beta_{v}^{*}, \bbeta_{-v}^{*}) - \hat{\bw}^T \nabla_{-v} \ell_{n}(\beta_{v}^{*}, \bbeta_{-v}^{*}) \\
& & \quad \quad + \;\; \left[ \nabla^{2}_{v,-v}\ell_{n}(\beta_{v}^{*},\bbeta_{-v,\alpha}) - \hat{\bw}^{T}\nabla^{2}_{-v,-v}\ell_{n}(\beta_{v}^{*},\bbeta_{-v,\alpha})\right] (\hat{\bbeta}_{-v}^{\lambda} - \bbeta_{-v}^{*}),
\end{eqnarray}
where $\bbeta_{-v,\alpha}=\alpha\hat{\bbeta}_{-v}^{\lambda}+(1-\alpha)\bbeta_{-v}^{*}$ for $\alpha\in[0,1]$. The key to understanding how the decorrelated score remedies the problems faced by the classical score test is the observation that
\begin{eqnarray}\label{eqKey}
 & & \left[ \nabla^{2}_{v,-v}\ell_{n}(\beta_{v}^{*},\bbeta_{-v,\alpha}) - \hat{\bw}^{T}\nabla^{2}_{-v,-v}\ell_{n}(\beta_{v}^{*},\bbeta_{-v,\alpha})\right]\\
 \nonumber &\approx & \EE\left[\nabla^{2}_{v,-v}\ell_{n}(\beta_{v}^{*},\bbeta_{-v}^{*}) - \bw^{*T}\nabla^{2}_{-v,-v}\ell_{n}(\beta_{v}^{*},\bbeta_{-v}^{*})\right] = \bJ^{*}_{v,-v}-\bJ^{*}_{v,-v}J_{-v,-v}^{*-1}J_{-v,-v}^{*} =0,
\end{eqnarray}
where $\bw^{*T} = \bJ^{*}_{v,-v} J^{*-1}_{-v,-v}$. Hence, provided $\bw^{*}$ is sufficiently sparse to avoid excessive noise accumulation, we are able to achieve rate acceleration in equation \eqref{eqKKey}, ultimately giving rise to a tractable limit distribution of a suitable rescaling of $\hat{S}(\beta_{v}^{*}, \hat{\bbeta}_{-v}^{\lambda})$.
%The proof strategy involves a demonstration that the left hand side of equation \eqref{eqKey} converges to the expression on the right hand side at a sufficiently fast rate to annihilate $(\hat{\bgamma}^{\lambda}-\bgamma^{*})$ in the expansion of the decorrelated score. $(\hat{\bgamma}^{\lambda}-\bgamma^{*})$ possesses non-standard limit distribution when rescaled by $\sqrt{n}$, leading to a non-standard limit distribution of the classical score statistic.
Since $\beta_{v}^{*}$ is restricted under the null hypothesis, $H_{0}:\beta_{v}^{*}=\beta_{v}^{H}$, the statistic in equation \eqref{eqDeCorrScoreStat} is accessible once $H_{0}$ is imposed. As \citet{NingLiu2014b} point out, %Additionally note that, by the second Bartlett identity \citep[e.g.][p120]{Ferguson1996},
%
%\begin{eqnarray*}
%\bw^{*} &=& (J^{*}_{\theta \bgamma} J^{*-1}_{\bgamma \bgamma})^{T}=J^{*-1}_{\bgamma \bgamma}\bJ^{*}_{\bgamma\theta}\\
%             &=& \EE\bigl[ \nab_{\bgamma}\ell_{n}(\theta^{*},\bgamma^{*})\nab_{\bgamma}\ell_{n}(\theta^{*},\bgamma^{*})^{T}\bigr]^{-1}\EE\bigl[\nab_{\bgamma}\ell_{n}(\theta^{*},\bgamma^{*}) \nabla_{v}\ell_{n}(\theta^{*},\bgamma^{*})\bigr] \\
%			&=& \EE\bigl[ \nab_{\bgamma}\ell_{n}(\theta^{H},\bgamma^{*})\nab_{\bgamma}\ell_{n}(\theta^{H},\bgamma^{*})^{T}\bigr]^{-1}\EE\bigl[\nab_{\bgamma}\ell_{n}(\theta^{H},\bgamma^{*}) \nabla_{v}\ell_{n}(\theta^{H},\bgamma^{*})\bigr]
%\end{eqnarray*}
%
$\bw^{*}$ is the solution to
\[
\bw^{*}=\argmin_{\bw} \EE \bigl[\nabla_{v}\ell_{n}(\beta_{v}^{H},\bbeta_{-v}^{*}) - \bw^{T}\nab_{-v}\ell_{n}(\beta_{v}^{H},\bbeta_{-v}^{*})\bigr]^{2}
\]
under $H_{0}:\beta_{v}^{*}=\beta_{v}^{H}$. We thus see that the population analogue of the decorrelation device is the linear combination $\bw^{*T}\nab_{-v}\ell_{n}(\beta_{v}^{H},\bbeta_{-v}^{*})$ that best approximates $\nabla_{v}\ell_{n}(\beta_{v}^{H},\bbeta_{-v}^{*})$ in a least squares sense.

Our divide and conquer score statistic under $H_{0}:\beta_{v}^{*}=\beta_{v}^{H}$ is
\begin{equation}\label{eqDCAggScore}
\overline{S}(\beta_{v}^{H})=\frac{1}{k}\sum_{j=1}^{k} \hat{S}\sj\bigl(\beta_{v}^{H},\hat{\bbeta}_{-v}^{\lambda}(\cD_{j})\bigr),
\end{equation}
where $\displaystyle{
\hat{S}\sj\bigl(\beta_{v},\hat{\bbeta}_{-v}^{\lambda}(\cD_{j})\bigr)=\nabla_{v} \ell_{n_{k}}\sj\bigl(\beta_{v}, \hat{\bbeta}_{-v}^{\lambda}(\cD_{j})\bigr) - \hat{\bw}(\cD_{j})^T \nabla_{-v} \ell_{n_{k}}\sj\bigl(\beta_{v}, \hat{\bbeta}_{-v}^{\lambda}(\cD_{j})\bigr)}$ and
\begin{equation}\label{eqDantzig}
				 \hat{\bw}(\cD_{j}) = \argmin_{\bw} \|\bw\|_1, \text{ s.t. }  \Bigl\|\nabla^{2}_{-v,v} \ell_{n_{k}}\sj\bigl(\hat{\beta}_{v}^{\lambda}(\cD_{j}), \hat{\bbeta}_{-v}^{\lambda}(\cD_{j})\bigr) - \bw^T \nabla^{2}_{-v,-v} \ell_{n_{k}}\sj\bigl(\hat{\beta}_{v}^{\lambda}(\cD_{j}), \hat{\bbeta}_{-v}^{\lambda}(\cD_{j})\bigr)\Bigr\|_{\infty} \le \mu.
\end{equation}
Equation \eqref{eqDantzig} is the Dantzig selector of \citet{CandesTao2007}.

\begin{theorem}\label{thmScoreDistn}
Let $\widehat{J}_{v|-v}$ be a consistent estimator of $J^{*}_{v|-v}$ and
\[
S\sj(\beta_{v}^{H},\bbeta_{-v}^*)=\nabla_{v} \ell_{n_{k}}\sj(\beta_{v}^{H}, \bbeta_{-v}^*) - \bw^{*T} \nabla_{-v} \ell_{n_{k}}\sj(\beta_{v}^{H}, \bbeta_{-v}^*).
\]
Suppose $\|\bw^{*}\|_{1}\lesssim s_{1}$ and Conditions \ref{con:estimation} and \ref{con:glm} are fulfilled. Then under $H_{0}:\beta_{v}^{*}=\beta_{v}^{H}$ with $\lambda\asymp\mu\asymp\sqrt{k\log d/n}$,
\[
\sqrt{n}\;\overline{S}(\beta_{v}^{H})= \sqrt{n}\frac{1}{k}\sum_{j=1}^{k} S\sj(\beta_{v}^{H},\bbeta_{-v}^*) + o_{\PP}(1) \;\; \text{ and }\;\;
\lim_{n\rightarrow \infty} \sup_{t\in \RR} \bigl|\PP\bigl(\overline{S}(\beta_{v}^{H})\widehat{J}_{v|-v}^{-1/2}\leq t\bigr) - \Phi(t)\bigr| = 0,
\]
for any $k \ll d$ satisfying $k=o\bigl(((s\vee s_{1})\log d)^{-1}\sqrt{n}\bigr)$, where $\overline{S}(\beta_{v}^{H})$ is defined in equation \eqref{eqDCAggScore}.
\end{theorem}

\begin{remark}
By the definition of $\bw^{*}$ and the block matrix inversion formula for $\Theta^{*}=(J^{*})^{-1}$, sparsity of $\bw^{*}$ is implied by sparsity of $\Theta^{*}$ as assumed in \citet{vdGB2014} and Condition \ref{con:theta} of Section \ref{sectionWaldvdG}. In turn, $\|\bw^{*}\|_{0}\lesssim s_{1}$ implies $\|\bw^{*}\|_{1}\lesssim s_{1}$ provided that the elements of $\bw^{*}$ are bounded.
\end{remark}

\begin{remark}
Although Theorem \ref{thmScoreDistn} is stated in the penalized GLM setting, the result holds more generally; further details are available in Lemma \ref{lemmaDistributionAggScore} of Appendix \ref{appendixAuxLemmata} in the Supplementary Material.
\end{remark}

To maintain the same computational complexity, an estimator of the conditional information needs to be constructed using a DC procedure. For this, we propose to use
\[
\overline{J}_{v|-v}=k^{-1}\sum_{j=1}^{k}\bigl(\nab^{2}_{v,v}\ell_{n_{k}}\sj(\overline{\beta}_{v}^{d},\overline{\bbeta}_{-v}) - \overline{\bw}^{T}\nab^{2}_{-v,v}\ell_{n_{k}}\sj(\overline{\beta}_{v}^{d},\overline{\bbeta}_{-v})\bigr),
\]
where $\overline{\beta}_{v}^{d}=k^{-1}\sum_{j=1}^{k} \widehat{\beta}_{v}^{d}(\cD_{j})$, $\overline{\bbeta}_{-v}=k^{-1}\sum_{j=1}^{k} \widehat{\bbeta}_{-v}^{\lambda}(\cD_{j})$ and $\overline{\bw}=k^{-1}\sum_{j=1}^{k} \widehat{\bw}(\cD_{j})$. By Lemma \ref{lemmaCI}, this estimator is asymptotically consistent.

\begin{lemma}\label{lemmaCI}
Suppose $\|\bw^{*}\|_{1}=O(s_{1})$ and Conditions \ref{con:estimation} and \ref{con:glm} are fulfilled. Then for any $k\ll d$ satisfying $k=o\bigl(((s\vee s_{1})\log d)^{-1}\sqrt{n}\bigr)$, $|\overline{J}_{v|-v}-J^{*}_{v|-v}|=o_{\mathbb{P}}(1)$.
\end{lemma}

%As in Section \ref{WaldLinearModel}, based on Theorem \ref{thmScoreDistn}, a divide and conquer score test of size $\alpha$ for $H_{0}:\theta^{*}=\theta^{H}$ against the alternative $H_{1}:\theta^{*}\neq \theta^{H}$ is given by equation \eqref{eqWaldTest} with $\overline{S}_{n}$ replaced by $\overline{S}_{n}=\sqrt{n}\;\overline{S}(\theta^{H})J_{\theta|\bgamma}^{*-1/2}$.

\section{Accuracy of Distributed Estimation}\label{sectionEstimation}

%As shown in Theorem \ref{thmWaldDistVdG} and in the proof of Theorem \ref{thmWaldDistribution}, the averaged debiased estimator satisfies the decomposition $\sqrt{n}(\overline{\bbeta}^{d} - \bbeta^{*}) = \bZ + o_{\PP}(1)$, where $\bZ$ is responsible for the asymptotic normality. Perhaps more importantly is the observation that $\bZ$ coincides with the leading term in an analogous decomposition of the corresponding full sample estimator $\widehat{\bbeta}^{d}$.
As explained in Section~\ref{sec2.2}, the information got lost in the divide-and-conquer process is not very much.
This motivates us to consider $\|\overline{\bbeta}^{d}-\widehat{\bbeta}^{d}\|_{2}$, the loss incurred by the divide and conquer strategy in comparison with the computationally infeasible full sample debiased estimator $\widehat{\bbeta}^{d}$. Indeed, it turns out that, for $k$ not too large, $\overline{\bbeta}^{d}-\widehat{\bbeta}^{d}$ appears only as a higher order term in the decomposition of $\overline{\bbeta}^{d}-\bbeta^{*}$ and thus $\|\overline{\bbeta}^{d}-\widehat{\bbeta}^{d}\|_{2}$ is negligible compared to the statistical error, $\|\widehat{\bbeta}^{d}-\bbeta^{*}\|_{2}$.  In other words, the divide-and-conquer errors are statistically negligible.

When the minimum signal strength is sufficiently strong, thresholding $\overline{\bbeta}^d$ achieves exact support recovery, motivating a refitting procedure based on the low dimensional selected variables. As a means to understanding the theoretical properties of this refitting procedure, as well as for independent interest, this section develops new theory and methodology for the low dimensional ($d<n$) linear and generalized linear models in addition to their high dimensional ($d\gg n$) counterparts. It turns out that simple averaging of low dimensional OLS or GLM estimators (denoted uniformly as $\hat\bbeta\sj$, without superscript $d$ as debias is not necessary) suffices to preserve the statistical error, i.e., achieving the same statistical accuracy as the estimator based on the whole data set. This is because, in contrast to the high dimensional setting, parameters are not penalized in the low dimensional case. With  $\overline{\bbeta}$ the average of $\hat\bbeta\sj$ over the $k$ machines and $\widehat{\bbeta}$ the full sample counterpart ($k=1$), we derive the rate of convergence of $\|\overline{\bbeta}-\widehat{\bbeta}\|_{2}$. Refitted estimation using only the selected covariates allows us to eliminate a $\log d$ term in the statistical rate of convergence of the estimator. We present the high dimensional and low dimensional results separately, with the analysis of the refitting procedures appearing as corollaries to the low dimensional analysis.

\subsection{The High-Dimensional Linear Model}\label{sectionEstimationHD}

Recall that the high dimensional DC estimator is $\overline{\bbeta}^{d}= k^{-1}\sum_{j=1}^{k}\hat{\bbeta}^{d}(\cD_{j})$, where $\hat{\bbeta}^{d}(\cD_{j})$ for $1\le j \le k$ is the debiased estimator defined in \eqref{splitwiseDebiased}. We also denote the debiased LASSO estimator using the entire dataset as $\hat\bbeta^d = \hat{\bbeta}^{d}(\cup_{j=1}^k\cD_{j})$. The following lemma shows that not only is $\overline{\bbeta}^{d}$ asymptotically normal, it approximates the full sample estimator $\hat\bbeta^d$ so well that it has the same statistical error as $\hat\bbeta^d$ provided the number of subsamples $k$ is not too large.

\begin{lemma}\label{lm:est-inf-rate}
  Under the Conditions \ref{con:sigma} and \ref{con:subg}, if $\lambda$, $\vartheta_{1}$ and $\vartheta_{2}$ are chosen as
  $\lambda  \asymp \sqrt{k\log d/n}$, $\vartheta_{1} \asymp \sqrt{k\log d/n}$ and $\vartheta_{2}n^{-1/2}=o(1)$, we have with probability $1-c/d$,
\begin{equation}\label{eq:est-inf-rate}
 \big\|\overline{\bbeta}^{d} - \hat\bbeta^d\big\|_{\infty} \le C\frac{sk\log d}{n} \text{ and } \big\|\overline{\bbeta}^{d} - \bbeta^*\big\|_{\infty} \le C \Big(\sqrt{\frac{\log d}{n}} + \frac{sk\log d}{n}\Big).
\end{equation}
\end{lemma}

\begin{remark}  The term $\sqrt{\frac{\log d}{n}}$ in \eqref{lm:est-inf-rate} is the estimation error of
$\big\|\widehat{\bbeta}^{d} - \bbeta^*\big\|_{\infty} $.
Lemma \ref{lm:est-inf-rate} does not rely on any specific choice of $k$, however, in order for the aggregated estimator $\overline{\bbeta}^{d}$ to attain the same $\|\cdot\|_{\infty}$ norm estimation error as the full sample LASSO estimator, $\hat{\bbeta}_{\rm LASSO}$, the required scaling is $k = O(\sqrt{n/(s^2 \log d)})$. This is a weaker scaling requirement than that of Theorem \ref{thmWaldDistribution} because the latter entails a guarantee of asymptotic normality, which is a stronger result. It is for the same reason that our estimation results only require $O(\cdot)$ scaling whilst those for testing require $o(\cdot)$ scaling.
\end{remark}

Although $\overline{\bbeta}^{d}$ achieves the same rate as the LASSO estimator under the infinity norm, it cannot achieve the minimax rate in $\ell_{2}$ norm since it is not a sparse estimator. To obtain an estimator with the $\ell_{2}$ minimax rate, we sparsify $\overline{\bbeta}^{d}$ by hard thresholding. For any $\bbeta \in \RR^d$, define the hard thresholding operator $\cT_{\nu}$ such that the $j$-th entry of $\cT_{\nu}(\bbeta)$ is
\begin{equation}\label{eq:hardth}
 [\cT_{\nu}(\bbeta)]_j = \bbeta_j \ind\{|\bbeta_j| \geq \nu\}, \text{ for } 1\le j \le d.
\end{equation}
According to \eqref{eq:est-inf-rate}, if $\bbeta_j^* = 0$, we have $|\overline{\bbeta}^{d}_j| \le C(\sqrt{{\log d}/{n}} + {sk\log d}/{n})$ with high probability. The following theorem characterizes the estimation rate of the thresholded estimator $\cT_{\nu}(\overline{\bbeta}^{d})$.

\begin{theorem}\label{thm:est-l2-rate}
Suppose Conditions \ref{con:sigma} and \ref{con:subg} are fulfilled and choose $\lambda  \asymp \sqrt{k\log d/n}$, $\vartheta_{1} \asymp \sqrt{k\log d/n}$ and  $\vartheta_{2}n^{-1/2}=o(1)$. Take the parameter of the hard threshold operator in \eqref{eq:hardth} as $\nu = C_0\sqrt{{\log d}/{n}}$ for some sufficiently large constant $C_0$. If the number of subsamples satisfies $k = O(\sqrt{n/(s^2 \log d)})$, for large enough $d$ and $n$, we have with probability $1-c/d$,
\begin{equation}\label{eq:therd-beta}
  \big\|\cT_{\nu}(\overline{\bbeta}^{d}) - \cT_{\nu}(\hat{\bbeta}^{d})\big\|_{2} \le C \frac{s^{3/2} k \log d}{n},\;
         \big\|\cT_{\nu}(\overline{\bbeta}^{d}) - \bbeta^*\big\|_{\infty} \le C \sqrt{\frac{\log d}{n}} \text{ and }  \big\|\cT_{\nu}(\overline{\bbeta}^{d}) - \bbeta^*\big\|_{2} \le C \sqrt{\frac{s\log d}{n}}.
\end{equation}
\end{theorem}

\begin{remark}
   In fact, in the proof of Theorem \ref{thm:est-l2-rate}, we show that if the thresholding parameter $\nu$ satisfies $\nu \ge \|\overline{\bbeta}^{d} - \bbeta^*\|_{\infty}$, we have $ \big\|\cT_{\nu}(\overline{\bbeta}^{d}) - \bbeta^*\big\|_{2} \le 2 \sqrt{2s} \cdot \nu$; it is for this reason that we choose $\nu \asymp \sqrt{{\log d}/{n}}$. Unfortunately, the constant is difficult to choose in practice. In the following paragraphs we propose a practical method to select the tuning parameter $\nu$.
\end{remark}

Let $(M^{(j)}X^{(j)T})_{\ell}$ denote the transposed $\ell^{th}$ row of $M^{(j)}X^{(j)T}$. Inspection of the proof of Theorem \ref{thmWaldDistribution} reveals that the leading term of term of  $\sqrt{n}\|\overline{\bbeta}^{d} - \bbeta^*\|_{\infty}$ satisfies  
\[
    T_0 = \max_{1\le \ell \le d} \frac{1}{\sqrt{k}}\sum_{j=1}^{k} \frac{1}{\sqrt{n_{k}}}(M^{(j)}X^{(j)T})_{\ell}^{T} \bepsilon^{(j)}.
\]
\cite{CCK2013} propose the Gaussian multiplier bootstrap to estimate the quantile of $T_0$. Let $\{\xi_i\}_{i=1}^n$ be i.i.d. standard normal random variable independent of $\{(Y_{i},\bX_{i})\}_{i=1}^{n}$. Consider the statistic 
\[
  W_0 = \max_{1\le \ell \le d} \frac{1}{\sqrt{k}}\sum_{j=1}^{k} \frac{1}{\sqrt{n_{k}}}(M^{(j)}X^{(j)T})_{\ell}^{T} (\hat\bepsilon^{(j)}\circ {\bxi}^{(j)}),
\]
where $\hat{\bepsilon}^{(j)} \in \RR^{n_k}$ is an estimator of ${\bepsilon}^{(j)}$ such that for any $i \in \mathcal{I}_{j}$, $\hat{\epsilon}_i^{(j)} = Y_{i}^{(j)} - \bX_{i}^{(j)} \hat{\bbeta}(\cD_j)$, and $\bxi^{(j)}$ is a subvector of $\{\xi_i\}_{i=1}^n$ with indices in $ \mathcal{I}_{j}$. Recall that ``$\circ$" denotes the Hadamard product. The $\alpha$-quantile of $W_0$ conditioning on $\{Y_i, \bX_i\}_{i=1}^n$ is defined as $c_{W_0}(\alpha) = \inf\{t\,|\, \PP\bigl(W_0 \le t \mid \bY,X\bigr) \ge \alpha\}$. We can estimate $c_{W_0}(\alpha)$ by Monte-Carlo and thus choose $\nu_0 = c_{W_0}(\alpha)/\sqrt{n}$. This choice ensures
\[
 \big\|\cT_{\nu_0}(\overline{\bbeta}^{d}) - \bbeta^*\big\|_{2} = O_{\PP}(\sqrt{s\log d/n}),
 \]
which coincides with the $\ell_{2}$ convergence rate of the LASSO.

\begin{remark}
Lemma \ref{lm:est-inf-rate} and Theorem \ref{thm:est-l2-rate} show that if the number of subsamples satisfies $k = o(\sqrt{n/(s^2 \log d)})$, $\|\overline{\bbeta}^{d} - \hat\bbeta^d \|_{\infty} = o_{\PP} \big( \sqrt{\log d/n} \big)$ and $\|\cT_{\nu}(\overline{\bbeta}^{d}) - \cT_{\nu}(\hat{\bbeta}^{d})\|_{2}  = o_{\PP}(\sqrt{s\log d/n} )$, and thus the error incurred by the divide and conquer procedure is negligible compared to the statistical minimax rate. The reason for this contraction phenomenon is that $\overline{\bbeta}^{d}$ and $\hat{\bbeta}^{d}$ share the same leading term in their Taylor expansions around $\bbeta^*$. The difference between them is only the difference of two remainder terms which is smaller order than the leading term. We uncover a similar phenomenon in the low dimensional case covered in Section \ref{sectionEstimationLD1}.  However, in the low dimensional case $\ell_{2}$ norm consistency is automatic while the high dimensional case requires an additional thresholding step to guarantee sparsity and, consequently, $\ell_2$ norm consistency.
 \end{remark}

 \subsection{The High-Dimensional Generalized Linear Model}\label{sectionEstimationHDGLM}

We can generalize the DC estimation of the linear model to GLM. Recall that  $\hat \bbeta^d(\cD_j)$ is the de-biased estimator defined in \eqref{eq:desparsified estimator} and the aggregated estimator is $\overline{\bbeta}^d= k^{-1}\sum_{j=1}^{k} \widehat{\bbeta}^{d}(\cD_{j})$. We still denote $\widehat{\bbeta}^{d} = \widehat{\bbeta}^{d}(\cup_{j=1}^k \cD_j)$. The next lemma bounds the error incurred by splitting the sample and the statistical rate of convergence of $\overline{\bbeta}^{d}$ in terms of the infinity norm.

\begin{lemma}\label{lm:est-inf-rate-glm}
 Under Conditions \ref{con:estimation}, \ref{con:glm} and \ref{con:theta}, for $\widehat{\bbeta}^{\lambda}$ with $\lambda \asymp \sqrt{k \log d / n}$, we have with probability $1-c/d$,
\begin{equation}\label{eq:est-inf-rate-glm}
 \big\|\overline{\bbeta}^{d} - \hat\bbeta^d\big\|_{\infty} \le C\frac{(s\vee s_{1})k\log d}{n} \text{ and } \big\|\overline{\bbeta}^{d} - \bbeta^*\big\|_{\infty} \le C \Big( \sqrt{\frac{\log d}{n}} + \frac{(s\vee s_{1})k\log d}{n}\Big).
\end{equation}
\end{lemma}

Applying a similar thresholding step as in the linear model, we obtain the following estimation rate in $\ell_2$ norm.

\begin{theorem}\label{thm:est-l2-rate-glm}
Under Conditions \ref{con:estimation} - \ref{con:theta}, choose $\lambda \asymp \sqrt{k \log d / n}$ and $\lambda_v \asymp \sqrt{k\log d/n}$. Take the parameter of the hard threshold operator in \eqref{eq:hardth} as $\nu = C_0 \sqrt{{\log d}/{n}}$ for some sufficiently large constant $C_0$. If the number of subsamples satisfies $k = O(\sqrt{n/((s\vee s_{1})^2 \log d)})$, for large enough $d$ and $n$, we have with probability $1-c/d$,
\begin{equation}\label{eq:therd-beta-glm}
  \big\|\cT_{\nu}(\overline{\bbeta}^{d}) - \cT_{\nu}(\hat{\bbeta}^{d})\big\|_{2} \le C \frac{(s\vee s_{1})s^{1/2} k \log d}{n},\quad
         \big\|\cT_{\nu}(\overline{\bbeta}^{d}) - \bbeta^*\big\|_{\infty} \le C  \sqrt{\frac{\log d}{n}}
\end{equation}
$\text{ and }  \big\|\cT_{\nu}(\overline{\bbeta}^{d}) - \bbeta^*\big\|_{2} \le C \sqrt{ {s\log d}/{n}}$.
\end{theorem}

\begin{remark}
As in the case of the linear model, Theorem \ref{thm:est-l2-rate-glm} reveals that the loss incurred by the divide and conquer procedure is negligible compared to the statistical minimax estimation error provided $k=o\bigl(\sqrt{n/(s_{1}\vee s)^{2}s\log d}\bigr)$.
\end{remark}

A similar proof strategy to that of Theorem \ref{thm:est-l2-rate-glm} allows us to construct an estimator of $\Theta^{*}_{vv}$ that achieves the required consistency with the scaling of Corollary \ref{corollWaldvdGDist}. Our estimator is $\widetilde{\Theta}_{vv}:=\bigl[\cT_{\zeta}(\overline{\Theta})\bigr]_{vv}$, where $\overline{\Theta} = k^{-1}\sum_{j=1}^{k}\widehat{\Theta}^{(j)}$ and $\cT_{\zeta}(\cdot)$ is the thresholding operator defined in equation \eqref{eq:hardth} with $\zeta = C_1 \sqrt{{\log d}/{n}}$ for some sufficiently large constant $C_1$.

	\begin{corollary}\label{corollVarianceEstimator}
Under the conditions and scaling of Theorem \ref{thmWaldDistVdG}, $\bigl|\widetilde{\Theta}_{vv} - \Theta_{vv}^{*}\bigr|=o_{\PP}(1)$.
\end{corollary}
Substituting this estimator in Corollary \ref{corollWaldvdGDist} delivers a practically implementable test statistic based on $k=o\bigl(((s\vee s_{1})\log d)^{-1}\sqrt{n}\bigr)$ subsamples.

\subsection{The Low-Dimensional Linear Model}\label{sectionEstimationLD1}

As mentioned earlier, the infinity norm bound derived in Lemma \ref{lm:est-inf-rate} can be used to do model selection, after which the selected support can be shared across all the local agents, significantly reducing the dimension of the problem as we only need to refit the data on the selected model. The remaining challenge is to implement the divide and conquer strategy in the low dimensional setting, which is also of independent interest. Here we focus on the linear model, while the generalized linear model is covered in Section \ref{sectionEstimationLD2}.

In this section $d$ still stands for dimension, but in contrast with the rest of this paper in which $d\gg n$, here we consider $d<n$. More specifically, we consider the linear model (\ref{eq:lm}) with $d<n$ and i.i.d sub-gaussian noise $\{\epsilon_i\}_{i=1}^n$. It is well known that the ordinary least square (OLS) estimator of $\bbeta^*$ is defined as $\hat\bbeta=({X}^T{X})^{-1}{X}^T{\bY}$. In the massive data setting, the communication cost of estimating and inverting covariance matrices is very high (order $O(kd^2)$). However, as pointed out by \cite{Xie2012}, this estimator exactly coincides with the DC estimator,
	\[
		\hat\bbeta=\left(\sum\limits_{j=1}^k X^{(j)T} X^{(j)}\right)^{-1}\sum\limits_{j=1}^k{X^{(j)T}}\bY\sj.
	\]
In this section, we study the DC strategy to approximate $\hat\bbeta$ with the communication cost only $O(kd)$, which implies that we can only communicate $d$ dimensional vectors.
	
The OLS estimator based on the subsample $\cD_j$ is defined as $\hat\bbeta(\cD_j)=(X^{(j)T}X^{(j)})^{-1}X^{(j)T}{\bY}^{(j)}$. In order to estimate $\bbeta^*$, a simple and natural idea is to take the average of $\{\hat\bbeta(\cD_j)\}_{j=1}^k$, which we denote by $\overline\bbeta$. The question is whether this estimator preserves the statistical error as $\hat\bbeta$. The following theorem gives an upper bound of the gap between $\overline\bbeta$ and $\hat\bbeta$, and shows that this gap is negligible compared with the statistical error of $\hat\bbeta$ as long as $k$ is not large.
	\begin{theorem}
		\label{thm:ols2}
		Consider the linear model (\ref{eq:lm}). Suppose Conditions \ref{con:sigma} and \ref{con:subg} hold and $\{\epsilon_i\}_{i=1}^n$ are i.i.d sub-Gaussian random variables with $\|\epsilon_i\|_{\psi_2}\le\sigma_1$. If the number of subsamples satisfies $k=O(nd/(d \vee \log n)^2)$, then for sufficiently large $n$ and $d$ it follows that
		\begin{equation}
			\|\overline\bbeta-\hat\bbeta\|_2=O_{\PP}\Bigl(\frac{\sqrt{k}(d \vee \log n)}{n}\Bigr), \quad\|\overline\bbeta-\bbeta^*\|_2=O_{\PP}\bigl(\sqrt{d/n}\bigr).
		\end{equation}
	\end{theorem}

	\begin{remark}
		\label{rem:ols2}
		By taking $k=o\bigl(nd/(d \vee \log n)^2\bigr)$, the loss incurred by the divide and conquer procedure, i.e., $\|\overline\bbeta-\hat\bbeta\|_2$, converges at a faster rate than the statistical error of the full sample estimator $\hat\bbeta$.
	\end{remark}
	
We now take a different viewpoint by returning to the high dimensional setting of Section \ref{sectionEstimationHD} ($d\gg n$) and applying Theorem \ref{thm:ols2} in the context of a refitting estimator. In this refitting setting, the sparsity $s$ of Lemma \ref{lm:est-inf-rate} becomes the dimension of a low dimensional parameter estimation problem on the selected support. Our refitting estimator is defined as
	\begin{equation}
		\overline\bbeta^r:=\frac{1}{k}\sum\limits_{j=1}^k (X_{\hat S}^{(j)T}X_{\hat S}^{(j)})^{-1}X_{\hat S}^{(j)T}\bY^{(j)},
	\end{equation}
	where $\hat S:=\{j: |\overline\beta^d_j|>2C\sqrt{\log d/n}\}$ and $C$ is the same constant as in (\ref{eq:est-inf-rate}).

	\begin{corollary}
		\label{cor:refitlm}
		Suppose $\beta^*_{\min}>2C\sqrt{\log d/n}$, where $\beta^*_{\min}:=\min_{1\leq j \leq d} |\beta^*_j|$ and $C$ is the same constant as in (\ref{eq:est-inf-rate}). Define the full sample oracle estimator as $\hat\bbeta^o=(X_{S}^{T}X_{S})^{-1}X_{S}^T\bY$, where $S$ is the true support of $\bbeta^*$.  If $k=O(\sqrt{n/(s^2\log d)})$, then for sufficiently large $n$ and $d$ we have
		\begin{equation}
			\|\overline\bbeta^r-\hat\bbeta^o\|_2=O_{\PP}\Bigl(\frac{\sqrt{k}(s \vee \log n)}{n}\Bigr), \quad \|\overline\bbeta^r-\bbeta^*\|_2=O_{\PP}\bigl(\sqrt{s/n}\bigr).
		\end{equation}
	\end{corollary}
We see from Corollary \ref{cor:refitlm} that $\overline\bbeta^r$ achieves the oracle rate when the minimum signal strength is not too weak and the number of subsamples $k$ is not too large.
	
\subsection{The Low-Dimensional Generalized Linear Model}\label{sectionEstimationLD2}
	
 The next theorem quantifies the gap between $\overline\bbeta$ and $\hat\bbeta$, where $\overline{\bbeta}$ is the average of subsambled GLM estimators and $\hat{\beta}$ is the full sample GLM estimator.
				
\begin{theorem}
		\label{thm:glm2}
		Under Condition \ref{con:glm}, if $k=O(\sqrt{n}/(d \vee \log n))$, then we have for sufficiently large $d$ and $n$,
		\begin{equation}
			\|\overline\bbeta-\hat\bbeta\|_2=O_{\PP}\Bigl(\frac{k\sqrt d (d \vee \log n)}{n}\Bigr), \quad  \|\overline\bbeta-\bbeta^*\|_2=O_{\PP}\bigl(\sqrt{d/n}\bigr).
		\end{equation}
\end{theorem}

\begin{remark}
		\label{rem:glm2}
		In analogy to Theorem \ref{thm:ols2}, by constraining the growth rate of the number of subsamples according to $k=o\bigl(\sqrt{n}/(d \vee \log n)\bigr)$, the error incurred by the divide and conquer procedure, i.e., $\|\overline\bbeta-\hat\bbeta\|_2$ decays at a faster rate than that of the statistical error of the full sample estimator $\hat\bbeta$.
	\end{remark}

As in the linear model, Lemma \ref{lm:est-inf-rate-glm} together with Theorem \ref{thm:glm2} allow us to study the theoretical properties of a refitting estimator for the high dimensional GLM. Estimation on the estimated support set is again a low dimensional problem, thus the $d$ of Theorem  \ref{thm:glm2} corresponds to the $s$ of Lemma \ref{lm:est-inf-rate-glm} in this refitting setting. The refitted GLM estimator is defined as
\begin{equation}
	\overline\bbeta^r=\frac{1}{k}\sum\limits_{j=1}^k \hat\bbeta^r(\cD_j),
\end{equation}
where $\hat\bbeta^r(\cD_j)=\argmin_{\bbeta\in\RR^d, \bbeta_{\hat S^c}=0}\ell_{n_{k}}\sj(\bbeta)$ and $\hat S:=\{j: |\overline\bbeta^d_j|>2C\sqrt{\log d/n}\}$. The following corollary quantifies the statistical rate of $\overline\bbeta^r$.

\begin{corollary}
	\label{cor:refitglm}
	Suppose $\beta^*_{\min}>2C\sqrt{\log d/n}$, where $\beta^*_{\min}:=\min_{1\leq j\leq d} |\beta^*_j|$ and $C$ is the same constant as in (\ref{eq:est-inf-rate-glm}). Define the full sample oracle estimator as $\hat\bbeta^o=\argmin_{\bbeta\in\RR^d, \bbeta_{S^c}=0}\ell_{n}(\bbeta)$, where $S$ is the true support of $\bbeta^*$. If $k=O\bigl(\sqrt{n/((s\vee s_{1})^2\log d)}\bigr)$, then for sufficiently large $n$ and $d$ we have
		\begin{equation}
			\|\overline\bbeta^{r}-\hat\bbeta^{o}\|_2=O_{\PP}\Bigl(\frac{k\sqrt s (s \vee \log n)}{n}\Bigr), \quad  \|\overline\bbeta^{r}-\bbeta^*\|_2=O_{\PP}\bigl(\sqrt{s/n}\bigr).
		\end{equation}
\end{corollary}
We thus see that $\overline\bbeta^r$ achieves the oracle rate when the minimum signal strength is not too weak and the number of subsamples $k$ is not too large.

\section{Numerical Analysis}\label{sectionNumerical}

In this section, we illustrate and validate our theoretical findings through simulations. For inference, we use QQ plots to compare the distribution of p-values for divide and conquer test statistics to their theoretical uniform distribution. We also investigate the estimated type I error and power of the divide and conquer tests. For estimation, we validate our claim of Sections \ref{sectionEstimationLD1} and \ref{sectionEstimationLD2} that the loss incurred by the divide and conquer strategy is negligible compared to the statistical error of the corresponding full sample estimator in the low dimensional case. An analogous empirical verification of the theory is performed for the high dimensional case, where we also compare the performance of the divide and conquer thresholding estimator of Section \ref{sectionEstimationHD} to the full sample LASSO and the average LASSO over subsamples.

\subsection{Results on Inference}\label{sectionNumericalInference}

We explore the probability of rejection of a null hypothesis of the form $H_{0}:\beta^{*}_{1}=0$ when data $(Y_{i},\bX_{i})_{i=1}^{n}$ are generated according to the linear model,
\[
\bY_{i}=\bX_{i}^{T}\bbeta^{*}+\epsilon_{i}, \quad \epsilon_{i}\sim N\bigl(0,\sigma_{\epsilon}^{2}\bigr),
\]
for $\sigma^{2}_{\epsilon}=1$ and $\bbeta^{*}$ an $s$ sparse $d$ dimensional vector with $d=850$ and $s=3$. In each Monte Carlo replication, we split the initial sample of size $n$ into $k$ subsamples of size $n/k$. In particular we choose $n=840$ because it has a large number of factors $k\in\{1,2,5,10,15,20,24,28,30,35,40\}$. The number of simulations is 250.  Using $\widehat{\bbeta}_{\text{LASSO}}$ as a preliminary estimator of $\bbeta^{*}$, we construct Wald and Rao score test statistics as described in Sections \ref{sectionWaldvdG} and \ref{sectionScore} respectively.

\begin{figure}[t]
\centering
\begin{tabular}{cc}
(A) Wald test&(B) Score test \\[-3pt]
 \includegraphics[trim=0.74in 2.7in 1in 2.7in, clip, width=.41\linewidth, height=.35\linewidth ]{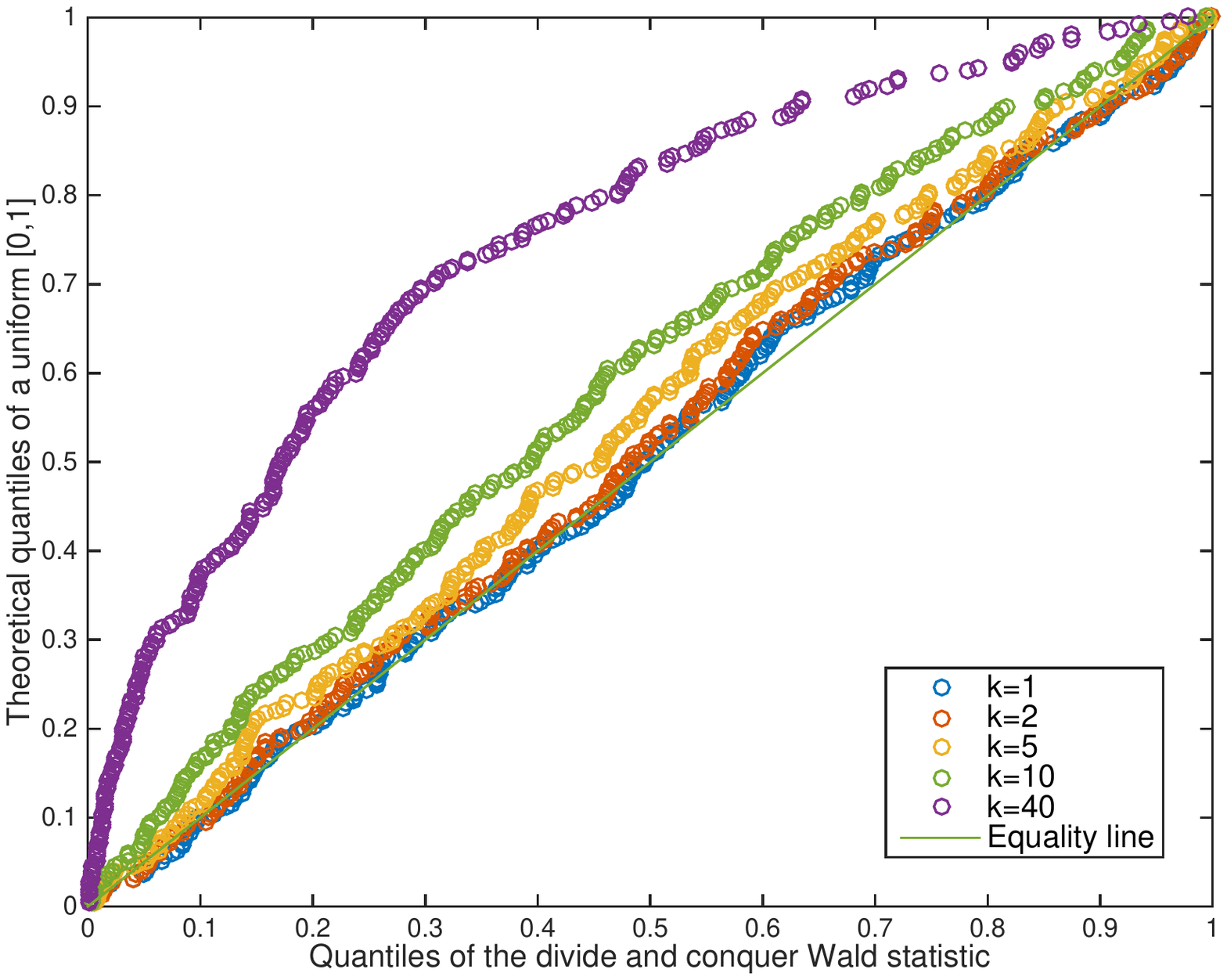}
 &\includegraphics[trim=0.74in 2.7in 1in 2.7in, clip, width=.41\linewidth, height=.35\linewidth ]{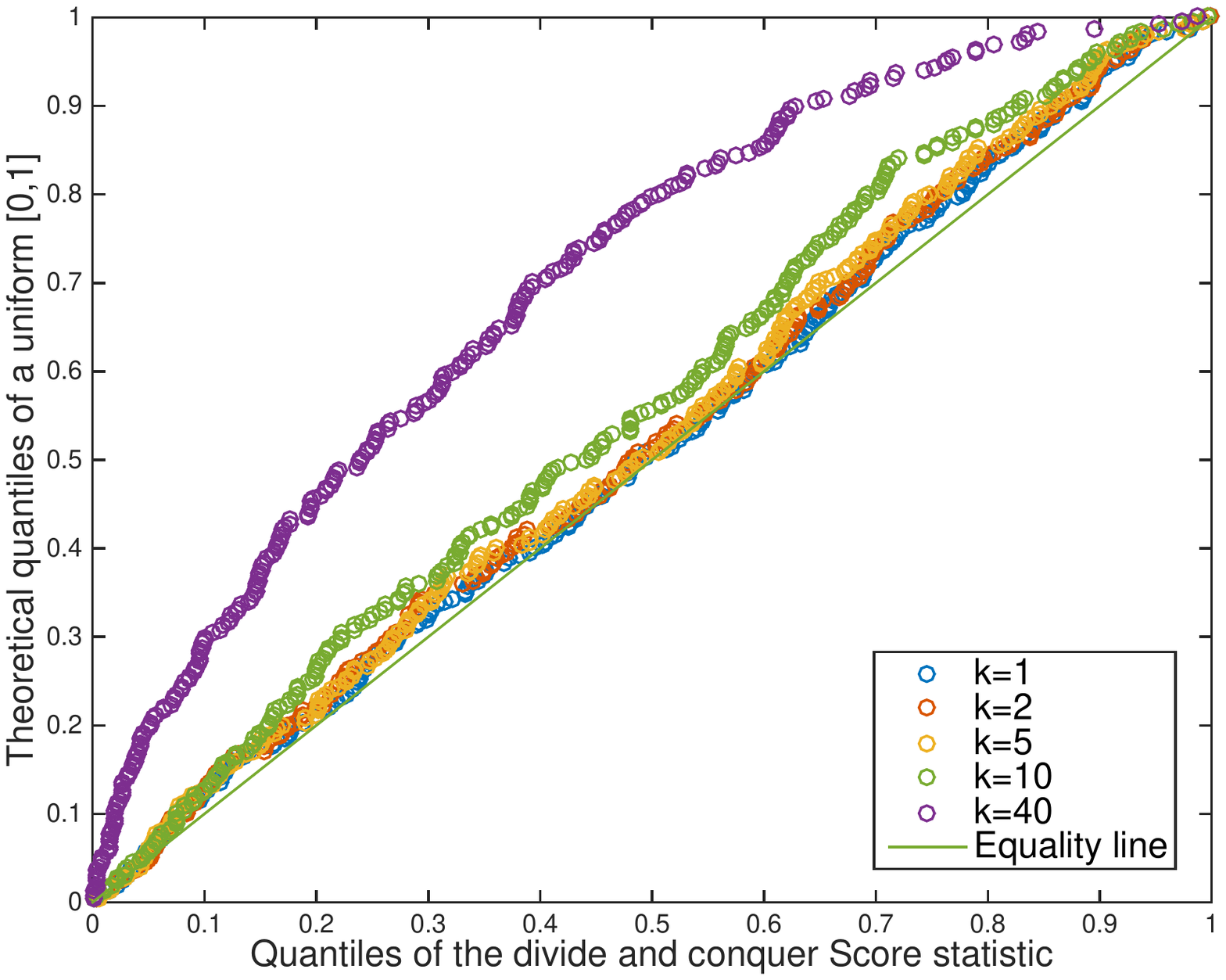}
\end{tabular}
\vskip-2pt
\caption{QQ plots of the p-values of the Wald (A) and score (B) divide and conquer test statistics against the theoretical quantiles of the uniform [0,1] distribution under the null hypothesis.}
\vskip-10pt
\label{figureAll}
\end{figure}

%\bigskip
%
%  \begin{minipage}{\linewidth}
%{\centering
%
%(A) \includegraphics[trim=0.74in 2.7in 1in 2.7in, clip, width=.41\linewidth, height=.35\linewidth ]{./Figs/QQPValsWald20150729.pdf}
%(B) \includegraphics[trim=0.74in 2.7in 1in 2.7in, clip, width=.41\linewidth, height=.35\linewidth ]{./Figs/QQPValsScore20150729.pdf}
%%(C) \includegraphics[trim=0.74in 2.7in 1in 2.7in, clip, width=.41\linewidth, height=.35\linewidth ]{./Figs/QQPValsLR20150729.pdf}
%
%%  \includegraphics[trim=0.9in 2.7in 1in 2.7in, clip, height=0.2\paperwidth]{./Figs/QQPValsLR.pdf}
%
%\figcaption{QQ plots of the p-values of the Wald (A) and score (B) divide and conquer test statistics against the theoretical quantiles of ta uniform [0,1] distribution under the null hypothesis.}\label{figureAll}
%
%}
%
%  \end{minipage}
%
%\bigskip

Panels (A) and (B) of Figure \ref{figureAll} are QQ plots of the p-values of the divide and conquer Wald and score test statistics under the null hypothesis against the theoretical quantiles of of the uniform [0,1] distribution for four different values of $k$. For both test constructions, the distributions of the p-values are close to uniform and remain so as we split the data set.
When $k=40$, the distribution of the corresponding p-values is visibly non-uniform, as expected from the theory developed in Sections \ref{sectionWaldvdG} and \ref{sectionScore}. Panel (A) of Figure \ref{figurePowerAll} also supports our theoretical findings, showing that, for both test constructions, the empirical level of the test is close to both the nominal $\alpha=0.05$ level and the level of the full sample oracle OLS estimator which knows the true support of $\bbeta^{*}$. Moreover, it remains at this level as long as we do not split the data set too many times. Panel (B) of Figure \ref{figurePowerAll} displays the power of the test for two different signal strengths, $\beta_{1}^{*}=0.125$ and $\beta_{1}^{*}=0.15$. We see that the power is also comparable with that of the full sample oracle OLS estimator which knows the true support of $\bbeta^{*}$.

\bigskip

\begin{figure}[t]
\centering
\begin{tabular}{cc}
\qquad (A) Type I error& \qquad (B) Power\\[-2pt]
 \qquad \includegraphics[trim=0.74in 2.7in 0.4in 2.7in, clip, width=.43\linewidth, height=.37\linewidth ]{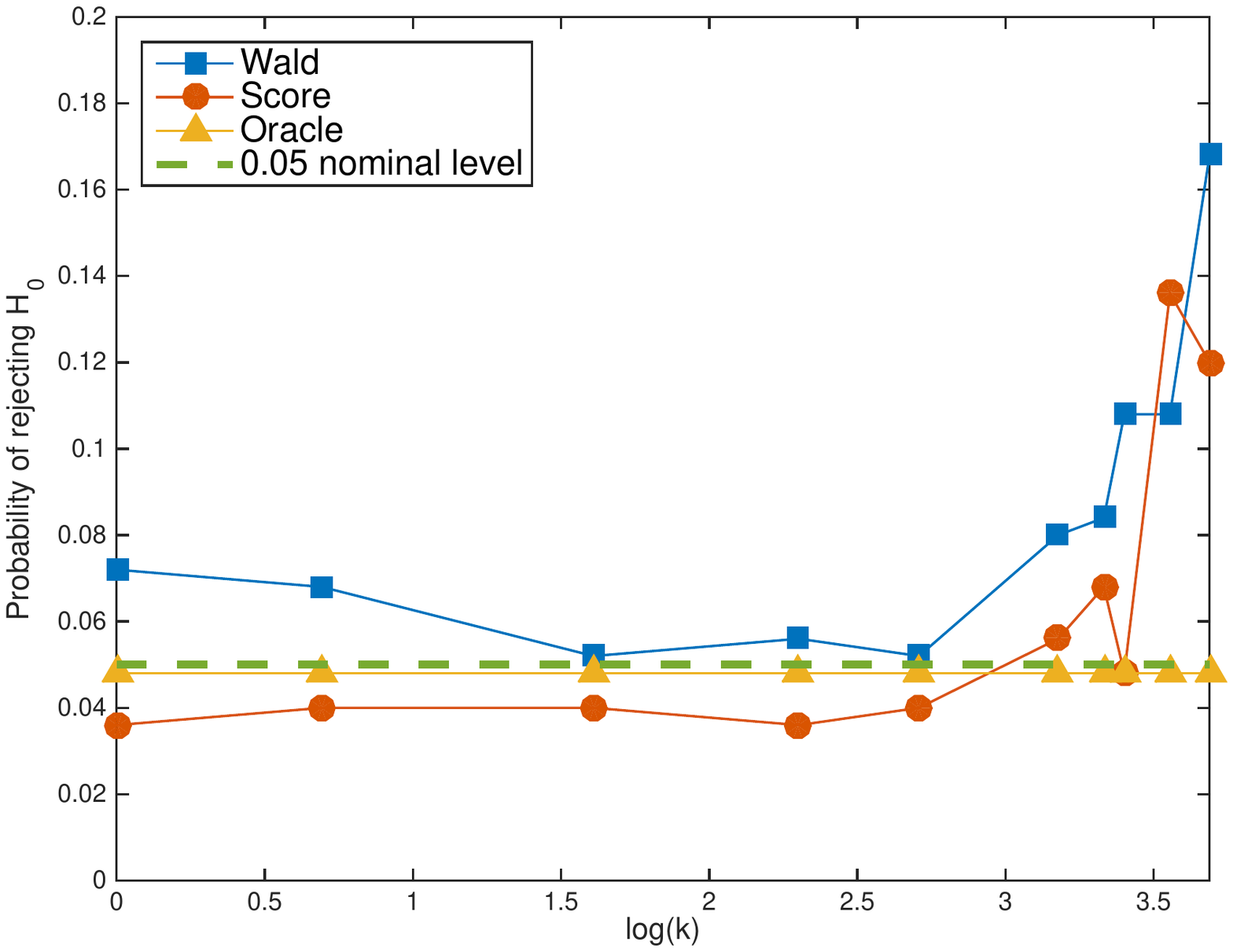}
 & \qquad \includegraphics[trim=0.74in 2.7in 0.4in 2.7in, clip, width=.43\linewidth, height=.37\linewidth ]{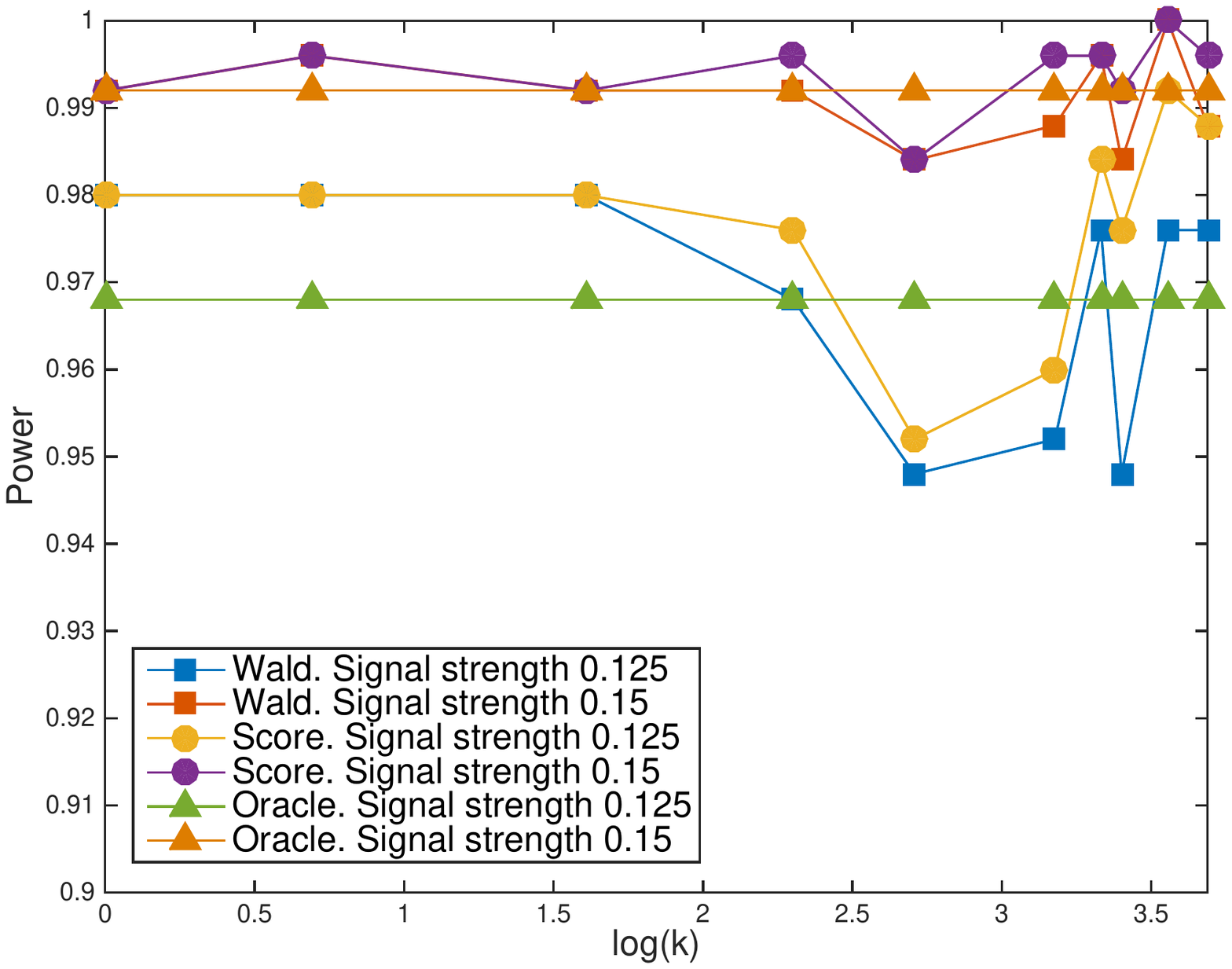}
\end{tabular}
\vskip-5pt
\caption{(A) Estimated probabilities of type I error for the Wald and score tests as a   function of $k$. (B) Estimated power with signal strength 0.125 and  0.15 for the Wald, and score tests as a function of $k$.}
\vskip-10pt
\label{figurePowerAll}
\end{figure}

%\bigskip
%
%  \begin{minipage}{\linewidth}
%{\centering
%
%%(A) \includegraphics[trim=0.74in 2.7in 0.4in 2.7in, clip, height=0.24\paperwidth]{./Figs/typeIFigNoLR20150722.pdf}
%%(B) \includegraphics[trim=0.74in 2.7in 0.4in 2.7in, clip, height=0.24\paperwidth]{./Figs/powerAllNoLR20150722.pdf}
%(A) \includegraphics[trim=0.74in 2.7in 0.4in 2.7in, clip, width=.43\linewidth, height=.37\linewidth ]{./Figs/typeIFigNoLR20150722.pdf}
%(B) \includegraphics[trim=0.74in 2.7in 0.4in 2.7in, clip, width=.43\linewidth, height=.37\linewidth ]{./Figs/powerAllNoLR20150722.pdf}
%
%}
%\figcaption{(A) Estimated probabilities of type I error for the Wald and score tests as a \newline function of $k$. (B) Estimated probabilities of type II error with signal strength 0.125 and \newline 0.15 for the Wald, and score tests as a function of $k$.}
%\label{figurePowerAll}
%  \end{minipage}

\subsection{Results on Estimation}

In this section, we turn our attention to experimental validation of our divide and conquer estimation theory, focusing first on the low dimensional case and then on the high dimensional case.

\subsubsection{The Low-Dimensional Linear Model}

All $n\times d$ entries of the design matrix $X$ are generated as i.i.d. standard normal random variables and the errors $\{\epsilon_i\}_{i=1}^n$ are i.i.d. standard normal as well. The true regression vector $\bbeta^*$ satisfies $\beta^*_j=10/\sqrt{d}$ for $j=1,\ldots, d/2$ and $\beta^*_j=-10/\sqrt{d}$ for $j>d/2$, which guarantees that $\|\bbeta^*\|_2=10$. Then we generate the response variable $\{Y_i\}_{i=1}^n$ according to the model (\ref{eq:lm}). Denote the full sample ordinary least-squares estimator and the divide and conquer estimator by $\hat\bbeta$ and $\overline\bbeta$ respectively. Figure \ref{fig:ols}(A) illustrates the change in the ratio $\|\overline\bbeta-\hat\bbeta\|_2/\|\hat\bbeta-\bbeta^*\|_2$ as the sample size increases, where $k$ assumes three different growth rates and $d=\sqrt{n}/2$. Figure \ref{fig:ols}(B) focuses on the relationship between the statistical error of $\overline\bbeta$ and $\log k$ under three different scalings of $n$ and $d$. All the data points are obtained based on average over 100 Monte Carlo replications.

\begin{figure}[t]
\centering
\begin{tabular}{cc}
(A) $\|\overline\bbeta-\hat\bbeta\|_2/\|\hat\bbeta-\bbeta^*\|_2$ &(B) $\|\overline\bbeta- \bbeta^*\|_2/\|\bbeta^*\|_2$\\[-2pt]
 \includegraphics[width=.48\linewidth, height=.37\linewidth]{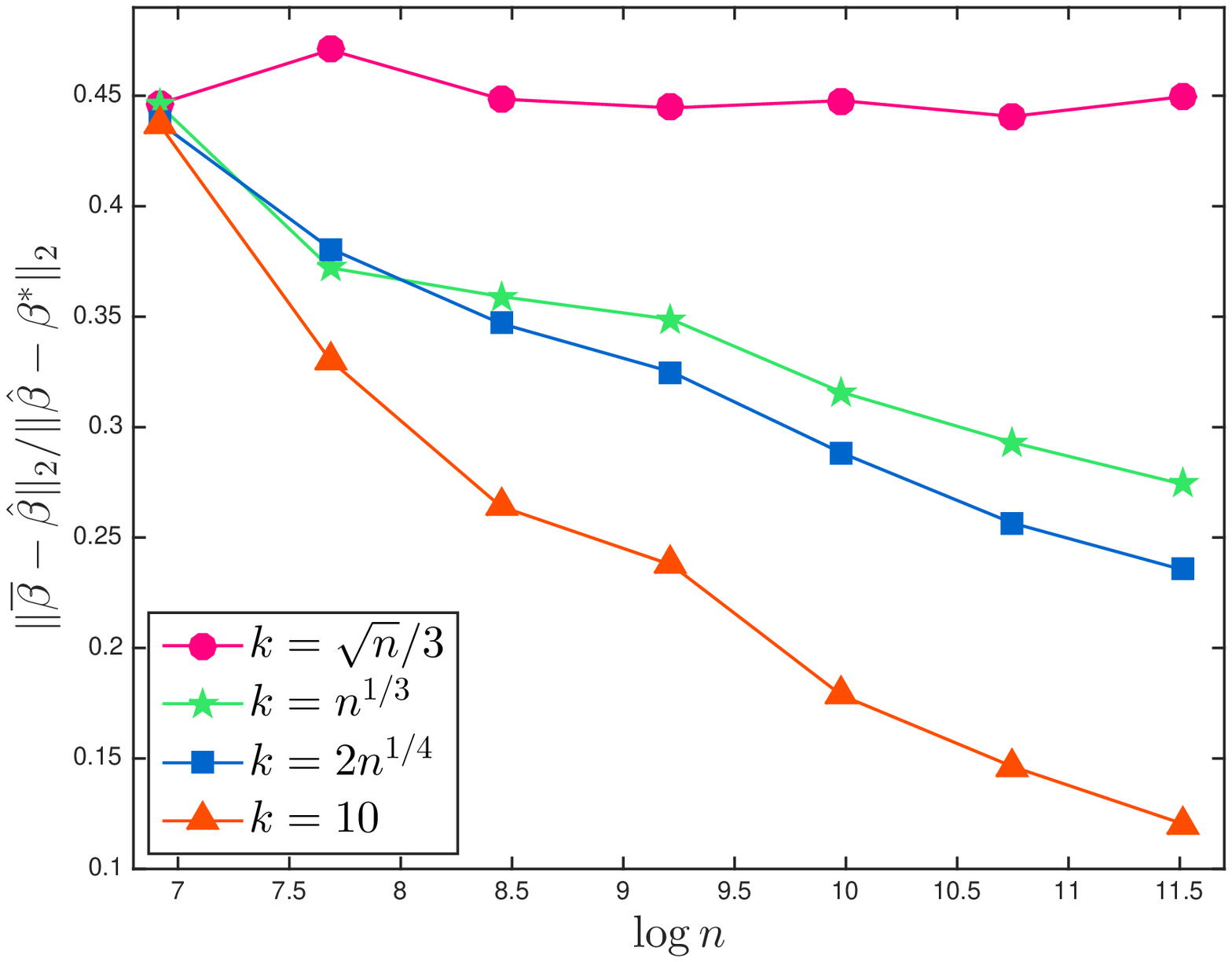}
 & \includegraphics[width=.48\linewidth, height=.37\linewidth]{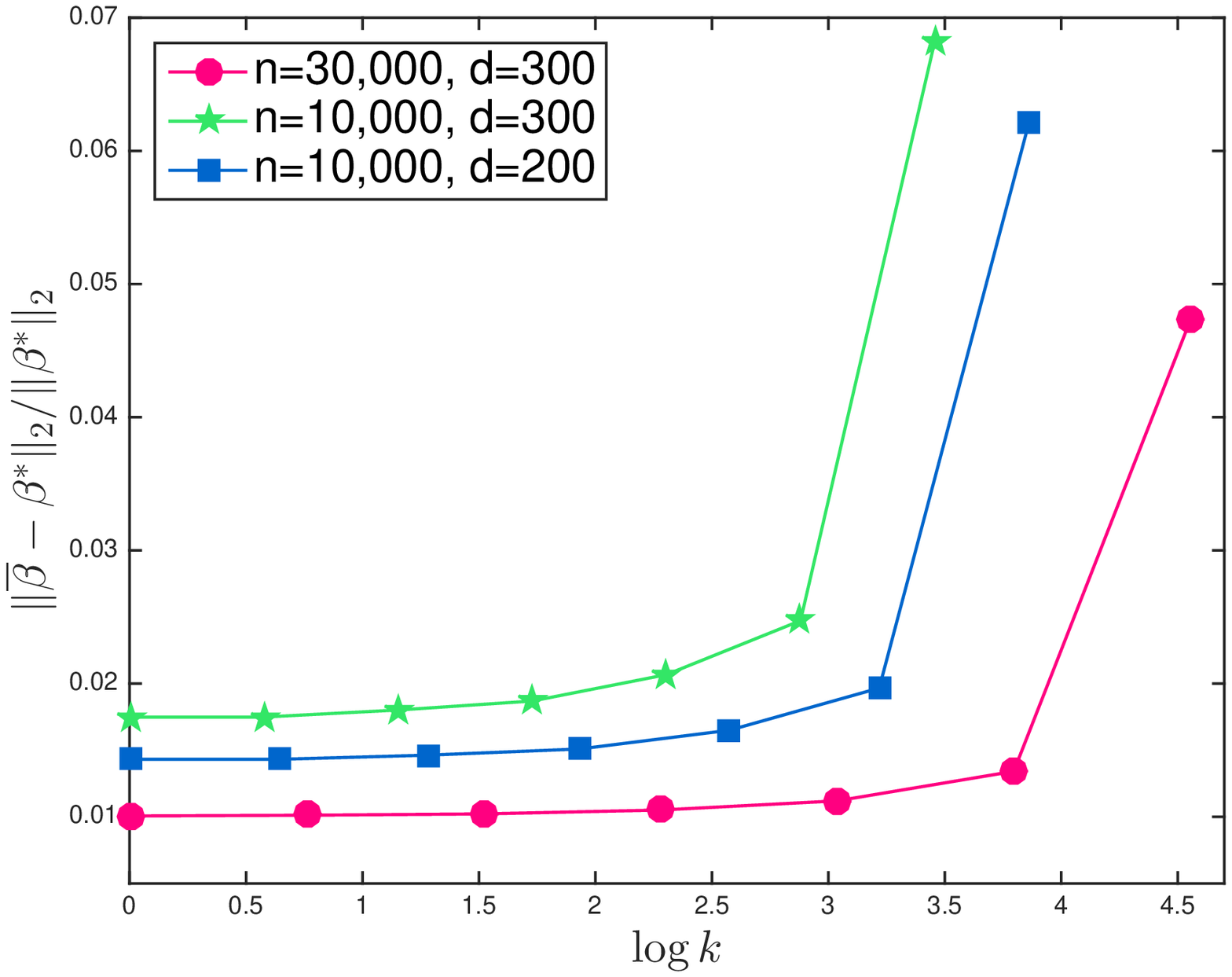}
\end{tabular}
\vskip-5pt
\caption{(A) The ratio between the loss of the divide and conquer procedure and the statistical error of the estimator based on the whole sample with $d=\sqrt{n}/2$ and different growth rates of $k$. (B) Statistical error of the DC estimator against $\log k$.}
\vskip-10pt
\label{fig:ols}
\end{figure}

%\setlength{\tabcolsep}{15pt}
%	\renewcommand{\arraystretch}{1.2}
%
%\bigskip
%
%  \begin{minipage}{\linewidth}
%
%(A) \includegraphics[trim=1cm 0 1cm 0, clip, width=.43\linewidth, height=.37\linewidth]{./Figs/fig2.eps}
%(B) \includegraphics[trim=1cm 0 1cm 0, clip, width=.43\linewidth, height=.37\linewidth]{./Figs/fig1.eps}
%
%\figcaption{(A) The ratio between the loss of the divide and conquer and the statistical error of the estimator based on the whole sample with $d=\sqrt{n}/2$ and different growing rates of $k$. (B) Statistical error of the DC estimator against $\log k$. }
%
%\label{fig:ols}
%  \end{minipage}
%
%\bigskip

	As Figure \ref{fig:ols}(A) demonstrates, when $k=O(n^{1/3})$, $O(n^{1/4})$ or $O(1)$, the ratio decreases with ever faster rates, which is consistent with the argument of Remark \ref{rem:ols2} that the ratio goes to zero when $k=o(n/d)=o(\sqrt{n})$. When $k=O(\sqrt{n})$, however, we observe that the ratio is essentially constant, which suggests the rate we derived in Theorem \ref{thm:ols2} is sharp.
	
	From Figure \ref{fig:ols}(B), we see that when $k$ is not large, the statistical error of $\overline\bbeta$ is very small because the loss incurred by the divide and conquer procedure is negligible compared to the statistical error of $\hat\bbeta$.  However, when $k$ is larger than a threshold, there is a surge in the statistical error, since the loss of the divide and conquer begins to dominate the statistical error of $\hat\bbeta$. We also notice that the larger the ratio $n/d$, the larger the threshold of $\log k$, which is again consistent with Remark \ref{rem:ols2}.
	
\subsubsection{The Low-Dimensional Logistic Regression}

In logistic regression, given covariates $\bX$, the response $Y|\bX \sim \text{Ber}(\eta(\bX))$, where $\text{Ber}(\eta)$ denotes the Bernoulli distribution with expectation $\eta$ and
	\[
		\eta(\bX)=\frac{1}{1+\exp(-{\bX^T}\bbeta^*)}.
	\]
	We see that $\text{Ber}(\eta(\bX))$ is in exponential dispersion family canonical form (\ref{eq:glm}) with $b(\theta)=\log(1+e^\theta)$, $\phi=1$ and $c(y)=1$. The use of the canonical link,
	\[
		\eta(\bX)=\frac{1}{1+e^{-\theta(\bX)}},
	\]
	leads to the simplification $\theta(\bX)={\bX^T}\bbeta^*$.
	
	In our Monte Carlo experiments, all $n\times d$ entries of the design matrix ${X}$ are generated as i.i.d.~standard normal random variables. The true regression vector $\bbeta^*$ satisfies $\beta^*_j=1/\sqrt{d}$ for $j\le d/2$ and $\beta^*_j=-1/\sqrt{d}$ for $j>d/2$, which guarantees that $\|\bbeta^*\|_2=1$. Finally, we generate the response variables $\{Y_i\}_{i=1}^n$ according to $\text{Ber}(\eta(\bX))$. Figure \ref{fig:glm}(A) illustrates the change of the ratio $\|\overline\bbeta-\hat\bbeta\|_2/\|\hat\bbeta-\bbeta^*\|_2$ as the sample size increases, where $k$ assumes three different growths rates and $d=20$. Figure \ref{fig:glm}(B) focuses on the relationship between the statistical error of $\overline\bbeta$ and $\log k$ under three different scalings of $n$ and $d$. All the data points are obtained based on an average over 100 Monte Carlo replications.

\begin{figure}[t]
\centering
\begin{tabular}{cc}
(A) $\|\overline\bbeta-\hat\bbeta\|_2/\|\hat\bbeta-\bbeta^*\|_2$ &(B) $\|\overline\bbeta- \bbeta^*\|_2/\|\bbeta^*\|_2$\\
 \includegraphics[width=.48\linewidth, height=.37\linewidth]{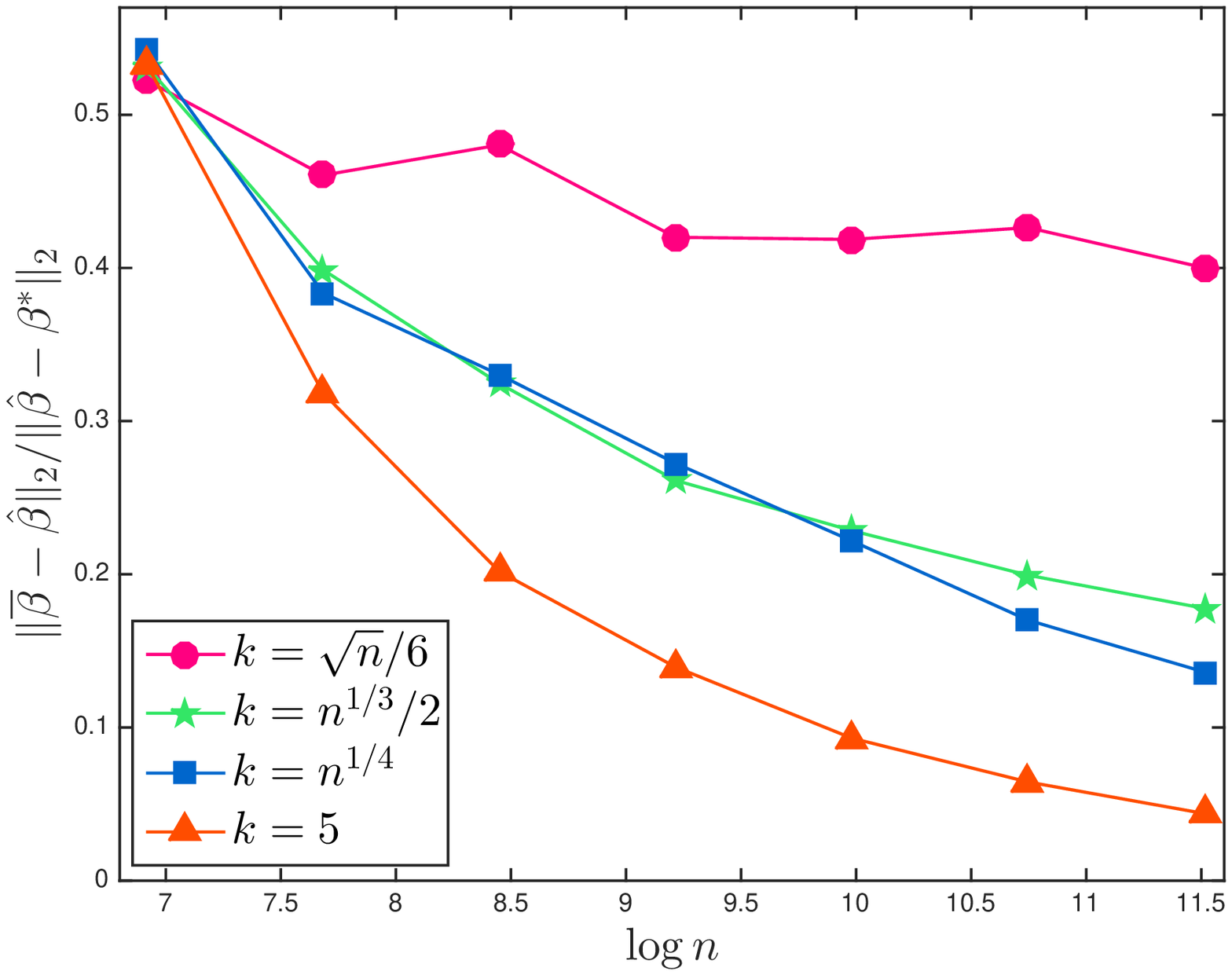}
 & \includegraphics[width=.48\linewidth, height=.37\linewidth]{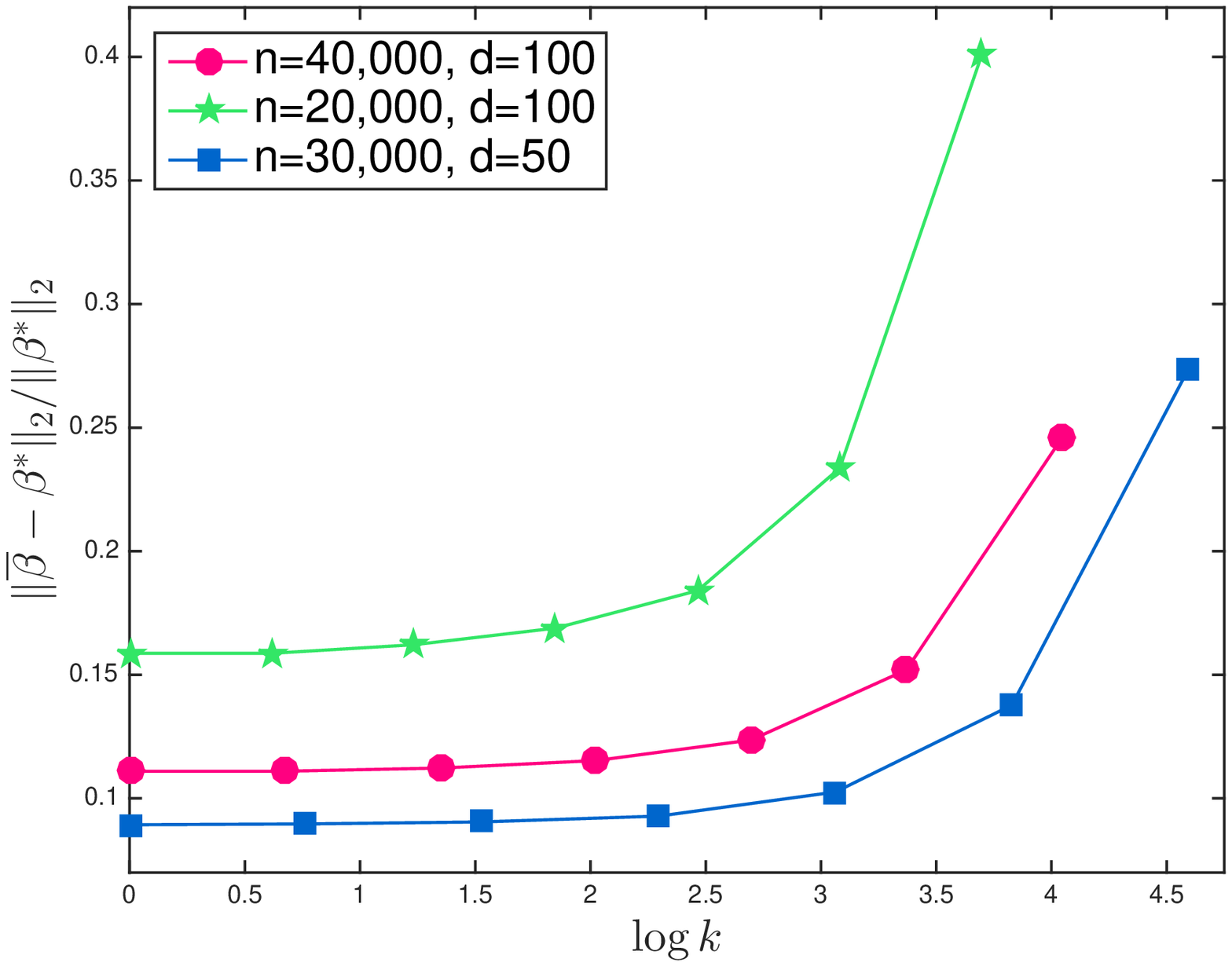}
\end{tabular}
\vskip-5pt
\caption{(A) The ratio between the loss of the divide and conquer procedure and the statistical error of the estimator based on the whole sample when $d=20$. (B) Statistical error of the DC estimator.}
\vskip-10pt
\label{fig:glm}
\end{figure}
	
%	\bigskip
%
%  \begin{minipage}{\linewidth}
%
%(A) \includegraphics[trim=1cm 0 1cm 0, clip, width=.43\linewidth, height=.37\linewidth]{./Figs/fig4.eps}
%(B) \includegraphics[trim=1cm 0 1cm 0, clip, width=.43\linewidth, height=.37\linewidth]{./Figs/fig3.eps}
%
%\figcaption{(A) The ratio between the loss of the divide and conquer and the statistical error of the estimator based on the whole sample when $d=20$. (B) Statistical Error of the DC estimator. }
%
%\label{fig:glm}
%  \end{minipage}
%
%  \bigskip

Figure \ref{fig:glm}  reveals similar phenomena to those revealed in Figure \ref{fig:ols} of the previous subsection. More specifically, Figure \ref{fig:glm}(A) shows that when $k=O(n^{1/3})$, $O(n^{1/4})$ or $O(1)$, the ratio decreases with even faster rates, which is consistent with the argument of Remark \ref{rem:glm2} that the ratio converges to zero when $k=o(\sqrt{n}/d)=o(\sqrt{n})$. When $k=O(\sqrt{n})$, however, we observe that the ratio remains essentially constant when $\log n$ is large, which suggests the rate we derived in Theorem \ref{thm:ols2} is sharp.

As for Figure \ref{fig:glm}(B), we again observe that the statistical error of $\overline\bbeta$ is very small when $k$ is sufficiently small, but grows fast when $k$ becomes large. The reasoning is the same as in the linear model, i.e.~when $k$ is large, the loss incurred by the divide and conquer procedure is non-negligible as compared with the statistical error of $\|\hat\bbeta\|_2$. In addition, as Figure \ref{fig:glm}(B) reveals, the larger is $\sqrt{n}/d$, the larger the threshold of $k$, which is again consistent with the threshold rate pointed out in Remark \ref{rem:glm2}.

\subsubsection{The High Dimensional Linear Model}

We now consider the same setting of Section \ref{sectionNumericalInference} with $n=1400$, $d=1500$ and $\beta_{j}^{*}=10$ for all $j$ in the support of $\bbeta^{*}$. In this context, we analyze the performance of the thresholded averaged debiased estimator of Section \ref{sectionEstimationHD}. Figure \ref{figureEstimationHD}(A) depicts the average over 100 Monte Carlo replications of $\|\bb - \bbeta^{*}\|_{2}$ for three different estimators: debiased divide-and-conquer $\bb=\mathcal{T}_{\nu}(\overline{\bbeta}^{d})$, the LASSO estimator based on the whole sample  $\bb=\widehat{\bbeta}_{\text{LASSO}}$ and the estimator obtained by na\"ively averaging the LASSO estimators from the $k$ subsamples $\bb=\overline{\bbeta}_{\text{LASSO}}$. The parameter $\nu$ is taken as $\nu=\sqrt{\log d/n}$ in the specification of $\mathcal{T}_{\nu}(\overline{\bbeta}^{d})$. As expected, the performance of $\overline{\bbeta}_{\text{LASSO}}$ deteriorates sharply as $k$ increases. $\mathcal{T}_{\nu}(\overline{\bbeta}^{d})$ outperforms $\widehat{\bbeta}_{\text{LASSO}}$ as long as $k$ is not too large. This is expected because, for sufficiently large signal strength, both $\widehat{\bbeta}_{\text{LASSO}}$ and $\mathcal{T}_{\nu}(\overline{\bbeta}^{d})$ recover the correct support, however $\mathcal{T}_{\nu}(\overline{\bbeta}^{d})$ is unbiased for those $\beta_{j}^{*}$ in the support of $\bbeta^{*}$, whilst $\widehat{\bbeta}_{\text{LASSO}}$ is biased. Figure \ref{figureEstimationHD}(B) shows the error incurred by the divide and conquer procedure $\|\mathcal{T}_{\nu}(\overline{\bbeta}^{d})-\mathcal{T}_{\nu}(\widehat{\bbeta}^{d})\|_{2}$ relative to the statistical error of the full sample estimator, $\|\mathcal{T}_{\nu}(\overline{\bbeta}^{d}) - \bbeta^{*}\|_{2}$, for four different scalings of $k$. We observe that, with $k=O(\sqrt{n/s^{2}\log d})$ and $n$ not too small, the relative error incurred by the divide and conquer procedure is essentially constant across $n$, demonstrating the theory developed in Theorem \ref{thm:est-l2-rate}.

\begin{figure}[t]
\centering
\begin{tabular}{cc}
(A) Estimation error &(B) DC error\\[-2pt]
\includegraphics[trim=0.74in 2.7in 0.4in 2.7in, clip, width=.43\linewidth, height=.37\linewidth ]{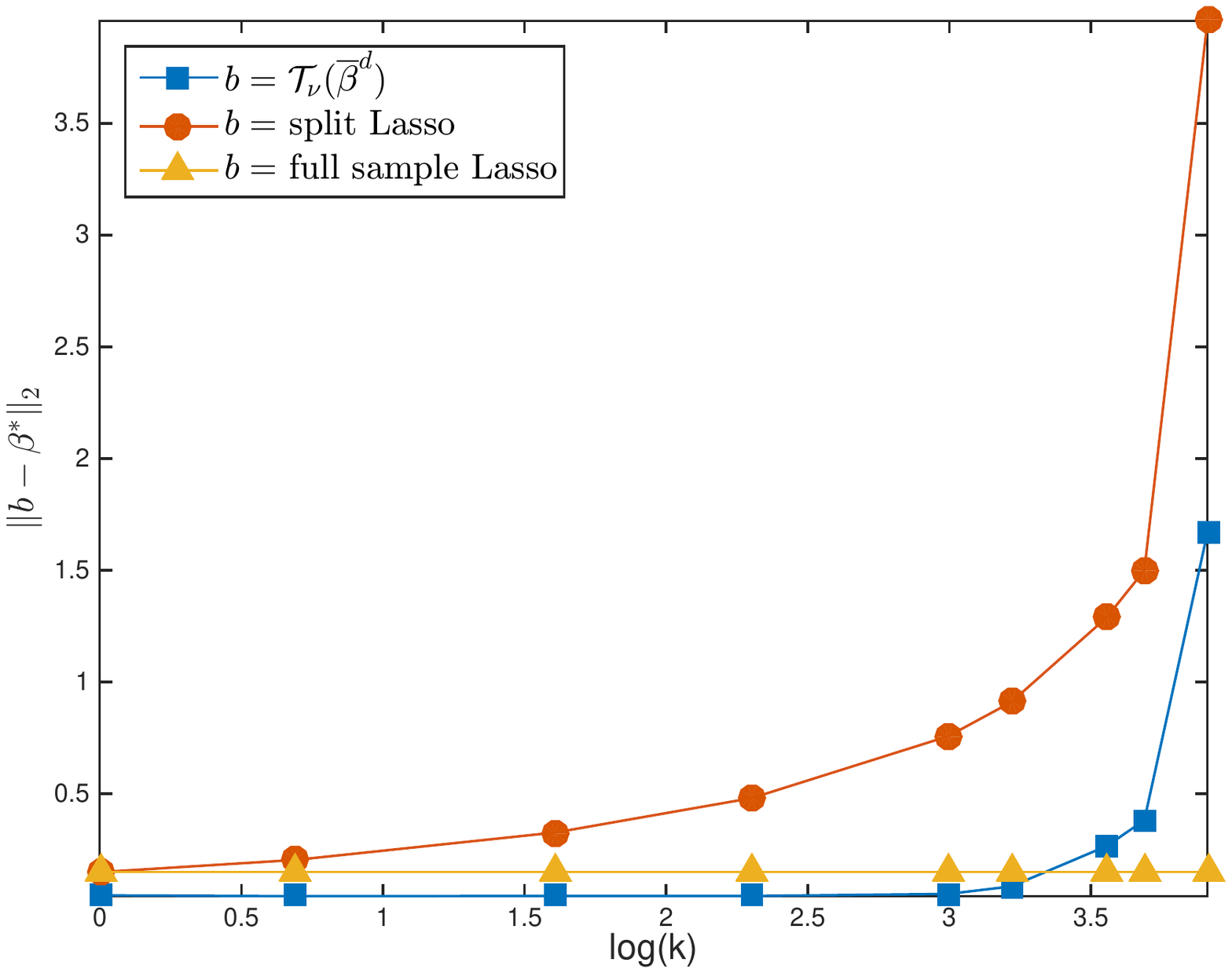}
 &\includegraphics[trim=0.74in 2.7in 0.4in 2.7in, clip, width=.43\linewidth, height=.37\linewidth ]{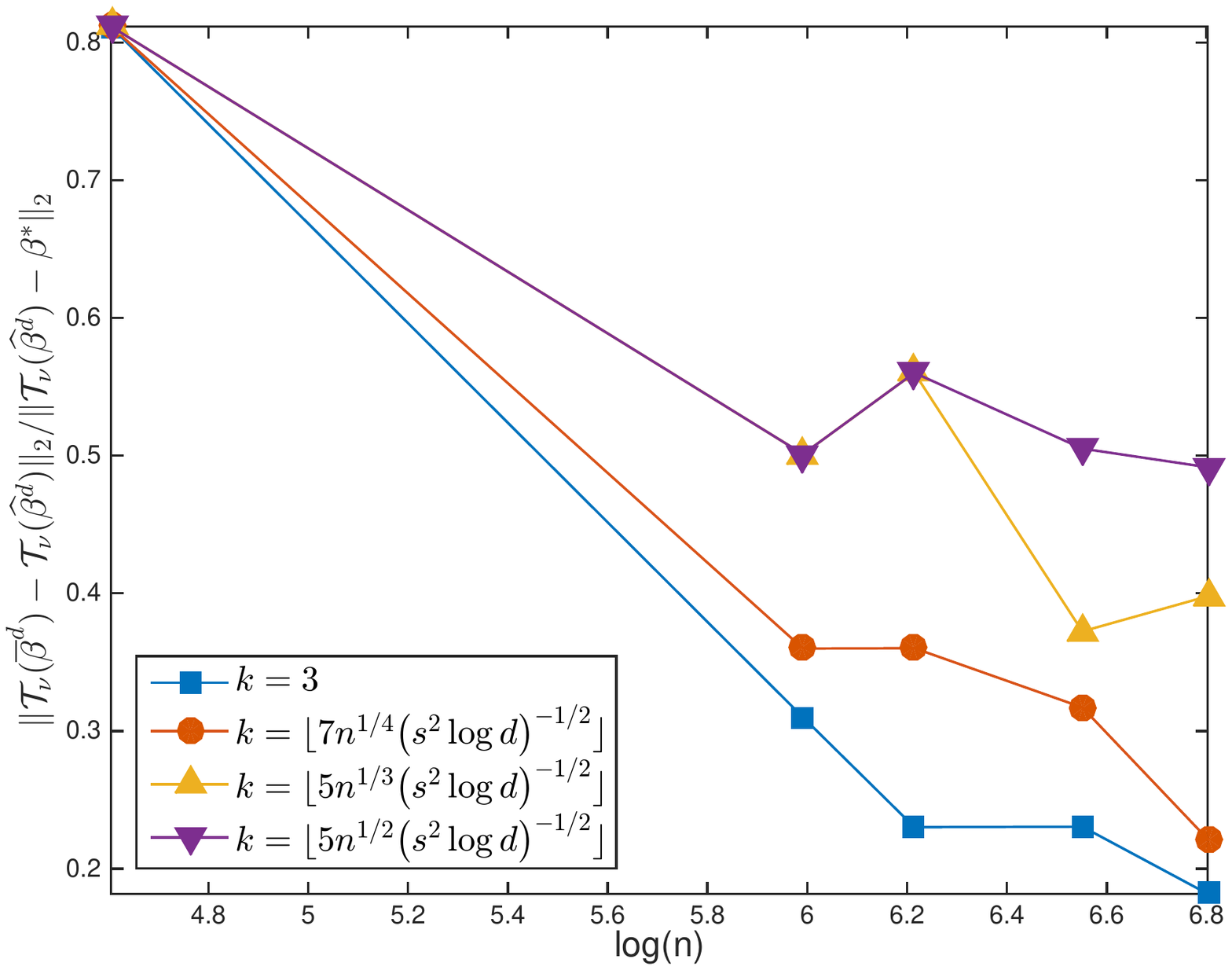}
\end{tabular}
\vskip-5pt
\caption{(A): Statistical error of the DC estimator, split LASSO and the full sample LASSO for $k\in\{1,2,5,10,20,25,35,40,50\}$ when $n=1400$, $d=1500$. (B): Euclidean norm difference between the DC thresholded debiased estimator and its full sample analogue.}
\vskip-10pt
\label{figureEstimationHD}
\end{figure}

\section{Discussion}\label{sectionDiscussion}

With the advent of the data revolution comes the need to modernize the classical statistical toolkit. For very large scale datasets, distribution of data across multiple machines is the only practical way to overcome storage and computational limitations. It is thus essential to build aggregation procedures for conducting inference based on the combined output of multiple machines. We successfully achieve this objective, deriving divide and conquer analogues of the Wald and score  statistics and providing statistical guarantees on their performance as the number of sample splits grows to infinity with the full sample size. Tractable limit distributions of each DC test statistic are derived. These distributions are valid as long as the number of subsamples, $k$, does not grow too quickly. In particular, $k=o\bigl(((s\vee s_{1})\log d)^{-1}\sqrt{n}\bigr)$ is required in a general likelihood based framework. If $k$ grows faster than $((s\vee s_{1})\log d)^{-1}\sqrt{n}$, remainder terms become non-negligible and contaminate the tractable limit distribution of the leading term. When attention is restricted to the linear model, a faster growth rate of $k=o\bigl((s\log d)^{-1}\sqrt{n}\bigr)$ is allowed.

The divide and conquer strategy is also successfully applied to estimation of regression parameters. We obtain the rate of the loss incurred by the divide and conquer strategy. Based on this result, we derive an upper bound on the number of subsamples for preserving the statistical error. For low-dimensional models, simple averaging is shown to be effective in preserving the statistical error, so long as $k= O(n/d)$ for the linear model and $k= O(\sqrt{n}/d)$ for the generalized linear model. For high-dimensional models, the debiased estimator used in the Wald construction is also successfully employed, achieving the same statistical error as the LASSO based on the full sample, so long as $k= O(\sqrt{n/s^{2}\log d})$.

Our contribution advances the understanding of distributed inference and estimation in the presence of large scale and distributed data, but there is still a great deal of work to be done in the area. We focus here on the fundamentals of statistical inference and estimation in the divide and conquer setting. Beyond this, there is a whole toolkit of statistical methodology designed for the single sample setting, whose split sample asymptotic properties are yet to be understood.

\section{Proofs}\label{sectionProofs}

\normalsize{In this section, we present the proofs of the main theorems appearing in Sections \ref{sectionWald}-\ref{sectionEstimation}. The statements and proofs of several auxiliary lemmas appear in  the Supplementary Material. To simplify notation, we take $\beta_{v}^{H}=0$ without loss of generality.}

\subsection{Proofs for Section \ref{sectionWald}}\label{sectionProofWald}

The proof of Theorem \ref{thmWaldDistribution}, relies on the following lemma, which bounds the probability that optimization problems in \eqref{eq:m-est} are feasible.

\begin{lemma}\label{lemmaM}
Assume $\Sigma=\EE\bigl(\bX_{i}\bX_{i}^{T}\bigr)$ satisfies $C_{\min}<\lambda_{\min}(\Sigma) \leq \lambda_{\max}(\Sigma) \leq C_{\max}$ as well as $\|\Sigma^{-1/2}\bX_{1}\|_{\psi_{2}} = \kappa$, then we have
\[
\PP\left(  \max_{j = 1,\ldots, k}  \|M\sj \hat{\Sigma}\sj - I \|_{\max} \le a \sqrt{\frac{\log d}{n}} \right) \ge 1- 2 k d^{-c_2}, \text{ \emph{where}}~~c_2 = \frac{a^2 C_{\min}}{24e^2 \kappa^4 C_{\max}} - 2.
\]
\end{lemma}

\begin{proof}
 The proof is an application of the union bound in Lemma 6.2 of \citet{JMJMLR2014}.
\end{proof}

Using Lemma \ref{lemmaM} we now prove Theorem \ref{thmWaldExpression}, from which Theorem \ref{thmWaldDistribution} easily follows. The term $\bZ$ in the decomposition of $\sqrt{n}(\overline{\bbeta}^{d} - \bbeta^{*})$ in Theorem \ref{thmWaldExpression} is responsible for the asymptotic normality of the proposed DC Wald statistic in Theorem \ref{thmWaldDistribution}, while the upper bound on $k$ ensures $\bDelta$ is asymptotically negligible. %, ultimately allowing us to construct a test statistic with asymptotically standard normal limit distribution.

\begin{theorem}\label{thmWaldExpression}
Suppose Conditions \ref{con:sigma} and \ref{con:subg} are fulfilled. Let $\lambda  \asymp \sqrt{k\log d/n}$ and $\vartheta_{1} \asymp \sqrt{k\log d/n}$. With $k=o((s\log d)^{-1}\sqrt{n})$, $\sqrt{n}(\overline{\bbeta}^{d}-\bbeta^{*})=\bZ+\bDelta$, where $\bZ=\frac{1}{\sqrt{k}}\sum_{j=1}^{k} \frac{1}{\sqrt{n_{k}}}M^{(j)}X^{(j)T} \bepsilon^{(j)}$ and $\|\bDelta\|_{\infty}=o_{\PP}(1)$.
\end{theorem}

\begin{proof}
For notational convenience, we write $\hat{\bbeta}^{\lambda}_{\text{LASSO}}(\cD_{j})$ simply as $\hat{\bbeta}^{\lambda}(\cD_{j})$. Decompose $\overline{\bbeta}^{d}-\bbeta^{*}$ as
\begin{eqnarray*}
\overline{\bbeta}^{d}-\bbeta^{*} &=& \frac{1}{k}\sum_{j=1}^{k}\Bigl(\hat{\bbeta}^{\lambda}(\cD_{j})-\bbeta^{*} + \frac{1}{n_{k}}M^{(j)}X^{(j)T}X^{(j)}\bigl(\bbeta^{*}-\hat{\bbeta}^{\lambda}(\cD_{j})\bigr)\Bigr) + \frac{1}{k}\sum_{j=1}^{k} \frac{1}{n_{k}}M^{(j)}X^{(j)T} \epsilon^{(j)} \\
&=& \frac{1}{k}\sum_{j=1}^{k}\bigl( I - M^{(j)}\hat{\Sigma}^{(j)}\bigr)\bigl(\hat{\bbeta}^{\lambda}(\cD_{j})-\bbeta^{*} \bigr) + \frac{1}{k}\sum_{j=1}^{k} \frac{1}{n_{k}}M^{(j)}X^{(j)T} \epsilon^{(j)},
\end{eqnarray*}
hence $\sqrt{n}(\overline{\bbeta}^{d}-\bbeta^{*})=\bZ + \bDelta$, where
\[
\bZ=\frac{1}{\sqrt{k}}\sum_{j=1}^{k} \frac{1}{\sqrt{n_{k}}}M^{(j)}X^{(j)T} \epsilon^{(j)} \;\; \text{ and } \; \;
\bDelta=\sqrt{n}\frac{1}{k}\sum_{j=1}^{k}\bigl( I - M^{(j)}\hat{\Sigma}^{(j)}\bigr)\bigl(\hat{\bbeta}^{\lambda}(\cD_{j})-\bbeta^{*} \bigr).
\]
Defining $\bDelta^{(j)}=\bigl( I - M^{(j)}\hat{\Sigma}^{(j)}\bigr)\bigl(\hat{\bbeta}^{\lambda}(\cD_{j})-\bbeta^{*} \bigr)$, we have
\[
\|\bDelta^{(j)}\|_{\infty}\leq\|\bDelta^{(j)}\|_{1}\leq\|M^{(j)}\hat{\Sigma}^{(j)}-I\|_{\max}\|\hat{\bbeta}^{\lambda}(\cD_{j})-\bbeta^{*}\|_{1}
\]
by H\"older's inequality, where $\|I-M^{(j)}\hat{\Sigma}^{(j)}\|_{\max}\leq \vartheta_{1}$ by the definition of $M^{(j)}$ and, for $\lambda  = C\sigma^2\sqrt{\log d/n_k}$,
\begin{equation}\label{eq:l1-lasso}
\PP\left(\big\|\hat{\bbeta}^{\lambda}(\cD_{j}) - \bbeta^{*} \big\|^2_1 > C\frac{s^2\log (2d)}{n_k} + t\right) \le \exp\left(-\frac{cn_kt}{s^2\sigma^2}\right)
 \end{equation}
by \citet{vdGBBook2011}. We thus bound the expectation of the $\ell_1$ loss by
 \begin{eqnarray} \label{eq:exp-l1}
 \EE\left[ \big\|\hat{\bbeta}^{\lambda}(\cD_{j}) - \bbeta^{*} \big\|^2_1 \right]  &\le&  \frac{2Cs^2\log (2d)}{n_k} + \int_{0}^{\infty} \exp\left(-\frac{cn_{k}t}{s^{2}\sigma^{2}}\right) dt \le {\frac{2Cs^2\log (2d)}{n_k}} + \frac{s^2\sigma^2}{c{n_k}}.
\end{eqnarray}
Define the event $\cE^{(j)} := \big\{\big\|\hat{\bbeta}^{\lambda}(\cD_{j}) - \bbeta^{*} \big\|_1 \le s\sqrt{{C\log (2d)}/{n_k}}\big\}$ for $j= 1, \ldots, k$. $
\|\bDelta^{(j)}\|_{\infty} \leq \Delta^{(j)}_{1}  + \Delta^{(j)}_{2} + \Delta^{(j)}_{3}$ where
\begin{eqnarray*}
\Delta^{(j)}_{1}&=& \|M^{(j)}\hat{\Sigma}^{(j)} - I\|_{\max} \|\hat{\bbeta}^{\lambda}(\cD_{j}) - \bbeta^{*}\|_1\ind\{\cE^{(j)}\}\\
& & \quad \quad - \;\; \EE\big[ \|M^{(j)}\hat{\Sigma}^{(j)} - I\|_{\max} \|\hat{\bbeta}^{\lambda}(\cD_{j}) - \bbeta^{*}\|_1\ind\{\cE^{(j)}\}\big] \\
\Delta^{(j)}_{2} &=& \|M^{(j)}\hat{\Sigma}^{(j)} - I\|_{\max} \|\hat{\bbeta}^{\lambda}(\cD_{j}) - \bbeta^{*}\|_1\ind\{\cE^{(j)c}\}\\
& & \quad \quad - \;\; \EE[\|M^{(j)}\hat{\Sigma}^{(j)} - I\|_{\max} \|\hat{\bbeta}^{\lambda}(\cD_{j}) - \bbeta^{*}\|_1\ind\{\cE^{(j)c}\}] \quad \text{and }\\
\Delta^{(j)}_{3} &=& \EE[\|M^{(j)}\hat{\Sigma}^{(j)} - I\|_{\max} \|\hat{\bbeta}^{\lambda}(\cD_{j}) - \bbeta^{*}\|_1].
\end{eqnarray*}
Consider $\Delta^{(j)}_{1}$, $\Delta^{(j)}_{2}$ and $\Delta^{(j)}_{3}$ in turn. By Hoeffding's inequality, we have for any $t>0$,
\begin{equation}\label{eq:del1}
\PP\left( \frac{1}{k} \sum_{j=1}^{k}\Delta^{(j)}_{1} > t  \right) \le \exp\left(-\frac{n_kkt^2}{Cs^{2}\vartheta_{1}^2\log (2d) }\right) \le \exp\left(-\frac{n_k nt^2}{Cs^2\log^2(2d) }\right).
\end{equation}
By Markov's inequality,
\begin{align}
\nonumber  \PP\left( \frac{1}{k} \sum_{j=1}^{k} \Delta^{(j)}_{2} > t  \right) &\leq \frac{\sum_{j=1}^{k}\EE[\Delta^{(j)}_{2}]}{kt} \leq 2t^{-1}\EE\big[ \|M^{(j)}\hat{\Sigma}^{(j)} - I\|_{\max} \|\hat{\bbeta}^{\lambda}(\cD_{j}) - \bbeta^{*}\|_1\ind\{\cE^{(j)c}\}\big] \\
\nonumber  &\leq 2t^{-1} \vartheta_{1} \sqrt{\EE\big[ \|\hat{\bbeta}^{\lambda}(\cD_{j}) - \bbeta^{*}\|_1^{2}\big] \PP(\cE^{(j)c})} \\
  &\le Ct^{-1} \sqrt{\frac{\log d}{n_k} \cdot\frac{s^2\log (2d)}{n_k} d^{-c}} \le Ct^{-1} {sn_{k}^{-1}d^{-c/2}\log d},\label{eq:del2}
  \end{align}
where the penultimate inequality follows from Jensen's inequality. Finally, by Jensen's inequality again,
\begin{eqnarray}
 \nonumber \frac{1}{k} \sum_{j=1}^{k} \Delta^{(j)}_{3} &=& \EE[ \|M^{(j)}\hat{\Sigma}^{(j)} - I\|_{\max} \|\hat{\bbeta}^{\lambda}(\cD_{j}) - \bbeta^{*}\|_1]\\
 &\leq & \vartheta_{1}  \sqrt{\EE\left[ \big\|\hat{\bbeta}^{\lambda}(\cD_{j}) - \bbeta^{*}\big\|^2_1 \right]}\leq C {\frac{s \log d}{n_k} }.\label{eq:del3}
\end{eqnarray}
Combining \eqref{eq:del1}, \eqref{eq:del2} and \eqref{eq:del3},
\begin{align}\label{eqDeltaBound}
 \nonumber \PP\left(\|\bDelta\|_{\infty} > 3C\sqrt{n}\cdot  \frac{s \log d}{n_k}\right) &\leq \sum_{u = 1}^3\PP\left(\frac{1}{k} \sum_{j=1}^{k} \Delta^{(j)}_{u}  > C\sqrt{n}\cdot  \frac{s \log d}{n_k}\right) \\
  &\le \exp(-ckn) + d^{-c/2} \rightarrow 0,
\end{align}
and taking $k=o\bigl((s\log d)^{-1}\sqrt{n}\bigr)$ delivers $\|\bDelta\|_{\infty}=o_{\PP}(1)$.
\end{proof}

\begin{proof}[Proof of Theorem \ref{thmWaldDistribution}.]
We verify the requirements of the Lindeberg-Feller central limit theorem \citep[e.g.][Theorem 4.12]{Kallenberg1997}. Write
\[
\overline{V}_{n}:=\sqrt{n}\frac{1}{k}\sum_{j=1}^{k}\frac{Z_{v}\sj}{\hat{Q}\sj} =\sum_{j=1}^{k}\sum_{i\in\mathcal{I}_{j}}\xi_{i v}\sj, \;\; \text{ where } \;\;
\xi_{i v}\sj:=\frac{\mb_{v}^{(j)T}\bX_{i}\sj\epsilon_{i}\sj}{\bigl(n\mb_{v}^{(j)T}\hat{\Sigma}^{(j)}\mb_{v}\sj\bigr)^{1/2}}.
\]
By the fact that $\epsilon_{i}$ is independent of $X$ for all $i$ and $\EE[\epsilon_{i}]=0$,
\begin{eqnarray*}
\EE(\xi_{i v}\sj|X)&=&\EE\left[\mb_{v}^{(j)T}\bX_{i}\sj\epsilon_{i}\sj/\bigl(n\mb_{v}^{(j)T}\hat{\Sigma}^{(j)}\mb_{v}\sj\bigr)^{1/2}|X\right] \\
							  &=& \bigl(n \mb_{v}^{(j)T}\hat{\Sigma}^{(j)}\mb_{v}\sj\bigr)^{-1/2}\mb_{v}^{(j)T}\bX_{i}\sj\EE\bigl(\epsilon_{i}\sj\bigr)=0.
\end{eqnarray*}
By independence of $\{\epsilon_{i}\}_{i=1}^{n}$ and the definition of $\hat{\Sigma}\sj$, we also have
\begin{eqnarray*}
 && \Var\left(\overline{V}_{n}\Bigl|X\right)=\sum_{j=1}^{k}\sum_{i\in\mathcal{I}_{j}}\Var\bigl(\xi_{i v}\sj|X\bigr)\\														    &=&k^{-1}\sum_{j=1}^{k} n_{k}^{-1}\bigl(\mb_{v}^{(j)T}\hat{\Sigma}^{(j)}\mb_{v}\sj\bigr)^{-1}\sum_{i\in\mathcal{I}_{j}}\mb_{v}^{(j)T}\bX_{i}\sj\bX_{i}^{(j)T}\mb_{v}\sj\Var\bigl(\epsilon_{i}\sj|X\bigr)=\sigma^{2}.
\end{eqnarray*}
It only remains to verify the Lindeberg condition, i.e.,
\begin{equation}\label{conditionLindeberg}
\lim_{k\rightarrow \infty} \lim_{n_{k}\rightarrow \infty} \frac{1}{\sigma^{2}}\sum_{j=1}^{k} \sum_{i\in\mathcal{I}_{j}}\EE\left[(\xi_{i v}\sj)^{2}\ind\bigl\{\bigl|\xi_{i v}\sj\bigr|>\epsilon\sigma\bigr\} \bigl| X\right]=0,  \quad \forall \; \epsilon>0.
\end{equation}
By Lemma \ref{lemmaLB}, $\bigl|\xi_{i v}\sj\bigr|\leq n^{-1/2}c_{n_{k}}^{-1}|\mb_{v}^{(j)T}\bX_{i}\sj ||\epsilon_{i}\sj|\leq n^{-1/2}c_{n_{k}}^{-1}\vartheta_{2}|\epsilon_{i}\sj|$, where $\liminf_{n_{k}}c_{n_{k}}=c_{\infty}>0$, hence the event $\bigl\{|\xi_{i v}\sj|>\epsilon\sigma\bigr\}$ is contained in the event $\bigl\{|\epsilon_{i}\sj|>\epsilon\sigma c_{n_{k}}\vartheta_{2}^{-1}\sqrt{n}\bigr\}$ and we have
\begin{eqnarray*}
&  & \frac{1}{\sigma^{2}} \sum_{j=1}^{k}  \sum_{i\in\mathcal{I}_{j}}\EE\left[(\xi_{i v}\sj)^{2}\ind\bigl\{|\xi_{i v}\sj |>\epsilon\sigma\bigr\} \bigl| X\right] \leq  \frac{1}{\sigma^{2}} \sum_{j=1}^{k} \sum_{i\in\mathcal{I}_{j}}\EE\left[(\xi_{i v}\sj)^{2}\ind\bigl\{|\epsilon_{i}\sj |>\epsilon\sigma c_{n_{k}}\vartheta_{2}^{-1}\sqrt{n}\bigr\} \bigl| X\right]\\
&=& \frac{1}{\sigma^{2}} \frac{1}{k}\sum_{j=1}^{k}   \bigl(\mb_{v}^{(j)T}\hat{\Sigma}\mb_{v}\sj\bigr)^{-1} \frac{1}{n_{k}}\sum_{i\in\mathcal{I}_{j}}\mb_{v}^{(j)T}\bX_{i}\sj\bX_{i}^{(j)T}\mb_{v}\sj \EE\Bigl[(\epsilon_{i}\sj)^{2}\ind\bigl\{|\epsilon_{i }\sj |>\epsilon\sigma c_{n_{k}}\vartheta_{2}^{-1}\sqrt{n}\bigr\}\Bigr] \\
&=& \frac{1}{\sigma^{2}} \EE\Bigl[(\epsilon_{i}\sj)^{2}\ind\bigl\{|\epsilon_{i}\sj |>\epsilon\sigma c_{n_{k}}\vartheta_{2}^{-1}\sqrt{n_{k}} \sqrt{k}\bigr\}\Bigr].
\end{eqnarray*}
Let $\delta=\epsilon\sigma c_{n_{k}}\vartheta_{2}^{-1}\sqrt{n}$. Then, for any $\eta>0$,
\begin{equation}\label{eqWaldLindeberg}
 \EE\bigl[\bigl(\epsilon_{i}\sj\bigr)^{2}\ind\bigl\{\bigl|\epsilon_{i}\sj\bigr|>\delta\bigr\}\bigr]  \leq   \EE\Bigl[\bigl(\epsilon_{i}\sj\bigr)^{2}\frac{\bigl|\epsilon_{i}\sj\bigr|^{\eta}}{\delta^{\eta}}\ind\bigl\{\bigl|\epsilon_{i}\sj\bigr|>\delta\bigr\}\Bigr] \leq \delta^{-\eta}\EE\bigl[\bigl|\epsilon_{i}\sj\bigr|^{2+\eta}\bigr].
\end{equation}
Since $\vartheta_{2}n^{-1/2}=o(1)$ by the statement of the theorem, the choice $\eta=2$ delivers
\begin{eqnarray}\label{eq:tail-moment}
\nonumber &  &  \frac{1}{\sigma^{2}} \lim_{k\rightarrow \infty}\lim_{n_{k}\rightarrow \infty} \sum_{j=1}^{k} \sum_{i\in\mathcal{I}_{j}}\EE\left[(\xi_{i v}\sj)^{2}\ind\bigl\{\bigl|\xi_{i v}\sj\bigr|>\epsilon\sigma\bigr\} \bigl| X\right] \\
&\leq & \lim_{k\rightarrow \infty}\lim_{n_{k}\rightarrow \infty} k^{-1} n_{k}^{-1}\vartheta_{2}c_{n_{k}}^{-2}\epsilon^{-2}\sigma^{-2}\EE\left(\bigl(\epsilon_{i}\sj\bigr)^{4}\right) =0
\end{eqnarray}
by the bounded forth moment assumption. By the law of iterated expectations, all conditional results hold in unconditional form as well. Hence, $\overline{V}_{n}\leadsto N(0,\sigma^{2})$ by the Lindeberg-Feller central limit theorem.
\end{proof}

\begin{proof}[Proof of Corollary \ref{corollWaldTestNull}]
Similar to \eqref{eq:tail-moment}, we also have
\[
\frac{1}{\sigma^{3}} \lim_{k\rightarrow \infty}\lim_{n_{k}\rightarrow \infty} \sum_{j=1}^{k} \sum_{i\in\mathcal{I}_{j}}\EE\left[(\xi_{i v}\sj)^{4}\ind\bigl\{\bigl|\xi_{i v}\sj\bigr|>\epsilon\sigma\bigr\} \bigl| X\right] = 0.
\]
The proof is complete through an application of  the self-normalized Berry-Essen inequality \citep{pena2008self}, noting that $\overline{S}_n = \overline{V}_{n} + o_P(1)$, as demonstrated in the previous proof.
 \end{proof}

\begin{proof}[Proof of Lemma \ref{sigmaEstConsistent}]
We first show that, for any $j\in\{1,\ldots,k\}$, $|\hat{\sigma}^{2}(\cD_{j})-\sigma^{2}|=o_{\PP}(k^{-1})$. To this end, letting
\[
\hat{\epsilon}_{i} = Y_{i}\sj - \bX_{i}^{(j)T}\hat{\bbeta}^{\lambda}(\cD_{j}) = Y_{i}\sj - \bX_{i}^{(j)T}\bbeta^{*} - \bX_{i}^{(j)T}\bigl(\hat{\bbeta}^{\lambda}(\cD_{j}) - \bbeta^{*}\bigr),
\]
we write
\[
|\hat{\sigma}^{2}(\cD_{j})-\sigma^{2}|  =  \Bigl|\frac{1}{n_{k}}\sum_{i\in\mathcal{I}_{j}}\hat{\epsilon}_{i}^{2} - \sigma^{2}\Bigr| \leq \Delta_{1}\sj + 2\Delta_{2}\sj + \Delta_{3}\sj,
\]
\begin{eqnarray*}
\Delta_{1}\sj &:=& \bigl|\frac{1}{n_{k}}\sum_{i\in\mathcal{I}_{j}}\epsilon_{i}^{2} - \sigma^{2}\bigr|, \quad
\Delta_{2}\sj := \bigl|\bigl(\hat{\bbeta}^{\lambda}(\cD_{j}) - \bbeta^{*}\bigr)\bigl(\frac{1}{n_{k}}\sum_{i\in\mathcal{I}_{j}}\bX_{i}\sj \epsilon_{i}\sj\bigr)\bigr|, \;\; \text{and} \\
\Delta_{3}\sj &:=& \bigl|\bigl(\hat{\bbeta}^{\lambda}(\cD_{j}) - \bbeta^{*}\bigr)^{T}\Bigl(\frac{1}{n_{k}}\sum_{i\in\mathcal{I}_{j}}\bX_{i}\sj\bX_{i}^{(j)T}\Bigr)\bigl(\hat{\bbeta}^{\lambda}(\cD_{j}) - \bbeta^{*}\bigr)\bigr| \\
                     &=& \bigl\|X\sj\bigl(\hat{\bbeta}^{\lambda}(\cD_{j}) - \bbeta^{*}\bigr)\bigr\|_{2}^{2}/n_{k} =O_{\PP}(\lambda^{2} s)
\end{eqnarray*}
by Theorem 6.1 of \citet{vdGBBook2011}. Hence, with $\lambda=C\sigma^{2}\sqrt{k\log d/n}$, $\Delta_{3}\sj =o_{\PP}(1)$ for $k=o\bigl((s\log d)^{-1}n\bigr)$, a fortiori for $k=o\bigl((s\log d)^{-1}\sqrt{n}\bigr)$. Letting
\begin{eqnarray*}
\Delta_{21}\sj &=& \bigl\|\hat{\bbeta}^{\lambda}(\cD_{j}) - \bbeta^{*}\bigr\|_{1}\Bigl\|\frac{1}{n_{k}}\sum_{i\in\mathcal{I}_{j}}\bX_{i}\sj \epsilon_{i}\sj - \EE[\bX_{i}\sj\epsilon_{i}\sj]\Bigr\|_{\infty}, \\
\Delta_{22}\sj &=& \bigl\|\hat{\bbeta}^{\lambda}(\cD_{j}) - \bbeta^{*}\bigr\|_{1}\bigl\|\EE[\bX_{i}\sj\epsilon_{i}\sj]\bigr\|_{\infty}.
\end{eqnarray*}
We obtain the bound
\[
\Delta_{2}\sj  = \Bigl|\bigl(\hat{\bbeta}^{\lambda}(\cD_{j}) - \bbeta^{*}\bigr)\bigl(\frac{1}{n_{k}}\sum_{i\in\mathcal{I}_{j}}\bX_{i}\sj \epsilon_{i}\sj - \EE[\bX_{i}\sj\epsilon_{i}\sj]\bigr) + \bigl(\hat{\bbeta}^{\lambda}(\cD_{j}) - \bbeta^{*}\bigr)\EE[\bX_{i}\sj\epsilon_{i}\sj]\Bigr| \leq \Delta_{21}\sj + \Delta_{22}\sj.
\]
By the statement of the Lemma, $\EE\bigl[\bX_{i}\sj\epsilon_{i}\sj\bigr] = \EE\bigl[\bX_{i}\sj\EE[\epsilon_{i}\sj|\bX_{i}\sj]\bigr] = 0$, hence $\Delta_{22}\sj=0$, while by the central limit theorem and Theorem 6.1 of \citet{vdGBBook2011},
\[
\Delta_{21}\sj \leq O_{\PP}(\lambda s) O_{\PP}(n_{k}^{-1/2}).
\]
We conclude $\Delta_{2}\sj= O_{\PP}\bigl(\lambda s n_{k}^{-1/2}\bigr)$, and with $\lambda\asymp\sigma^{2}\sqrt{k\log d/n}$, $\Delta_{2}\sj=o(1)$ with $k=o\bigl(n(s\log d)^{-2/3}\bigr)$, a fortiori for $k=o\bigl(\sqrt{n}(s\log d)^{-1}\bigr)$. Finally, noting that $\sigma^{2}=\EE[\epsilon_{i}\sj]$, $\Delta_{1}\sj=O_{\PP}(n_{k}^{-1/2})=o_{\PP}\bigl(1\bigr)$ by the central limit theorem. Combining the  bounds, we obtain $|\hat{\sigma}^{2}(\cD_{j})-\sigma^{2}|=o_{\PP}(1)$ for any $j\in\{1,\ldots,k\}$ and therefore $|\overline{\sigma}^{2}-\sigma^{2}|\leq k^{-1}\sum_{j=1}^{k}|\hat{\sigma}^{2}(\cD_{j})-\sigma^{2}|=o_{\PP}(1)$.
\end{proof}

The proofs of Theorem \ref{thmWaldDistVdG} and Corollary  \ref{corollWaldvdGDist} are stated as an application of Lemmas \ref{lemmaWaldDistVdGWeakerConditions} and \ref{lemmaWaldDistVdG2WeakerConditions}, which apply under a more general set of requirements.

\begin{proof}[Proof of Theorem \ref{thmWaldDistVdG}]

We verify (A1)-(A4) of Lemma \ref{lemmaWaldDistVdGWeakerConditions}. For (A1), decompose the object of interest as
\[
\frac{1}{n_{k}}\|X^{(j)}\widehat{\Theta}^{(j)}\|_{\max} = \frac{1}{n_{k}}\|X^{(j)}\bigl(\widehat{\Theta}^{(j)} - \Theta^{*}\bigr)\|_{\max} + \frac{1}{n_{k}}\|X^{(j)}\Theta^{*}\|_{\max} = \Delta_{1} + \Delta_{2},
\]
where $\Delta_{1}$ can be further decomposed and bounded by
\begin{eqnarray*}
\frac{1}{n_{k}}\bigl\|X^{(j)}\bigl(\widehat{\Theta}^{(j)} - \Theta^{*}\bigr)\bigr\|_{\max}&=&\frac{1}{n_{k}}\max_{1\leq i \leq n}\max_{1\leq v\leq d}\Big[\bigl| \bX_{i}^{(j)T}\bigl(\widehat{\bTheta}_{v}^{(j)} - \bTheta_{v}^{*}\bigr)\bigr|\Bigr] \\
&\leq & \frac{1}{n_{k}}\max_{1\leq i \leq n}\|\bX_{i}\|_{\infty}\max_{1\leq v\leq d}\|\widehat{\bTheta}_{v}^{(j)}-\bTheta_{v}^{*}\|_{1}.
\end{eqnarray*}
We have
\[
\PP(\Delta_{1}>q/2)\leq \PP\Bigl(\max_{1\leq v \leq d} \|\widehat{\bTheta}_{v}^{(j)} - \bTheta_{v}^{*}\|_{1} > \frac{n}{kM}\frac{q}{2}\Bigr)<\psi
\]
and by Condition \ref{con:theta}, $\psi=o(d^{-1})=o(k^{-1})$ for any $q \geq  2CMs_{1} (k/n)^{3/2}\sqrt{\log d}$, a fortiori for $q$ a constant. Since $\bX_{i}$ is sub-Gaussian, a matching probability bound can easily be obtained for $\Delta_{2}$, thus we obtain $\PP\bigl(n_{k}^{-1}\bigl\|X^{(j)}\widehat{\Theta}^{(j)}\bigr\|_{\max}\bigr)\leq 2\psi$ for $\psi=o(k^{-1})$. (A2) and (A3) of  Lemma \ref{lemmaWaldDistVdGWeakerConditions} are applications of Lemmas \ref{lemmaA2} and \ref{lemmaA3} respectively. To establish (A4), observe that $\bigl(\hat{\bTheta}^{(j)T}_{v} \nab^{2}\ell_{n_{k}}^{(j)}\bigl(\hat{\bbeta}^{\lambda}(\cD_{j})\bigr)-\be_{v}\bigr) = \Delta_{1} + \Delta_{2} + \Delta_{3}$, where $\Delta_{1} = \bigl(\hat{\bTheta}^{(j)}_{v} -\bTheta_{v}^{*}\bigr)^{T}\nab^{2}\ell_{n_{k}}^{(j)}\bigl(\hat{\bbeta}^{\lambda}(\cD_{j})\bigr)$, $\Delta_{2} = \bTheta^{*T}_{v}\bigl(\nab^{2}\ell_{n_{k}}^{(j)}\bigl(\hat{\bbeta}^{\lambda}(\cD_{j})\bigr) - \nab^{2}\ell_{n_{k}}^{(j)}(\bbeta^{*})\bigr)$ and $\Delta_{3} = \bTheta^{*T}_{v}\nab^{2}\ell_{n_{k}}^{(j)}(\bbeta^{*})- \be_{v}$. We thus consider $\bigl|\Delta_{\ell}(\widehat{\bbeta}^{\lambda}(\cD_{j}) - \bbeta^{*})\bigr|$ for $\ell=1,2,3$.
\begin{eqnarray*}
\bigl|\Delta_{2}(\widehat{\bbeta}^{\lambda}(\cD_{j}) - \bbeta^{*})\bigr| &=& \Bigl|\frac{1}{n_{k}}\sum_{i\in\mathcal{I}_{j}}\bTheta_{v}^{*T}\bX_{i}\bX_{i}^{T}\bigl(\widehat{\bbeta}^{\lambda}(\cD_{j}) - \bbeta^{*}\bigr)\bigl[b''(\bX_{i}^{T}\widehat{\bbeta}^{\lambda}\bigl(\cD_{j})\bigr)-b''(\bX_{i}^{T}\bbeta^{*})\bigr] \Bigr| \\
&\leq & U_{3}\max_{1\leq i \leq n}\bigl|\bTheta_{v}^{*T}\bX_{i}\bigr|\frac{1}{n_{k}} \bigl\|X\bigl(\widehat{\bbeta}^{\lambda}(\cD_{j})-\bbeta^{*}\bigr)\bigr\|_{2}^{2}
\end{eqnarray*}
$\PP\Bigl(\bigl\|X\bigl(\widehat{\bbeta}^{\lambda}(\cD_{j})-\bbeta^{*}\bigr)\bigr\|_{2}^{2}\gtrsim n^{-1}sk\log(d/\delta)\Bigr)<\delta$ by Lemma \ref{lemmaA3}, thus $\PP\Bigl(\bigl|\Delta_{2}(\widehat{\bbeta}^{\lambda}(\cD_{j}) - \bbeta^{*})\bigr|>t\Bigr)<\delta$ for $t\asymp M U_{3}n^{-1}s k\log(d/\delta)$. Invoking H\"older's inequality, Hoeffding's inequality and Condition \ref{con:estimation}, we also obtain, for $t\asymp n^{-1}s k\log(d/\delta)$,
\[
\PP\Bigl(\bigl|\Delta_{3}(\widehat{\bbeta}^{\lambda}(\cD_{j}) - \bbeta^{*})\bigr|>t\Bigr)\leq \PP\Bigl(\Bigl\|\bTheta_{v}^{*T}\Bigl(\frac{1}{n_{k}}\sum_{i\in\mathcal{I}_{j}}b''(\bX_{i}^{T}\bbeta^{*})\bX_{i}\bX_{i}^{T}\Bigr) - \be_{v}\Bigr\|_{\max}\bigl\|\widehat{\bbeta}^{\lambda}(\cD_{j}) - \bbeta^{*}\bigr\|_{1}>t\Bigr).
\]
Therefore $\PP\Bigl(\bigl|\Delta_{2}(\widehat{\bbeta}^{\lambda}(\cD_{j}) - \bbeta^{*})\bigr|>t\Bigr)<2\delta$. Finally, with $t\asymp n^{-1}(s\vee s_{1})k\log(d/\delta)$,
\[
\PP\Bigl(\bigl|\Delta_{1}(\widehat{\bbeta}^{\lambda}(\cD_{j}) - \bbeta^{*})\bigr|>t\Bigr)\leq \PP\Bigl(\Bigl\|\frac{1}{n_{k}}\sum_{i\in\mathcal{I}_{j}}X_{i}^{T}\bigl(\widehat{\bTheta}_{v} - \bTheta_{v}\bigr)b''(\bX_{i}^{T}\hat{\bbeta}^{\lambda}(\cD_{j}))\Bigr\|_{2} \Bigl\|\frac{1}{n_{k}}X^{(j)}\bigl(\widehat{\bbeta}^{\lambda}(\cD_{j}) - \bbeta^{*}\bigr)\Bigr\|_{2}>t\Bigr),
\]
hence $\PP\Bigl(\bigl|\Delta_{1}(\widehat{\bbeta}^{\lambda}(\cD_{j}) - \bbeta^{*})\bigr|>t\Bigr)<2\delta$. This follows because $\PP\bigl(\Bigl\|\frac{1}{n_{k}}X^{(j)}\bigl(\widehat{\bbeta}^{\lambda}(\cD_{j}) - \bbeta^{*}\bigr)\Bigr\|_{2}\gtrsim n^{-1/2}\sqrt{sk\log(d/\delta}\bigr)<\delta$ by Lemma \ref{lemmaA3} and
\[
\PP\Bigl(\Bigl\|\frac{1}{n_{k}}\sum_{i\in\mathcal{I}_{j}}X_{i}^{T}\bigl(\widehat{\bTheta}_{v} - \bTheta_{v}\bigr)b''(\bX_{i}^{T}\hat{\bbeta}^{\lambda}(\cD_{j}))\Bigr\|_{2}\gtrsim n^{-1/2}\sqrt{s_{1}k\log(d/\delta}\Bigr)<\delta
\]
by Lemma C.4 of \citet{NingLiu2014b}.
\end{proof}

\begin{proof}[Proof of Corollary \ref{corollWaldvdGDist}]
We verify (A5)-(A9) of Lemma \ref{lemmaWaldDistVdG2WeakerConditions}. (A5) is satisfied because $\widetilde{\Theta}_{vv}$ is consistent under the required scaling by the statement of the corollary. (A6) is satisfied by Condition \ref{con:theta}. To verify (A7), first note that $\nabla \ell_i(\bbeta^*)=(b'(\bX_i^T\bbeta^*)-Y_i)\bX_i$. According to Lemma \ref{lem:mgf}, we know that conditional on $X$, $b'(\bX_i^T\bbeta^*)-Y_i$ is a sub-gaussian random variable. Therefore Lemma \ref{lem:hineq} delivers
\[
	\PP\left(\|\frac{1}{n}\sum\limits_{j=1}^k \sum\limits_{i\in \mathcal{I}_j} \nabla \ell_i(\bbeta^*)\|_\infty>t \given X \right)\le d\exp\left(1-\frac{ct^2}{nM^2}\right),
\]
which implies that with probability $1-c/d$,
\begin{equation}\label{eq:max-grad-log}
 \|\sum\limits_{j=1}^k \sum\limits_{i\in \mathcal{I}_j} \nabla \ell_i(\bbeta^*)\|_\infty=C\sqrt{n\log d}
\end{equation}

It only remains to verify (A8). Let $\xi^{(j)}_{iv}=\bTheta^{*T}_{v}\nab \ell_{i}^{(j)}(\bbeta^{*})/\sqrt{n\Theta^{*}_{vv}}$. By the definition of the log likelihood,
\[
\EE[\xi_{iv}^{(j)}]= \frac{\bTheta^{*T}_{v}\EE[\nab\ell_{i}^{(j)}(\bbeta^{*})]}{(n\Theta^{*}_{vv})^{1/2}}=0
\]
and by independence of $\{(Y_{i},\bX_{i})\}_{i=1}^{n}$,
\begin{eqnarray*}
& &\Var\bigl(\sum_{j=1}^{k}\sum_{i\in\mathcal{I}_{j}}\xi_{iv}^{(j)}\bigr) = \sum_{j=1}^{k} \sum_{i\in\mathcal{I}_{j}} \Var\bigl(\xi_{iv}^{(j)}\bigr) = \sum_{j=1}^{k} \sum_{i\in\mathcal{I}_{j}} \EE[(\xi_{iv}^{(j)})^{2}]\\
 &=& \frac{1}{n}\sum_{i=1}^{n}(\Theta^{*}_{vv})^{-1}\bTheta^{*T}_{v}\EE\bigl[\bigl(\nab \ell_{i}(\bbeta^{*})\bigr)\bigl(\nab \ell_{i}(\bbeta^{*})\bigr)^{T}] \bTheta^{*}_{v} = \frac{1}{n}\sum_{i=1}^{n}(\Theta^{*}_{vv})^{-1} [\Theta^{*} J^{*} \Theta^{*}]_{vv} = 1.
\end{eqnarray*}
By Condition \ref{con:glm}, $\theta_{\min}>0$, the event $\{|\xi_{iv}^{(j)}|>\epsilon\}$ coincides with the event $\bigl\{\bigl|\bTheta_{v}^{*T}\nab\ell_{i}(\bbeta^{*})\bigr| > \epsilon \sqrt{\theta_{\min}n}\bigr\}=\bigl\{\bigl|\bTheta_{v}^{*T}\bX_{i}(Y_i-b'(\bX_i^T\bbeta^*))\bigr| > \epsilon \sqrt{\theta_{\min}n}\bigr\}$. Furthermore, since $\bigl|\bTheta_{v}^{*T}\bX_{i}\bigr|\leq M$ by Condition \ref{con:theta}, %$s_{1}\lesssim (ks\log d)^{-1}\sqrt{n}$
this event is contained in the event $\bigl\{\bigl|Y_i-b'(\bX_i^T\bbeta^*)\bigr| > \delta \bigr\}$, where $\delta=\epsilon\sqrt{\theta_{\min}n}/M$. By an analogous calculation to that of equation \eqref{eqWaldLindeberg}, we have
\[
\EE\Bigl[ \bigl(Y_i-b'(\bX_i^T\bbeta^*)\bigr)^{2}\ind\{|Y_i-b'(\bX_i^T\bbeta^*)|>\delta\}|X\Bigr] \leq \delta^{-\eta}\EE\bigl[\bigl(Y_i-b'(\bX_i^T\bbeta^*)\bigr)^{2+\eta}|X\bigr].
\]
Hence, setting $\eta=2$ and noting that $\EE\bigl[(Y_i-b'(\bX_i^T\bbeta^*))^{2+\eta}|X\bigr]\le C\sqrt{2+\eta}\phi U_2$ by Lemma \ref{lem:mgf}, it follows that
\begin{eqnarray}\label{vdGLindeberg}
\nonumber & & \lim_{k\rightarrow \infty}\lim_{n_{k}\rightarrow \infty}\sum_{j=1}^{k} \sum_{i\in \mathcal{I}_{j}}\EE\bigl[(\xi_{i,v}\sj)^{2}\ind\{|\xi_{i,v}\sj|>\epsilon\}\bigr] \\
\nonumber &\leq & (\theta_{\min})^{-1} \lim_{k\rightarrow \infty} \lim_{n_{k}\rightarrow \infty} n^{-1} \sum_{j=1}^{k} \sum_{i\in \mathcal{I}_{j}} \bTheta_{v}^{*T} \EE[\bX_{i}\bX_{i}^{T}]\bTheta^*_{v}\delta^{-2}\\
& \le & (\theta_{\min})^{-1} \lim_{k\rightarrow \infty} \lim_{n_{k}\rightarrow \infty} M^3s_1^2/(n\epsilon^2\theta_{\min}) = 0,
\end{eqnarray}
where the last inequality follows because $\|\Sigma\|_{\max}=\|\EE[\bX_{i}\bX_{i}^{T}]\|_{\max}<M^2$ by Condition \ref{con:glm}. Similarly, we have for any $\epsilon >0$,
\[
\epsilon^{-3} \lim_{k\rightarrow \infty}\lim_{n_{k}\rightarrow \infty}\sum_{j=1}^{k} \sum_{i\in \mathcal{I}_{j}}\EE\bigl[(\xi_{i,v}\sj)^{3}\ind\{|\xi_{i,v}\sj|>\epsilon\}\bigr] = 0 .
\]
Applying the self-normalized Berry-Essen inequality, we complete the proof of this corollary.
\end{proof}

\subsection{Proofs for Section \ref{sectionScore}}\label{sectionProofScore}

The proof of Theorem  \ref{thmScoreDistn} relies on several preliminary lemmas, collected in the Supplementary Material. Without loss of generality we set $H_{0}:\beta_{v}^{*}=0$ to ease notation.

\begin{proof}[Proof of Theorem \ref{thmScoreDistn}]
Since $\overline{S}(0) = k^{-1} \sum_{j = 1}^k \hat{S}\sj(0,\widehat{\bbeta}_{-v}^{\lambda}(\cD_{j}))$, and (B1)-(B4) of Condition \ref{conditionScoreLR} are fulfilled under Conditions \ref{con:glm} and \ref{con:estimation} by Lemma \ref{lemmaCheckB1toB4} (see Appendix \ref{appendixAuxLemmata}). The proof is now simply an application of Lemma \ref{lemmaDistributionAggScore} with $\beta_{v}^{*}=0$ under the restriction of the null hypothesis.
\end{proof}

\begin{proof}[Proof of Lemma \ref{lemmaCI}]
The proof is an application of Lemma \ref{lemmaPartialInfoEstimator}, noting that (B1)-(B5) of Condition \ref{conditionScoreLR} are fulfilled under Conditions \ref{con:glm} and \ref{con:estimation} by Lemmas \ref{lemmaCheckB1toB4} and \ref{lemmaCheckB5B6}.
\end{proof}

 \subsection{Proofs for Section \ref{sectionEstimation}}

Recall from Section \ref{sectionBackground} that for an arbitrary matrix $M$, $\bM_{\ell}$ denotes the transposed $\ell^{th}$ row of $M$ and $[\bM]_{\ell}$ denotes the $\ell^{th}$ column of $M$.

 \begin{proof}[Proof of Lemma \ref{lm:est-inf-rate}]
 According to Theorem \ref{thmWaldExpression}, we have  $\sqrt{n}(\overline{\bbeta}^{d}-\bbeta^{*})=\bZ+\bDelta$, where $\bZ=\frac{1}{\sqrt{k}}\sum_{j=1}^{k} \frac{1}{\sqrt{n_{k}}}M^{(j)}X^{(j)T} \bepsilon^{(j)}$.
 In \eqref{eqDeltaBound}, we prove that $\|\bDelta\|_{\infty}/\sqrt{n} \le C sk\log d/n$ with probability larger than $1- \exp(-ckn) - d^{-c/2} \ge 1 - c_1/d$ for some constant $c_1$. Since $\hat \bbeta^d$ is a special case of $\overline \bbeta^d$ when $k=1$, we also have $\sqrt{n}(\hat{\bbeta}^{d}-\bbeta^{*})=\bZ+\bDelta_1$, where  \eqref{eqDeltaBound} gives $\|\bDelta\|_{\infty}/\sqrt{n} \le C s\log d/n$. Therefore, we have $\|\overline{\bbeta}^{d} - \hat{\bbeta}^{d}\|_{\infty} \le C sk\log d/n$ with high probability.

 It only remains to bound the rate of $\|\bZ\|_{\infty}/ \sqrt{n}$. By Condition \ref{con:subg}, conditioning on $\{\bX_i\}_{i=1}^n$, we have for any $\ell = 1, \ldots, d$,
\begin{equation}\label{eq:Zl-rate}
     \PP\Big(|Z_\ell |/\sqrt{n} >t \,\Big|\, \{\bX_i\}_{i=1}^n\Big)  = \PP\Big(\Big|\frac{1}{n}\sum_{j=1}^{k}\bM_{\ell}^{(j)T}X^{(j)T}\bepsilon^{(j)}\Big| > t\,\Big|\, \{\bX_i\}_{i=1}^n\Big) \le 2 \exp\Big(-\frac{cnt^2}{\kappa^2Q_{\ell}}\Big),
\end{equation}
 where $\kappa$ is the variance proxy of $\epsilon$ defined in Condition \ref{con:subg} and
 \[
 Q_{\ell} = \frac{1}{n} \sum_{j=1}^k \|X^{(j)}\bM_{\ell}^{(j)T}\|_2^2.
 \]
Let $Q_{\max} = \max_{1\le \ell  \le d} Q_{\ell}$. Applying the union bound to \eqref{eq:Zl-rate}, we have
\begin{align*}
   \PP\Big(\|\bZ\|_{\infty}/\sqrt{n} > t \,\Big|\, \{\bX_i\}_{i=1}^n\Big)
   & \le \PP\Big(\max_{1 \le \ell \le d}|Z_\ell|/\sqrt{n} >t \,\Big|\, \{\bX_i\}_{i=1}^n\Big) \\
   &\le \sum_{\ell=1}^d \PP\Big(|Z_\ell|/\sqrt{n} >t \,\Big|\, \{\bX_i\}_{i=1}^n\Big)    \le 2d \exp\Big(-\frac{cnt^2}{\kappa^2Q_{\max}}\Big).
\end{align*}
Let $t=\sqrt{ 2\kappa^2 Q_{\max}\log d /(cn)}$, then with conditional probability $1-2/d$,
\begin{equation}\label{eq:condition-rate}
 \|\bZ\|_{\infty}/\sqrt{n} \le \sqrt{ \kappa^2 Q_{\max}\log d /(cn)}.
\end{equation}
The last step is to bound $Q_{\max}$. By the definition of $Q_{\ell}$, we have
\begin{equation}
   Q_{\ell} =  \frac{1}{k} \sum_{j=1}^k \bM_{\ell}^{(j)T}\hat{\Sigma}^{(j)}\bM_{\ell}^{(j)} \le  \frac{1}{k} \sum_{j=1}^k [\bOmega]_{\ell}^T\hat{\Sigma}^{(j)}[\bOmega]_{\ell} = \frac{1}{k} \sum_{j=1}^k \frac{1}{n_k}\sum_{i \in \cD_j}(\bX_i^{T}[\bOmega]_{\ell})^2 = \frac{1}{n} \sum_{i=1}^n (\bX_i^{T}[\bOmega]_{\ell})^2,
\end{equation}
where $\Omega=\Sigma^{-1}$. The inequality is due to the fact that $M_{\ell}^{(j)}$ is the minimizer in \eqref{eq:m-est}. By condition \eqref{con:subg} and the connection between subgaussian and subexponential distributions, the random variable $(\bX_i^{T}\bOmega_{\ell})^2$ satisfies
\[
\sup_{q\geq 1}q^{-1}\bigl(\EE|(\bX_i^{T}\bOmega_{\ell})^2|^{q}\bigr)^{1/q} \le 4 \kappa^2 \Omega_{\ell\ell}.
\]
Therefore, by Bernstein's inequality for subexponential random variables, we have
\[
  \PP\Big( \Big|\frac{1}{n} \sum_{i=1}^n (\bX_i^{T}[\bOmega]_{\ell})^2 - \EE[\bX_1^{T}[\bOmega]_{\ell}]^2 \Big|> t\Big) \le 2\exp\bigg(-c\Big(\frac{nt^2}{16\kappa^4 \Omega^2_{\ell\ell}}\Big) \bigwedge \Big(\frac{nt}{4\kappa^2 \Omega_{\ell\ell}}\Big) \bigg).
\]
Applying the union bound again, we have
\begin{align*}
 \lefteqn{\PP\Big(\max_{1\le \ell \le d} \Big|\frac{1}{n} \sum_{i=1}^n (\bX_i^{T}[\bOmega]_{\ell})^2 - \EE[\bX_1^{T}[\bOmega]_{\ell}]^2 \Big|> 8 \kappa^2 \Omega_{\ell\ell}\sqrt{\frac{\log d}{cn}}\Big)} \\
 &\le \sum_{j=1}^d\PP\Big( \Big|\frac{1}{n} \sum_{i=1}^n (\bX_i^{T}[\bOmega]_{\ell})^2 - \EE[\bX_1^{T}[\bOmega]_{\ell}]^2 \Big|> 8 \kappa^2 \Omega_{\ell\ell}\sqrt{\frac{\log d}{cn}}\Big) \le 2/d.
\end{align*}
Therefore, with probability $1-2/d$, there exist a constant $C_1$ such that
\[
  Q_{\max} = \max_{1\le \ell \le d} Q_{\ell}  \le \max_{1\le \ell \le d} \Big|\frac{1}{n} \sum_{i=1}^n (\bX_i^{T}\bOmega_{\ell})^2 - \EE[\bX_1^{T}\bOmega_{\ell}]^2 \Big| + \EE[\bX_1^{T}\bOmega_{\ell}]^2 \le 8 \kappa^2 \bOmega_{jj}\sqrt{\frac{\log d}{cn}} + \Omega_{jj} \le C_1,
\]
where the last inequality is due to Condition \ref{con:sigma}. By  \eqref{eq:condition-rate}, we have with probability $1-4/d$, $ \|\bZ\|_{\infty}/\sqrt{n} \le \sqrt{ \kappa^2 C_1\log d /(cn)}$. Combining this with the result on $\|\bDelta\|_{\infty}$ delivers the rate in the lemma.
\end{proof}

\begin{proof}[Proof of Theorem \ref{thm:est-l2-rate}]
 By Lemma \ref{lm:est-inf-rate} and  $k = O(\sqrt{n/(s^2 \log d)})$, there exists a sufficiently large $C_0$ such that for the event
$
 \cE:= \{\big\|\overline{\bbeta}^{d} - \bbeta^*\big\|_{\infty} \le C_0 \sqrt{{\log d}/{n}}\},
$
we have $\PP(\cE)\ge 1-c/d$. We choose $\nu = C_0 \sqrt{{\log d}/{n}}$, which implies that, under $\cE$, we have $\nu \ge  \big\|\overline{\bbeta}^{d} - \bbeta^*\big\|_{\infty}$.

Let $\cS$ be the support of $\bbeta^*$. The derivations in the remainder of the proof hold on the event $\cE$. Observe $\cT_{\nu}(\overline{\bbeta}_{\cS^c}^{d}) = {\bf 0}$ as $\|\overline{\bbeta}_{\cS^c}^{d}\|_{\infty} \le \nu$. For $j \in \cS$, if $|\bbeta_j^*| \ge 2\nu$, we have $|\overline{\bbeta}_{j}^{d}| \ge |\beta_j^*| - \nu \ge \nu$ and thus $|\cT_{\nu}(\overline{\beta}_{j}^{d}) -  \beta_j^*|= |\overline{\beta}_{j}^{d}-  \beta_j^*| \le \nu$.
While if $|\beta^{*}_j| < 2\nu$, $|\cT_{\nu}(\overline{\beta}_{j}^{d}) -  \beta_j^*| \le |\beta_j^*| \vee |\overline{\beta}_{j}^{d}-  \beta_j^*| \le 2\nu$. Therefore, on the event $\cE$,
\[
    \big\|\cT_{\nu}(\overline{\bbeta}^{d}) - \bbeta^*\big\|_{2}  = \big\|\cT_{\nu}(\overline{\bbeta}_{\cS}^{d}) - \bbeta_{\cS}^*\big\|_{2} \le 2\sqrt{s}\nu \text{ and } \big\|\cT_{\nu}(\overline{\bbeta}^{d}) - \bbeta^*\big\|_{\infty}  = \big\|\cT_{\nu}(\overline{\bbeta}_{\cS}^{d}) - \bbeta_{\cS}^*\big\|_{\infty} \le 2\nu.
\]
The statement of the theorem follows because $\nu = C_0 \sqrt{{\log d}/{n}}$ and $\PP(\cE)\ge 1-c/d$.
Following the same reasoning, on the event  $\cE' :=\cE \cup \{\big\|\hat{\bbeta}^{d} - \bbeta^*\big\|_{\infty} \le C_0 \sqrt{{\log d}/{n}}\} \cup \{\big\|\hat{\bbeta}^{d} - \overline{\bbeta}^d\big\|_{\infty} \le C_0 sk{{\log d}/{n}}\}$, we have
\[
  \big\|\cT_{\nu}(\overline{\bbeta}^{d}) - \cT_{\nu}(\overline{\bbeta}^{d})\big\|_{2}  =  \big\|\cT_{\nu}(\overline{\bbeta}_{\cS}^{d}) - \cT_{\nu}(\hat{\bbeta}_{\cS}^{d})\big\|_{2} \le  \big\| \overline{\bbeta}_{\cS}^{d}- \hat{\bbeta}_{\cS}^{d} \big\|_{2}\le  \sqrt{s}\big\| \overline{\bbeta}_{\cS}^{d}- \hat{\bbeta}_{\cS}^{d} \big\|_{\infty} \le Cs^{3/2} k \log d /n.
\]
As Lemma \ref{lm:est-inf-rate} also gives $\PP(\cE') \ge 1- c/d$, the proof is complete.
\end{proof}

\begin{proof}[Proof of Lemma \ref{lm:est-inf-rate-glm}]
The strategy of proving this lemma is similar to the proof of Lemma \ref{lm:est-inf-rate}. In the proof of Lemma \ref{lemmaWaldDistVdGWeakerConditions} and Theorem \ref{thmWaldDistVdG}, we have shown that
 \[
   (\overline{\bbeta}^{d} - \bbeta^*) = \underbrace{-\frac{1}{k} \sum_{j=1}^{k} \hat{\Theta}^{\sj T}\nab\ell_{n_{k}}\sj (\bbeta^{*})}_{\Tb} + \frac{1}{k}\sum_{j=1}^k\bDelta_j,
 \]
 where the remainder term for each $j$ is
 \[
   \bDelta_j = \bigg(I - \hat \Theta^{(j)T} \frac{1}{n_k}\sum_{i\in \cI_j} b''(\tilde{\eta}_i)\bX_i \bX_i^T\bigg)(\hat \bbeta^{\lambda}(\cD_j) - \bbeta^*)
 \]
 and $\tilde{\eta}_i = t\bX_i^T \bbeta^* + (1-t)\bX_i^T \hat \bbeta^{\lambda}(\cD_j)$ for some $t \in (0,1)$. We bound $\bDelta_j$  by decomposing it into three terms:
\begin{align*}
 \|\bDelta_j\|_{\infty} &\le  \underbrace{\Big\|\big(I -   \Theta^* \frac{1}{n_k}\sum_{i\in \cI_j} b''(\bX_i^T \bbeta^*)\bX_i \bX_i^T\big)(\hat \bbeta^{\lambda}(\cD_j) - \bbeta^*) \Big\|_{\infty} }_{I_1} \\
 & + \underbrace{\Big\| \Theta^* \frac{1}{n_k}\sum_{i\in \cI_j} (b''(\bX_i^T\hat \bbeta^{\lambda}(\cD_j) ) - b''(\bX_i^T \bbeta^*))\bX_i \bX_i^T\bigg)(\hat \bbeta^{\lambda}(\cD_j) - \bbeta^*)\Big\|_{\infty} }_{I_2}\\
 & +\underbrace{\Big\| ( \hat \Theta^{(j)} - \Theta^*)^T \frac{1}{n_k}\sum_{i\in \cI_j} b''(\bX_i^T\hat \bbeta^{\lambda}(\cD_j) )  \bX_i \bX_i^T\bigg)(\hat \bbeta^{\lambda}(\cD_j) - \bbeta^*)\Big\|_{\infty} }_{I_3}.
\end{align*}
By Hoeffding's inequality and Condition \ref{thmWaldDistribution}, the first term is bounded by
\begin{equation}\label{eq:i1-est}
 |I_1| \le \Big\| I -   \Theta^* \frac{1}{n_k}\sum_{i\in \cI_j} b''(\bX_i^T \bbeta^*)\bX_i \bX_i^T\Big\|_{\max}\Big\|\hat \bbeta^{\lambda}(\cD_j) - \bbeta^* \Big\|_1 \le C\frac{sk\log d}{n},
\end{equation}
with probability $1-c/d$.  By Condition \ref{con:glm} (iii), Condition \ref{con:theta}  (iv) and Lemma \ref{lemmaA3}, we have with probability $1-c/d$,
\begin{align}\label{eq:i2-est}
 |I_2| \le\max_i\big\| \bTheta^* \bX_i\big\|_{\infty}\frac{1}{n_k}\sum_{i\in \cI_j}  U_3[\bX_i(\hat \bbeta^{\lambda}(\cD_j) - \bbeta^*)]^2 \le  C\frac{sk\log d}{n}.
\end{align}
 Finally, we bound $I_3$ by with probability $1-c/d$,
 \begin{eqnarray}\label{eq:i3-est}
\nonumber |I_3| &\le & \bigg(U_2 \frac{1}{n_k}\sum_{i\in \cI_j}  b''(\bX_i^T\hat \bbeta^{\lambda}(\cD_j) )  [\bX_i^T ( \hat \Theta^{(j)} - \Theta^*) ]^2 \bigg)^{1/2}\bigg(\frac{1}{n_k}\sum_{i\in \cI_j}   [\bX_i(\hat \bbeta^{\lambda}(\cD_j) - \bbeta^*)]^2\bigg)^{1/2} \\
& \le &  C\frac{(s_1 \vee s )k\log d}{n},
\end{eqnarray}
where the last inequality is due to Lemma \ref{lemmaA3} and Lemma C.4 of \citet{NingLiu2014b}.

Combining \eqref{eq:i1-est} - \eqref{eq:i3-est} and applying the union bound, we have
 \[
 \Big\|\frac{1}{k}\sum_{j=1}^k\bDelta_j \Big\|_{\infty} \le \max_{j}\|\bDelta_j\|_{\infty} =O_P\Big(  \frac{(s_1 \vee s )k\log d}{n}\Big).
 \]
  Therefore, we only need to bound  the infinity norm of  the leading term $\Tb$.  By Condition \ref{con:theta} and equation \eqref{eq:max-grad-log}, we have with probability $1-c/d$,
\begin{equation}
 \max_{1 \le j \le k}\max_{1\le v \le d} \|\hat\bTheta\sj_v-\bTheta^*_v\|_1 \le C s_1 \sqrt{\log d/n} \text{ and }
  \Big\|\frac{1}{k}\sum_{j=1}^k\nab\ell_{n_{k}}\sj (\bbeta^{*})\Big\|_\infty \le C\sqrt{\log d/n}.
\end{equation}
This, together with Condition \ref{con:glm} and Condition \ref{con:theta} give the bound,
\[
 \|\Tb\|_{\infty} \le   \Big(M\max_{v,j} \|\hat\bTheta\sj_v-\bTheta^*_v\|_1   +  \max_{i} \|\bX_i^T\bTheta^*\|_{\infty}  \Big)  \Big\|\frac{1}{k}\sum_{j=1}^k\nab\ell_{n_{k}}\sj (\bbeta^{*})\Big\|_\infty  \le  C\Big( \sqrt{\frac{\log d}{n}} +   {\frac{s_1\log d}{n}}\Big),
 \]
 with probability $1-c/d$. Since $\hat \bbeta^d$ is a special case of $\overline \bbeta^d$ when $k=1$, the proof of the lemma is complete.
\end{proof}

{\begin{proof}[Proof of Corollary \ref{corollVarianceEstimator}]
By an analogous proof strategy to that of Theorem \ref{thm:est-l2-rate-glm}, $\bigl|[\cT_{\zeta}(\overline{\Theta})]_{vv} - \Theta^{*}_{vv}\bigr|=O_{p}\bigl(\sqrt{n^{-1}\log d}\bigr)=o_{\PP}(1)$ under the conditions of the Corollary provided $k=o\bigl(((s\vee s_{1})\log d)^{-1}\sqrt{n}\bigr)$.
\end{proof}

\begin{proof}	
[Proof of Theorem \ref{thm:ols2}]
		\begin{equation}
			\label{eq:decompose}
			\begin{aligned}
				\overline\bbeta-\hat\bbeta&=\frac{1}{k}\sum\limits_{j=1}^k(({X}^{(j)})^T{X}^{(j)})^{-1}({X}^{(j)})^T{\bY}^{(j)}-({X}^T{X})^{-1}{X}^T{\bY} \\
				&=\frac{1}{k}\sum\limits_{j=1}^k\left(\left({X^{(j)T}}{X}^{(j)}/n_k\right)^{-1}-({X^TX}/n)^{-1}\right){X^{(j)T}}\bepsilon^{(j)}/n_k\\
				&=\frac{1}{k}\sum\limits_{j=1}^k\left(\left({X^{(j)T}}{X}^{(j)}/n_k\right)^{-1}-{\Sigma}^{-1}\right){X^{(j)T}}\bepsilon^{(j)}/n_k+\left({\Sigma}^{-1}-({X^TX}/n)^{-1}\right){X}^T{\bepsilon}/n.
			\end{aligned}
		\end{equation}
For simplicity, denote ${X^{(j)T}}{X}^{(j)}/n_{k}$ by $S^{(j)}_{X}$, ${X^TX}/n$ by ${ S_X}$, $(S^{(j)}_{X})^{-1}-({\Sigma})^{-1}$ by ${D}_1^{(j)}$ and $({\Sigma})^{-1}-{S_{X}}^{-1}$ by ${D}_2$. For any $\tau\in\RR$, define an event $\cE^{(j)}=\{\|(S^{(j)}_{X})^{-1}\|_2\leq2/C_{\min}\}\cap\{\|S^{(j)}_{X}-{\Sigma}\|_2\leq (\delta_1 \vee \delta_1^2)\}$ for all $j=1,\ldots,k$, where $\delta_1=C_1\sqrt{d/n_k}+\tau/\sqrt{n_k}$, and an event $\cE=\{\|({S_{X}})^{-1}\|_2\leq2/C_{\min}\}\cap\{\|{S_{X}}-{\Sigma}\|_2<(\delta_2 \vee \delta_2^2)\}$, where $\delta_2=C_1\sqrt{d/n}+\tau/\sqrt n$. Note that by Lemma \ref{lem:cctts} and \ref{lem:mineg}, the probability of both $(\cE^{(j)})^{c}$ and $\cE^c$ are very small. In particular
		\[
			\PP(\cE^c)\le \exp(-cn)+\exp(-c_1\tau^2)\ \text{and}\ \PP((\cE^{(j)})^c)\le \exp(-cn/k)+\exp(-c_1\tau^2).
		\]
		Then, letting $\cE_0:=\bigcap\limits_{j=1}^k \cE^{(j)}$, an application of the union bound and Lemma \ref{lem:olsmgf} delivers
	\[
		\begin{aligned}
			\PP\left(\|\overline\bbeta-\hat\bbeta\|_2>t\right) & \le \PP\left(\left\{\Bigl\|\frac{1}{k}\sum\limits_{j=1}^k({X}^{(j)}{D}_1^{(j)})^T\bepsilon^{(j)}/n_k\Bigr\|_2>t/2\right\}\cap \cE_0\right)\\
			&+\PP\left(\left\{\|({X}{D}_2)^T\bepsilon/n\|_2>t/2\right\}\cap \cE\right)+\PP(\cE_0^c)+\PP(\cE^c)\\
			&\le 2\exp\left(d\log(6)-\frac{t^2C_{\min}^3n}{32C_3\sigma_1^2 \delta_1^2}\right)+k\exp(-cn/k)+(k+1)\exp(-c_1\tau^2).	
		\end{aligned}
	\]
When $d\to\infty$ and $\log n=o(d)$, choose $\tau=\sqrt{d/c_1}$ and $\delta_1=O(\sqrt{kd/n})$. Then there exists a constant $C$ such that
	\[
		\PP\left(\|\overline\bbeta-\hat\bbeta\|_2>C\frac{\sqrt{k}d}{n}\right)\le (k+3)\exp(-d)+k\exp(-\frac{cn}{k}).
	\]
	Otherwise choose $\tau=\sqrt{\log n/c_1}$ and $\delta_1=O(\sqrt{k\log n/n})$. Then there exists a constant $C$ such that
	\[
		\PP\left(\|\overline\bbeta-\hat\bbeta\|_2>C\frac{\sqrt{k}\log n}{n}\right)\le \frac{k+3}{n}+k\exp(-\frac{cn}{k}).
	\]
	Overall, we have
	\[
		\PP\left(\|\overline\bbeta-\hat\bbeta\|_2>C\frac{\sqrt{k}(d \vee \log n)}{n}\right)\le ck\exp(-(d \vee \log n))+k\exp(-{cn}/{k}),
	\]	
	which leads to the final conclusion.
	\end{proof}
	
	\begin{proof}[Proof of Corollary \ref{cor:refitlm}]
Define an event $\cE=\{\|\overline\bbeta^d-\bbeta^*\|_\infty\le 2C\sqrt{\log d/n}\}$, then by the condition on the minimal signal strength and Lemma \ref{lm:est-inf-rate}, for some constant $C'$ we have
	\[
		\begin{aligned}
			\PP\left(\|\overline\bbeta^r-\hat\bbeta^o\|_2>C'\frac{\sqrt{k}(s \vee \log n)}{n}\right) &\le\PP\left(\left\{\|\overline\bbeta^r-\hat\bbeta^o\|_2>C'\frac{\sqrt{k}(s \vee \log n)}{n}\right\}\cap\cE\right)+\PP(\cE^c)\\
			&\le \PP\left(\left\{\|\overline\bbeta^o-\hat\bbeta^o\|_2>C'\frac{\sqrt{k}(s \vee \log n)}{n}\right\}\cap\cE\right)+c/d \\
			&\le ck\exp(-(s \vee \log n))+k\exp(-{cn}/{k})+c/d.
		\end{aligned}
	\]	
	where $\overline\bbeta^o=\frac{1}{k}\sum\limits_{j=1}^k (X_{S}^{(j)T}X_{S}^{(j)})^{-1}X_{S}^{(j)T}\bY^{(j)}$, which is the average of the oracle estimators on the subsamples. Then the conclusion can be easily validated.
\end{proof}

\begin{proof}
	[Proof of Theorem \ref{thm:glm2}] The following notation is used throughout the proof.
	\[
			\begin{aligned}
				S(\bbeta)&:=\nabla^{2}\ell_{n}(\bbeta)=\frac{1}{n} {X^T}D(X\bbeta)X, \quad{S}^{(j)}(\bbeta):=\nabla^{2}\ell_{n_{k}}^{(j)}(\bbeta)=\frac{1}{n_k} X^{(j)T}D(X^{(j)}\bbeta)X^{(j)},\\
				S_{X}&:= \frac{1}{n}X^TX, \quad{S}^{(j)}_{X}:=\frac{1}{n_k} X^{(j)T}X^{(j)}
			\end{aligned}
	\]
	For any $j=1,\ldots,k$, $\hat\bbeta^{(j)}$ satisfies
		\[
			\nabla \ell_{n_{k}}^{(j)}(\hat\bbeta^{(j)})=\frac{1}{n_k}X^{(j)T}(\bY^{(j)}-\bmu(X^{(j)}\hat\bbeta^{(j)}))=0.
		\]
	 Through a Taylor expansion of the left hand side at the point $\bbeta=\bbeta^*$, we have
		\[
			\frac{1}{n_k}X^{(j)T}(\bY^{(j)}-\bmu(X^{(j)}\bbeta^*))-{S}^{(j)}(\hat\bbeta^{(j)}-\bbeta^*)-\rb^{(j)}=0,
		\]
		where the remainder term $\rb^{(j)}$ is a $d$ dimensional vector with $g^{th}$ component
		\[
			\begin{aligned}
				r_g^{(j)}&=\frac{1}{6n_k}(\hat\bbeta^{(j)}-\bbeta^*)^T\nabla_{\bbeta}^2 [(\bX_g^{(j)})^T\bmu(X^{(j)}\bbeta)](\hat\bbeta^{(j)}-\bbeta^*)\\
				&=\frac{1}{6n_k}(\hat\bbeta^{(j)}-\bbeta^*)^TX^{(j)T}\text{diag}\{\bX_g^{(j)}\circ\bmu''((X^{(j)}\widetilde{\bbeta}^{(j)}))\}X^{(j)}(\hat\bbeta^{(j)}-\bbeta^*),
			\end{aligned}
		\]
		where $\widetilde{\bbeta}^{(j)}$ is in a line segment between $\hat\bbeta^{(j)}$ and $\bbeta^*$. It therefore follows that
		\[
			\hat\bbeta^{(j)}=\bbeta^*+({S}^{(j)})^{-1}[X^{(j)T}(\bY^{(j)}-\bmu(X^{(j)}\bbeta^*))+n_k\rb^{(j)}].
		\]
		A similar equation holds for the global MLE $\hat\bbeta$:
		\[
			\hat\bbeta=\bbeta^*+S^{-1}[X^T(\bY-\bmu(X\bbeta^*))+n\rb],
		\]
		where for $g=1,\ldots,d$,
		\[
			r_g=\frac{1}{6n}(\hat\bbeta-\bbeta^*)^T{X^T}\text{diag}\{\bX_g\circ\bmu''((X\widetilde{\bbeta}^{(j)}))\}X(\hat\bbeta-\bbeta^*).
		\]
		Therefore we have
		\[
			\begin{aligned}
				\frac{1}{k}&\sum\limits_{j=1}^k \hat\bbeta^{(j)}-\hat\bbeta=\frac{1}{k}\sum\limits_{j=1}^k\left\{({S}^{(j)})^{-1}-\Sigma^{-1}\right\}X^{(j)T}(\bY^{(j)}-\bmu(X^{(j)}\bbeta^*))\\
				&-\left\{S^{-1}-\Sigma^{-1}\right\}{X^T}(\bY-\bmu(X\bbeta^*))+\bR=\bB+\bR,
			\end{aligned}
		\]
		where $\bR=(1/k)\sum\limits_{j=1}^k ({S}^{(j)})^{-1}\rb^{(j)}-S^{-1}\rb$. We next derive stochastic bounds for $\|\bB\|_2$ and $\|\bR\|_2$ respectively, but to study the appropriate threshold, we introduce the following events with probability that approaches one under appropriate scaling. For $j=1,\ldots,k$ and $\kappa, \tau, t>0$,
		\[
			\begin{aligned}
				\cE^{(j)}&:=\{\|({S}^{(j)})^{-1}\|_2\leq2/C_{\min}\}\cap\bigl\{\|{S}^{(j)}-\Sigma\|_2\leq(\delta_1 \vee\delta_1^2)\bigr\}\cap\{\|S_{X}^{(j)}\|_2\le 2C_{\max}\}, \\
				\cE&:=\{\|{S}^{-1}\|_2\leq 2/L_{\min}\}\cap\bigl\{\|{S}-\Sigma\|_2 \le(\delta_2 \vee\delta_2^2)\bigr\}\cap\{\|S_{X}\|_2\le 2C_{\max}\}, \\
				\cF^{(j)}&:=\left\{\|\hat\bbeta^{(j)}-\bbeta^*\|_2>t\right\}, \quad \cF :=\left\{\|\hat\bbeta-\bbeta^*\|_2>t\right\},
			\end{aligned}
		\]
		 where $\delta_1=C_1\sqrt{d/n_k}+\tau/\sqrt{n_k}$ and $\delta_2=C_1\sqrt{d/n_k}+\tau/\sqrt n$. Denote the intersection of all the above events by $\cA$. Note that Condition \ref{con:glm} implies that $\sqrt{b''(\bX_i^T\bbeta)}\bX_i$ are i.i.d. sub-gaussian vectors, so by Lemmas \ref{lem:cctts}, \ref{lem:mineg}, \ref{lem:maxeg} and \ref{lem:glm1}, we have
		 \[
		 	\begin{aligned}
		 		\PP(\cA^c) & \le (2k+1)\exp\left(-\frac{cn}{k}\right)+(k+1)\exp(-c_1\tau^2)+2k\exp\left(d\log 6-\frac{nC_{\min}^2L_{\min}^2t^2}{2^{11}C_{\max}U_2\phi k}\right).
			\end{aligned}
		 \]
	
We first consider the bounded design, i.e., Condition \ref{con:glm} (ii). In order to bound $\|\bR\|_2$, we first derive an upper bound for $r_g^{(j)}$. Under the event $\cA$, by Lemma \ref{lem:lipshitz} we have
		\[
			\max_{\substack{1 \le g \le d, 1 \le j \le k}} r_g^{(j)} \le \frac{1}{3} M U_3C_{\max}t^2\ \text{and}\ \max_{\substack{1 \le g \le d}} r_g \le \frac{1}{3}M U_3C_{\max}t^2.
		\]
	It follows that, under $\cA$,
		\begin{equation}
			\label{eq:glm2r}
			\|\bR\|_2\le \frac{2}{3}M\sqrt{d}U_3C_{\max}t^2.
		\end{equation}
		
Note that $\bB$ is very similar to the RHS of Equation (\ref{eq:decompose}). Now we use essentially the same proof strategy as in the OLS part to bound $\|\bB\|_2$. Following similar notations as in OLS, we denote $({S}^{(j)})^{-1}-\Sigma^{-1}$ by ${D}_1^{(j)}$, $S^{-1}-\Sigma^{-1}$ by ${D}_2$, $\bY^{(j)}-\bmu(X^{(j)}\bbeta^*)$ by $\bepsilon^{(j)}$ and $\bY-\bmu(X\bbeta^*)$ by $\bepsilon$. For concision, we relegate the details of the proof to Lemma \ref{lem:glmb}, which delivers the following stochastic bound on $\|\bB\|_2$.
		
\begin{equation}
			\label{eq:glm2b}
			\PP(\left\{\|\bB\|_2>t_1\right\}\cap \cA)\le  2\exp\left(d\log(6)-\frac{C_{\min}^4L_{\min}^2nt_1^2}{128\phi U_2C_{\max}(\delta_1\vee \delta_1^2)^2}\right).
		\end{equation}
		
Combining Equation (\ref{eq:glm2b}) with (\ref{eq:glm2r}) leads us to the following inequality.
	\[
		\begin{aligned}
			\PP &\left(\|\overline\bbeta-\hat\bbeta\|_2> \frac{2}{3}M\sqrt{d}U_3C_{\max}t^2+t_1\right) \le  (2k+1)\exp\left(-\frac{cn}{k}\right)+(k+1)\exp(-c_1\tau^2)\\
			&+ (k+1)\exp\left(d\log 6-\frac{C_{\min}^2L_{\min}^2nt^2}{2^{11}C_{\max}U_2\phi k}\right)+2\exp\left(d\log6-\frac{C_{\min}^4L_{\min}^2nt_1^2}{128\phi U_2C_{\max}(\delta_1 \vee \delta_1^2)^2}\right).
		\end{aligned}
	\]
	Choose $t=t_{1}=\sqrt{d/n_{k}}$ and, when $d\gg \log n$, choose $\tau=\sqrt{d/c_1}$ and $\delta_1=O(\sqrt{kd/n})$. Then there exists a constant $C>0$ such that
	\[
		\PP\left(\|\overline\bbeta-\hat\bbeta\|_2>C\frac{kd^{3/2}}{n}\right)\le (2k+1)\exp(-\frac{cn}{k})+2(k+2)\exp(-d).
	\]
	When it is not true that $d\gg \log n$, choose $\tau=\sqrt{\log n/c_1}$ and $\delta=O(\sqrt{k\log n/n})$. Then there exists a constant $C>0$ such that
	\[
		\PP\left(\|\overline\bbeta-\hat\bbeta\|_2>C\frac{k\sqrt d\log n}{n}\right)\le (2k+1)\exp(-\frac{cn}{k})+\frac{k+3}{n}.
	\]
	Overall, we have
	\[
		\PP\left(\|\overline\bbeta-\hat\bbeta\|_2>C\frac{k\sqrt d (d \vee \log n)}{n}\right)\le ck\exp(-{cn}/{k})+ck\exp(-c\max(d, \log n)),
	\]
	which leads to the final conclusion.
	\end{proof}

\begin{proof}[Proof of Corollary \ref{cor:refitglm}]
	Define an event $\cE=\{\|\overline\bbeta^d-\bbeta^*\|_\infty\le 2C\sqrt{\log d/n}\}$, then by the conditions of Corollary \ref{cor:refitglm} and results of Lemma \ref{lm:est-inf-rate-glm} and Theorem \ref{thm:glm2},
	\[
		\begin{aligned}
			\PP\left(\|\overline\bbeta^r-\hat\bbeta^o\|_2>C'\frac{k\sqrt{s}(s \vee \log n)}{n}\right) &\le\PP\left(\left\{\|\overline\bbeta^r-\hat\bbeta^o\|_2>C'\frac{k\sqrt{s}(s \vee \log n)}{n}\right\}\cap\cE\right)+\PP(\cE^c)\\
			&\le \PP\left(\left\{\|\overline\bbeta^o-\hat\bbeta^o\|_2>C'\frac{k\sqrt{s}(s \vee \log n)}{n}\right\}\cap\cE\right)+c/d \\
			&\le ck\exp(-(s \vee \log n))+k\exp(-{cn}/{k})+c/d.
		\end{aligned}
	\]	
	where $\overline\bbeta^o=\frac{1}{k}\sum\limits_{j=1}^k \hat\bbeta^o(\cD_j)$, $\hat\bbeta^o(\cD_j)=\argmax_{\bbeta\in\RR^d, \bbeta_{S^c}=0} \ell\sj(\bbeta)$ and $C'$ is a constant. Then it is not hard to see that the final conclusion is true.
\end{proof}

\noindent \textbf{Acknowledgements:} The authors thank Weichen Wang, Jason Lee and Yuekai Sun for helpful comments.

\bibliographystyle{ims}
\bibliography{dcInferenceBib2}

\newpage

\setcounter{page}{1}

\begin{center}
\textit{\large Supplementary material to}
\end{center}
\begin{center}
\title{\LARGE Distributed Estimation and Inference with Statistical Guarantees}
\vskip10pt
\author{Heather Battey$^{*\dagger}$ \and Jianqing Fan$^{*}$ \and Han Liu$^{*}$ \and Junwei Lu$^{*}$ \and Ziwei Zhu$^{*}$}
\end{center}

\let\thefootnote\relax\footnotetext{
\noindent$^\ast$Department of Operations Research and Financial Engineering, Princeton University, Princeton, NJ 08540; Email: $\{$\texttt{hbattey},\texttt{jqfan},\texttt{hanliu},\texttt{junweil},\texttt{ziweiz}$\}$ \texttt{@princeton.edu};\\
\noindent$^\dagger$Department of Mathematics, Imperial College London, London, SW7 2AZ; Email: \texttt{h.battey@imperial.ac.uk}
}

\setcounter{section}{0}
\renewcommand{\thesection}{\Alph{section}}

\maketitle
\begin{abstract}
  This document contains the supplementary material to the paper
  ``Distributed Estimation and Inference with Statistical Guarantees".  In Appendix
  \ref{appendixAuxLemmata}, we provide the proofs of technical results required for the analysis of divide and conquer inference. Appendix~\ref{sec:est-apen} collects the proofs of lemmas for the estimation part.
  
  \end{abstract}

\begin{appendix}

\section{Auxiliary Lemmas for Inference}\label{appendixAuxLemmata}

In this section, we provide the proofs of the technical lemmas for the divide and conquer inference.

\begin{lemma}\label{lemmaLB}
Under Condition \ref{con:subg}, $\bigl(\mb_{v}^{(j)T}\hat{\Sigma}\mb_{v}\sj\bigr)^{-1/2}\geq c_{n_{k}}$ for any $j\in\{1,\ldots,k\}$ and for any $v\in\{1,\ldots,d\}$, where $c_{n_{k}}$ satisfies $\liminf_{n_{k}\rightarrow\infty}c_{n_{k}}=c_{\infty}>0$.
\end{lemma}

\begin{proof}
The proof appears in the proof of Lemma B1 of \citet{ZKL2014}.
\end{proof}

\begin{lemma}
		\label{lem:mgf}
		Under the GLM (\ref{eq:glm}), we have
		\[
			\EE \exp(t(Y-\mu(\theta)))=\exp(\phi^{-1}(b(\theta+t\phi)-b(\theta)-\phi tb'(\theta))),
		\]
		and typically when there exists $U>0$ such that $b''(\theta)<U$ for all $\theta\in\RR$, we will have
		\[
			\EE \exp(t(Y-\mu(\theta)))\le\exp\left(\frac{\phi Ut^2}{2}\right),
		\]
		which implies that $Y$ is a sub-Gaussian random variable with variance proxy $\phi U$.  
	\end{lemma}
	\begin{proof}
		\[
			\begin{aligned}
				\EE\exp\left(t(Y-\mu(\theta))\right)&=\int_{-\infty}^{+\infty} c(y)\exp\left(\frac{y\theta-b(\theta)}{\phi}\right)\exp(t(y-\mu(\theta))) dy\\
				&=\int_{-\infty}^{+\infty} c(y)\exp\left(\frac{(\theta+t\phi)y-(b(\theta)+\phi tb'(\theta))}{\phi}\right)dy\\
				&=\int_{-\infty}^{+\infty} c(y)\exp\left(\frac{(\theta+t\phi)y-b(\theta+t\phi)+b(\theta+t\phi)-(b(\theta)+\phi tb'(\theta))}{\phi}\right)dy\\
				&=\exp\left(\phi^{-1}(b(\theta+t\phi)-b(\theta)-\phi tb'(\theta))\right).
			\end{aligned}
		\]
		When $b''(\theta)<U$. the mean value theorem gives
		\[
			\EE\exp\left(t(Y-\mu(\theta))\right)=\exp\left(\frac{b''(\tilde\theta)\phi^2t^2}{2\phi}\right) \le \exp\left(\frac{\phi Ut^2}{2}\right).
		\]
	\end{proof}

\begin{lemma}\label{lemmaA2}
Under Condition \ref{con:glm},  we have for any $\bbeta, \bbeta' \in \RR^{d}$ and any $i=1,\ldots,n$, $\bigl|\ell_{i}''(\bX_{i}^{T}\bbeta)-\ell_{i}''(\bX_{i}^{T}\bbeta')\bigr|\leq K_{i}|\bX_{i}^{T}(\bbeta-\bbeta')|$, where $0<K_{i}<\infty$.
\end{lemma}

\begin{proof}
By the canonical form of the generalized linear model (equation \eqref{eqGLMLikelihood}),
\[
\bigl|\ell_{i}''(\bX_{i}^{T}\bbeta)-\ell_{i}''(\bX_{i}^{T}\bbeta')\bigr| = \bigl|b''(\bX_{i}^{T}\bbeta)-b''(\bX_{i}^{T}\bbeta')\bigr|\leq |b'''(\widetilde{\eta})||\bX_{i}^{T}(\bbeta - \bbeta')|
\]
by the mean value theorem, where $\widetilde{\eta}$ lies in a line segment between $\bX_{i}^{T}\bbeta$ and $\bX_{i}^{T}\bbeta'$. $|b'''(\eta)|<U_{3}<\infty$ by Condition \ref{con:glm} for any $\eta$, hence the conclusion follows with $K_{i}=U_{3}$ for all $i$.
\end{proof}

\begin{lemma}\label{lemmaA3}
Under Conditions \ref{con:lasso} and \ref{con:estimation} (i), we have for any $\delta\in(0,1)$ such that $\delta^{-1}\ll d$,
\[
\PP\Bigl(\frac{1}{n}\bigl\|\bX(\widehat{\bbeta}^{\lambda} - \bbeta^{*})\bigr\|_{2}^{2}\gtrsim s\frac{\log(d/\delta)}{n}\Bigr) < \delta
\]
\end{lemma}

\begin{proof}
Decompose the object of interest as
\begin{eqnarray*}
\frac{1}{n}\bigl\|\bX(\widehat{\bbeta}^{\lambda} - \bbeta^{*})\bigr\|_{2}^{2} &=& (\widehat{\bbeta}^{\lambda} - \bbeta^{*})^{T}(\widehat{\Sigma} - \Sigma)(\widehat{\bbeta}^{\lambda} - \bbeta^{*}) + (\widehat{\bbeta}^{\lambda} - \bbeta^{*})^{T}\Sigma(\widehat{\bbeta}^{\lambda} - \bbeta^{*}) \\
&\leq & \|\widehat{\Sigma} - \Sigma\|_{\max}\|\widehat{\bbeta}^{\lambda} - \bbeta^{*}\|_{1}^{2} + \lambda_{\max}(\Sigma)\|\widehat{\bbeta}^{\lambda} - \bbeta^{*}\|_{2}^{2}.
\end{eqnarray*}
This gives rise to the tail probability bound
\begin{equation}\label{eqProbBound}
\PP\Bigl(\frac{1}{n}\bigl\|\bX(\widehat{\bbeta}^{\lambda} - \bbeta^{*})\bigr\|_{2}^{2} > t\Bigr) \leq \PP\Bigl( \|\widehat{\Sigma} - \Sigma\|_{\max}\|\widehat{\bbeta}^{\lambda} - \bbeta^{*}\|_{1}^{2}>\frac{t}{2}\Bigr) + \PP\Bigl(\lambda_{\max}(\Sigma)\|\widehat{\bbeta}^{\lambda} - \bbeta^{*}\|_{2}^{2} > \frac{t}{2}\Bigr).
\end{equation}
Let $\cM:=\big\{\|\hat\Sigma-\Sigma\|_{\infty}\le M\big\}$. Since $\{\bX_i\}_{i=1}^n$ is bounded, it is sub-Gaussian as well. Suppose $\|\bX_i\|_{\psi_2}<\kappa$, then by Lemma \ref{lem:bsineq} we have,
	\[
		\begin{aligned}
			\PP(\cM^c) & \le \sum\limits_{p,q=1}^d \PP(|\hat \Sigma_{pq}\sj-\Sigma_{pq}|>M)\\
			&\le d^2\exp\left(-C n\cdot\min\Bigl\{\frac{M^2}{\kappa^4}, \frac{M}{\kappa^2}\Bigr\}\right),
		\end{aligned}
	\]
	where $C$ is a constant. Hence taking $M=n^{-1}\log(d/\delta)$,
	\[
	\PP(\mathcal{M}^{c}) \leq d^{2} \exp\left\{-Cn \min\Bigl\{\frac{(\log(d/\delta))^{2}}{\kappa^{4}n^{2}}, \frac{(\log(d/\delta))^{2}}{\kappa^{2}n} \Bigr\}\right\}
	\]
	and the right hand side is less than $\delta$ for $\delta^{-1}\ll d$. Thus by Condition \ref{con:estimation}, the first term on the right hand side of equation \eqref{eqProbBound} is
	\[
\PP\Bigl(\bigl\|\widehat{\Sigma} - \Sigma\bigr\|_{\max}\bigl\|\widehat{\bbeta}^{\lambda} - \bbeta^{*}\bigr\|_{1}^{2} \gtrsim \frac{s\log(d/\delta)}{n}\Bigr)<2\delta.
	\]
	Furthermore, by Condition \ref{con:glm} (i), the second term on the right hand side of equation \eqref{eqProbBound} is
	\[
	\PP\Bigl(\lambda_{\max}(\Sigma)\bigl\|\widehat{\bbeta}^{\lambda} - \bbeta^{*}\bigr\|_{2}^{2} \gtrsim C_{\max}\frac{s\log(d/\delta)}{n}\Bigr)<\delta.
	\]
	Taking $t$ as the dominant term, $t\asymp C_{\max}n^{-1}s\log(d/\delta)$, yields the result.
\end{proof}

\begin{lemma}
	\label{lem:lipshitz}
	Under Condition \ref{con:glm},  we have for any $i=1,\ldots,n$,
	\[
		|b''(\bX_i^T\bbeta_1)-b''(\bX_i^T\bbeta_2)|\le MU_3\|\bbeta_1-\bbeta_2\|_1,
	\]
	and if we consider the sub-Gaussian design instead, we have
	\[
		\PP\left(|b''(\bX_i^T\bbeta_1)-b''(\bX_i^T\bbeta_2)|\ge h U_3\|\bbeta_1-\bbeta_2\|_1\right)\le  nd\exp\left(1-\frac{Ch^2}{s_1^2}\right).
	\]
\end{lemma}
\begin{proof}
	For the bounded design, by Condition \ref{con:glm} (iii), we have
	\[
		|b''(\bX_i^T\bbeta_1)-b''(\bX_i^T\bbeta_2)|\le U_3|\bX_i^T(\bbeta_1-\bbeta_2)|\le U_3\|\bX_i\|_{\max}\|\bbeta_1-\bbeta_2\|_1\le MU_3\|\bbeta_1-\bbeta_2\|_1.
	\]
	For the sub-Gaussian design, denote the event $\{\max_{\substack{1\le i\le n, 1\le j\le d}}|X_{ij}|\le h\}$ by $\cC$, where $\kappa$ is a positive constant. Then it follows that,
		\[
			\PP\left(\cC^c\right)\le nd\exp\left(1-\frac{Ch^2}{s_1^2}\right),
		\]
	where $C$ is a constant. Since on the event $\cC$, $|b''(\bX_i^T\bbeta_1)-b''(\bX_i^T\bbeta_2)|\le h U_3\|\bbeta_1-\bbeta_2\|_1$, we reach the conclusion.
\end{proof}

\begin{remark}
	For the sub-Gaussian design, in order to let the tail probability go to zero, $h\gg \log((n\vee d))$.
\end{remark}

\begin{lemma}\label{lemmaWaldDistVdGWeakerConditions}
Suppose, for any $k\ll d$ satisfying $k=o\bigl(((s\vee s_{1})log d)^{-1}\sqrt{n}\bigr)$, the following conditions are satisfied. (A1) $\PP\left(n_{k}^{-1}\bigl\|X^{(j)}\widehat\bTheta\sj\bigr\|_{\max} \geq H\right)\le \xi$, where $H$ is a constant and $\xi=o(k^{-1})$. (A2) For any $\bbeta,\bbeta' \in \RR^{d}$ and for any $i\in\{1,\ldots,n\}$, $\bigl|\ell''_{i}(\bX_{i}^{T}\bbeta) - \ell''_{i}(\bX_{i}^{T}\bbeta')\bigr|\leq K_{i}\bigl|\bX_{i}^{T}(\bbeta - \bbeta')\bigr|$ with $\PP(K_i> h)\le \psi$ for $\psi=o(k^{-1})$ and $h=O(1)$. (A3) $\PP\Bigl(n_{k}^{-1}\bigl\|X^{(j)}(\widehat{\bbeta}^{\lambda} - \bbeta^{*})\bigr\|_{2}^{2}\gtrsim n^{-1}sk\log(d/\delta)\Bigr)<\delta$. (A4) $\PP\Bigl(\max_{1\leq v\leq d}\Bigl|\bigl(\hat{\bTheta}^{(j)T}_{v} \nab^{2}\ell_{n_{k}}^{(j)}\bigl(\hat{\bbeta}^{\lambda}(\cD_{j})\bigr)-\be_{v}\bigr)\bigl(\hat{\bbeta}^{\lambda}(\cD_{j}) - \bbeta^{*}\bigr)\Bigr| \gtrsim n^{-1}sk\log(d/\delta)\Bigr)< \delta$. Then
\[
\overline{\beta}_{v}^{d} - \beta_{v}^{*} = -\frac{1}{k} \sum_{j=1}^{k} \hat{\bTheta}_{v}^{(j)T}\nab\ell_{n_{k}}\sj (\bbeta^{*}) + o_{\PP}(n^{-1/2}).
\]
for any $1\leq v \leq d$.
\end{lemma}

\begin{proof}[Proof of Lemma \ref{lemmaWaldDistVdGWeakerConditions}] $\overline{\beta}_{v}^{d}-\beta_{v}^{*}=k^{-1}\sum_{j=1}^{k} \bigl(\hat{\beta}_{v}(\cD_j) - \beta_{v}^{*}) \bigr)$. By the definition of $\hat{\bbeta}^{d}(\cD_{j})$,
\[
\hat{\beta}_{v}^{d} (\cD_{j}) - \beta_{v}^{*} = \hat{\beta}_{v}^{\lambda} (\cD_{j}) - \beta_{v}^{*} - \hat{\bTheta}_{v}^{(j)T}\nab\ell_{n_{k}}\sj (\hat{\bbeta}^{\lambda}(\cD_{j})).
\]
Consider a mean value expansion of $\nab \ell_{n_{k}}\sj(\hat{\bbeta}^{\lambda}\bigl(\cD_{j})\bigr)$ around $\bbeta^{*}$:
\[
\nab \ell_{n_{k}}\sj \bigl(\hat{\bbeta}^{\lambda}(\cD_{j})\bigr) = \nab \ell_{n_{k}}\sj (\bbeta^{*}) + \nab^{2} \ell_{n_{k}}\sj(\bbeta_{\alpha})\bigl(\hat{\bbeta}^{\lambda}(\cD_{j}) - \bbeta^{*}\bigr),
\]
where $\bbeta_{\alpha}=\alpha \hat{\bbeta}^{\lambda}(\cD_{j}) + (1-\alpha)\bbeta^{*}$, $\alpha\in[0,1]$. So
\[
\frac{1}{k}\sum_{j=1}^{k}\hat{\beta}_{v}^{d}(\cD_{j}) - \beta_{v}^{*} = - \frac{1}{k}\sum_{j=1}^{k}\hat{\bTheta}_{v}^{(j)T}\nab\ell_{n_{k}}\sj(\bbeta^{*}) - \underbrace{\frac{1}{k}\sum_{j=1}^{k}\bigl(\hat{\bTheta}_{v}^{(j)T}\nab^{2}\ell_{n_{k}}\sj(\bbeta_{\alpha}) - \be_{v}\bigr)(\hat{\bbeta}^{\lambda}(\cD_{j}) - \bbeta^{*})}_{\Delta}
\]
and $|\Delta|\leq \frac{1}{k}\sum_{j=1}^{k} \bigl(|\Delta_{1}\sj| + |\Delta_{2}\sj|\bigr)$ where
\[
\bigl|\Delta_{1}^{(j)}\bigr| =\Bigl|\bigl(\hat{\bTheta}^{(j)T}_{v} \nab^{2}\ell_{n_{k}}^{(j)}\bigl(\hat{\bbeta}^{\lambda}(\cD_{j})\bigr)-\be_{v}\bigr)\bigl(\hat{\bbeta}^{\lambda}(\cD_{j}) - \bbeta^{*}\bigr)\Bigr|.
\]
By (A4) of the lemma, for $t \asymp n^{-1} s k \log(d/\delta)$,
\[
\PP\Bigl(|\sum_{j=1}^{k} \Delta_{1}\sj|> kt \Bigr) \leq \PP\Bigl( \cup_{j=1}^{k}|\Delta_{1}\sj|>t\Bigr)\leq \sum_{j=1}^{k} \PP(|\Delta_{1}\sj|>t) < k\delta.
\]
Substituting $\delta=o(k^{-1})$ in the expression for $t$ and noting that $k\ll d$, we obtain $k^{-1}\sum_{j=1}^{k}\Delta_{1}^{(j)}=o_{\PP}(n^{-1/2})$ for $k=o\bigl( (s\log d)^{-1}\sqrt{n} \bigr)$. By (A2),
\begin{eqnarray*}
\bigl| \Delta_{2}\sj \bigr| &=& \Bigl|\hat{\bTheta}_{v}^{(j)T}\bigl(\nab^{2}\ell_{n_{k}}\sj(\bbeta_{\alpha}) - \nab^{2}\ell_{n_{k}}\sj(\hat{\bbeta}^{\lambda}(\cD_{j}))\bigr)\bigl(\hat{\bbeta}^{\lambda}(\cD_{j}) - \bbeta^{*}\bigr)\Bigr| \\
&=& \Bigl|\frac{1}{n_{k}}\sum_{i\in\mathcal{I}_{j}}\hat{\bTheta}_{v}^{(j)T}\bX_{i}\bX_{i}^{T}\bigl(\hat{\bbeta}^{\lambda}(\cD_{j}) - \bbeta^{*}\bigr) \bigl(\ell''_{i}(\bX_{i}^{T}\bbeta_{\alpha}) - \ell''_{i}(\bX_{i}^{T}\hat{\bbeta}^{\lambda}(\cD_{j}))\bigr)\Bigr| \\
&\leq & \Bigl(\max_{1\leq i \leq n} K_{i}\Bigr) \Bigl(\frac{1}{n_{k}}\|X^{(j)}\widehat{\Theta}^{(j)}\|_{\max}\Bigr)\Bigl\|\frac{1}{n_{k}}X^{(j)}(\widehat{\bbeta}^{\lambda}(\cD_{j})-\bbeta^{*})\Bigr\|_{2}^{2},
\end{eqnarray*}
therefore by (A1) and (A3) of the lemma, for $t \asymp n^{-1} sk \log(d/\delta)$,
\[
\PP\Bigl(|\sum_{j=1}^{k} \Delta_{2}\sj|> kt \Bigr) \leq \PP\Bigl( \cup_{j=1}^{k}|\Delta_{2}\sj|>t\Bigr)\leq \sum_{j=1}^{k} \PP(|\Delta_{2}\sj|>t) < k(\psi + \delta + \xi).
\]
Substituting $\delta=o(k^{-1})$ in the expression for $t$ and noting that $k\ll d$, we obtain $k^{-1}\sum_{j=1}^{k}\Delta_{2}^{(j)}=o_{\PP}(n^{-1/2})$ for $sk\log(d/\delta)=o(\sqrt{n})$, i.e.~for $k=o\bigl( (s\log d)^{-1}\sqrt{n} \bigr)$. Combining these two results delivers $\Delta=o_{\PP}(n^{-1/2})$ for $k=o\bigl( (s\log d)^{-1}\sqrt{n} \bigr)$.
\end{proof}

\begin{lemma}\label{lemmaWaldDistVdG2WeakerConditions}
Suppose, in addition to Conditions (A1)-(A5) of Lemma \ref{lemmaWaldDistVdGWeakerConditions}, (A5) $\bigl|\widetilde{\Theta}_{vv} -\Theta^{*}_{vv} \big|=o_{\PP}(1)$ for all $v\in\{1,\ldots, d\}$; (A6) $1/\Theta^{*}_{vv}=O(1)$ for all $v\in\{1,\ldots,d\}$; (A7) $\|\sum_{1 \le j\le k}\sum_{i \in \mathcal{I}_j} \nabla\ell_i(\bbeta^*)\|_{\infty}=O_{\PP}(\sqrt{n\log d})$; (A8) For each $v\in\{1,\ldots,d\}$, letting $\xi_{iv}\sj=\bTheta^{*T}_{v}\nab \ell_{i}\sj(\bbeta^{*})/\sqrt{n\Theta^{*}_{vv}}$, $\EE\bigl[\xi_{iv}\sj\bigr]=0$, $\Var\bigl(\sum_{j=1}^{k} \sum_{i\in\mathcal{I}_{j}}\xi_{iv}\sj\bigr)=1$ and, for all $\epsilon>0$,
\begin{equation}\label{eqLindebergVdG}
\lim_{k\rightarrow \infty}\lim_{n_{k}\rightarrow \infty}\sum_{j=1}^{k} \sum_{i\in \mathcal{D}_{j}}\EE\bigl[(\xi_{iv}\sj)^{2}\ind\{|\xi_{iv}\sj|>\epsilon\}\bigr]=0.
\end{equation}
Then under $H_{0}:\beta_{v}^{*}=\beta_{v}^{H}$, taking $k=o(((s\vee s_{1})\log d)^{-1}\sqrt n)$ delivers $\overline{S}_{n}\leadsto N(0,1)$, where $\overline{S}_{n}$ is defined in equation \eqref{vdGWaldStat}.
\end{lemma}

\begin{proof}
Rewrite equation \eqref{vdGWaldStat} as
\begin{eqnarray}\label{barSExpression}
\nonumber \overline{S}_{n} %&=&  \sqrt{n} \frac{1}{k}\sum_{j=1}^{k} \frac{\widehat{\theta}^{d}-\theta^{H}}{\sqrt{\bigl[ \bigr]_{vv}}} \\
					   &=&  \sqrt{n} \frac{1}{k}\sum_{j=1}^{k} \left[\frac{\widehat{\beta}_{v}^{d}-\beta_{v}^{H}}{(\Theta^{*}_{vv})^{1/2}} + \frac{\widehat{\beta}_{v}^{d}-\beta_{v}^{H}}{(\Theta^{*}_{vv})^{1/2}}\left(\frac{(\Theta^{*}_{vv})^{1/2}}{\bigl[\widehat{\Theta}^{(j)} \widehat{H}^{(j)} \widehat{\Theta}^{(j)T}\bigr]_{vv}^{1/2}} - 1\right) \right] \\
					   &=&\sum_{j=1}^{k}\sum_{i\in\mathcal{I}_{j}} \bigl(\Delta_{1,i}^{(j)} +\Delta_{2,i}^{(j)}\bigr), \qquad \text{where}
\end{eqnarray}
\[
\Delta_{1,i}^{(j)}=\frac{\widehat{\bTheta}^{(j)T}_{v}\nab \ell_{i}^{(j)}(\bbeta^{*})}{(n\Theta^{*}_{vv})^{1/2}}, \qquad
\Delta_{2,i}^{(j)}=\frac{\widehat{\bTheta}^{(j)T}_{v}\nab \ell_{i}^{(j)}(\bbeta^{*})}{(n\Theta^{*}_{vv})^{1/2}}\left(\frac{(\Theta^{*}_{vv})^{1/2}}{\overline{\Theta}_{vv}^{1/2}} - 1\right).
\]
Further decomposing the first term, we have
\[
\sum_{j=1}^{k}\sum_{i\in\mathcal{I}_{j}} \Delta_{1,i}^{(j)} = \sum_{j=1}^{k}\sum_{i\in\mathcal{I}_{j}} \xi_{i,v}^{(j)} + \Delta, \;\; \text{ where } \;\; \Delta =\sum_{j=1}^{k} \sum_{i\in\mathcal{I}_{j}}\bigl( \widehat{\bTheta}^{(j)}_{v}-\bTheta^{*}_{v} \bigr)^{T} \frac{\nab\ell_{i}(\bbeta^{*})}{(n\Theta^{*}_{vv})^{1/2}}
\]
and $ \sum_{j=1}^{k}\sum_{i\in\mathcal{I}_{j}} \xi_{i,v}^{(j)} \leadsto N(0,1)$ by the Lindeberg-Feller central limit theorem. Then by H\"older's inequality, Condition \ref{con:theta} and Assumption (A6) and (A7),
\begin{eqnarray*}
|\Delta | &\leq& \max_{1\leq j \leq k}\bigl\|\widehat{\bTheta}^{(j)}_{v}-\bTheta^{*}_{v}\bigr\|_{1}\frac{\|\sum_{j=1}^{k}\sum_{i\in\mathcal{I}_{j}} \nabla\ell_i(\bbeta^*) \bigr\|_{\infty}}{(n\Theta^*_{vv})^{1/2}} \\
             &=& O_{\PP}\Big(s_1\sqrt{\frac{k\log d}{n}}\Big)O_{\PP}(\sqrt{\log d}) = o_{\PP}(1),
\end{eqnarray*}
where the last equation holds with the choice of $k=o((s_1\log d)^{-1}\sqrt n)$.
Letting $\overline{\Delta}^{(j)} = (\Theta^{*}_{vv})^{1/2} - \overline{\Theta}_{vv}^{1/2}$ we have
\begin{eqnarray*}
\sum_{j=1}^{k}\sum_{i\in\mathcal{I}_{j}} \Delta_{2,i}^{(j)} &=& \sum_{j=1}^{k}\sum_{i\in\mathcal{I}_{j}}\frac{\bTheta^{*T}_{v}\nab\ell_{i}^{(j)}(\bbeta^{*})}{(\Theta^{*}_{vv})^{1/2}}\overline{\Delta}^{(j)} + \sum_{j=1}^{k}\sum_{i\in\mathcal{I}_{j}}\bigl(\widehat{\bTheta}^{(j)}_{v} - \bTheta^{*}_{v}\bigr)^{T}\frac{\nab \ell_{i}(\bbeta^{*})}{(\bTheta^{*}_{vv})^{1/2}} \overline{\Delta}^{(j)} \\
					&=& \sum_{j=1}^{k}\sum_{i\in\mathcal{I}_{j}} \bigl(\Delta_{21,i}^{(j)} + \Delta_{22,i}^{(j)}\bigr), \text{ where }
\end{eqnarray*}
\[
\Bigl|\sum_{j=1}^{k}\sum_{i\in\mathcal{I}_{j}} \Delta_{21,i}^{(i)}\Bigr| \leq \Bigl|\sum_{j=1}^{k}\sum_{i\in\mathcal{I}_{j}} \xi_{i,v}^{(j)} \Bigr|\bigl| \overline{\Theta}^{1/2}_{vv} - (\Theta^{*}_{vv})^{1/2}\bigr|.
\]
Since $\Theta^{*}_{vv}\geq 0$, $\displaystyle{\overline{\Theta}_{vv}^{1/2} = |\overline{\Theta}_{vv}|^{1/2} = |\overline{\Theta}_{vv} - \Theta^{*}_{vv} + \Theta^{*}_{vv}|^{1/2} \leq |\overline{\Theta}_{vv} - \Theta^{*}_{vv}|^{1/2} + (\Theta^{*}_{vv})^{1/2}}$. Similarly
\[
(\Theta^{*}_{vv})^{1/2} = |\Theta^{*}_{vv}|^{1/2} = |\Theta^{*}_{vv} - \overline{\Theta}_{vv} + \overline{\Theta}_{vv}|^{1/2} \leq |\Theta^{*}_{vv} - \overline{\Theta}_{vv} |^{1/2} + \overline{\Theta}_{vv}^{1/2},
\]
yielding $|\overline{\Theta}_{vv}^{1/2} - (\Theta^{*}_{vv})^{1/2}|\leq |\overline{\Theta}_{vv} - \Theta^{*}_{vv}|^{1/2}$ and consequently, by assumption (A5),
\[
\bigl|\overline{\Delta}^{(j)}\bigr|=\bigl|\overline{\Theta}_{vv}^{1/2} - (\Theta^{*}_{vv})^{1/2}\bigr|=o_{\PP}(1).
\]
Invoking (A9) and the Lindeberg-Feller CLT, $\displaystyle{
\Bigl|\sum_{j=1}^{k}\sum_{i\in\mathcal{I}_{j}} \Delta_{21,i}^{(i)}\Bigr|=O_{\PP}(1)o_{\PP}(1) = o_{\PP}(1)}$. Similarly
\[
\Bigl|\sum_{j=1}^{k}\sum_{i\in\mathcal{I}_{j}} \Delta_{22,i}^{(j)}\Bigr|\leq  \max_{1\leq j \leq k} \|\widehat{\bTheta}^{(j)}_{v}-\bTheta^{*}_{v}\|_{1}\bigl|\overline{\Delta}^{(j)}\bigr|\Bigl|\bigl(\bTheta^{*T}_{v}\bTheta^{*}_{v}\bigr)^{-1/2}\sum_{j=1}^{k}\sum_{i\in\mathcal{I}_{j}} \xi_{iv}^{(j)} \Bigr| = o_{\PP}(1).
\]
Combining all terms in the decomposition \eqref{barSExpression} delivers the result.
\end{proof}

(B1)-(B5) of Condition \ref{conditionScoreLR} are used in the proofs of subsequent lemmas.

\begin{condition}\label{conditionScoreLR}
(B1) $\|\bw^*\|_1 \lesssim s_1 $, $\|J^{*}\|_{\max}<\infty$ and for any $\delta \in (0,1)$,
\[
\PP\Bigl(\|\hat{\bbeta}_{-v}^{\lambda} - \bbeta_{-v}^{*}\|_1 \gtrsim n^{-1/2}s\sqrt{\log (d/\delta)}\Bigr)<\delta \;\; \text{ and } \;\; \PP\Bigl( \|\hat{\bw} - \bw^{*}\|_1 \gtrsim n^{-1/2}s_1\sqrt{\log (d/\delta)}\Bigr) < \delta.
\]
(B2) For any $\delta \in (0,1)$,
\[
\PP\Bigl(\|\nabla_{-v} \ell_{n}(\beta_{v}^{*}, \bbeta_{-v}^*)\|_{\infty} \gtrsim n^{-1/2}\sqrt{\log (d/\delta)}\Bigr)<\delta.
\]
(B3) Suppose $\hat{\bbeta}_{-v}^{\lambda}$ satisfies (B1). Then for $\bbeta_{-v,\alpha} = \alpha\bbeta_{-v}^* + (1-\alpha)\hat{\bbeta}_{-v}^{\lambda}$ and for any $\delta \in (0,1)$,
\[
\PP\left( \sup_{\alpha\in [0,1]} \Bigl| \bigl(\nabla^2_{v,-v} \ell_{n}(\beta_{v}^{*}, \bbeta_{-v,\alpha}) - \hat{\bw}^T \nabla^2_{-v,-v} \ell_{n}(\beta_{v}^{*}, \bbeta_{-v,\alpha})\bigr)(\widehat{\bbeta}_{-v}^{\lambda}-\bbeta_{-v}^{*})\Bigr| \gtrsim s_{1} s \frac{\log (d/\delta)}{n}\right) < \delta.
\]
(B4) There exists a constant $C>0$ such that $C<I^*_{\theta|\bgamma} < \infty$, and for $\bv^* = (1, -\bw^{*T})^T$, it holds that
 \[
 \frac{\sqrt{n} \bv^{*T} \nabla\ell_{n}(\beta_{v}^{*}, \bbeta_{-v}^*)}{\sqrt{\bv^{*T} J^{*} \bv^*}} \leadsto  N(0,1).
 \]
(B5) For any $\delta$, if there exists an estimator $\widetilde{\bbeta}=(\widetilde{\beta}_{v}^{T},\widetilde{\bbeta}_{-v}^{T})^{T}$ satisfying $\|\widetilde{\bbeta}-\bbeta^{*}\|_{1}\leq Cs\sqrt{n^{-1}\log (d/\delta)}$ with probability $>1-\delta$, then
\[
\PP\Bigl(\bigl\|\nab^{2}\ell_{n}(\widetilde{\bbeta})-J^{*}\bigr\|_{\max}\gtrsim n^{-1/2}\sqrt{\log(d/\delta)}\Bigr) < \delta.
\]
\end{condition}

The proof of Theorem \ref{thmScoreDistn} is an application of Lemma \ref{lemmaDistributionAggScore}. To apply this Lemma, we must first verify (B1) to (B4) of Condition \ref{conditionScoreLR}. We do this in Lemma \ref{lemmaCheckB1toB4}.

\begin{lemma}\label{lemmaCheckB1toB4}
Under the requirements of Theorem \ref{thmScoreDistn}, (B1) - (B4) of Condition \ref{conditionScoreLR} are fulfilled.
\end{lemma}

\begin{proof}
\noindent \textbf{Verification of (B1)}.
As stated in Theorem \ref{thmScoreDistn}, $\|\bw^{*}\|_{1}=O(s_{1})$ and $\|J^{*}\|_{\max}<\infty$ by part (i) of Condition \ref{con:glm}. The rest of (B1) follows from the proof of Lemma C.3 of \citet{NingLiu2014b}. \\

\noindent \textbf{Verification of (B2)}. Let $\bX_{i}=(Q_{i}, \bZ_{i}^{T})^{T}$. Since $\| \nab_{\bgamma}\ell_{n}(\bbeta^{*})\|_{\infty} =\bigl\|-\frac{1}{n}\sum_{i=1}^{n}\bigl(Y_{i}-b'(\bX_{i}^{T}\bbeta^{*})\bigr)\bZ_{i}\bigr\|_{\infty}$, since the product of a subgaussian random variable and a bounded random variable is subgaussian, and since  $\EE[\nab_{\bgamma}\ell_{n}(\bbeta^{*})]=0$, we have by Condition \ref{con:glm}, Bernstein's inequality and the union bound
\[
\PP\bigl(\|\nab_{\bgamma}\ell_{n}(\bbeta^{*})\|_{\infty}>t\bigr) < (d-1)\exp\{-nt^{2}/M^{2}\sigma^{2}_{b}\}.
\]
Setting $2(d-1)\exp\{-nt^{2}/M^{2}\sigma^{2}_{b}\}=\delta$ and solving for $t$ delivers the result. \\

\noindent \textbf{Verification of (B3)} Let $\bbeta^{*}_{\alpha}=(\theta^{*},\bgamma_{\alpha})$ and decompose the object of interest as
\begin{equation}\label{eqnObjectInterest}
\bigl|\bigl(\nabla^2_{v,-v} \ell_{n}(\beta_{v}^{*}, \bbeta_{-v,\alpha}) - \hat{\bw}^T \nabla^2_{-v,-v} \ell_{n}(\beta_{v}^{*}, \bbeta_{-v,\alpha})\bigr)(\widehat{\bbeta}_{-v}^{\lambda}-\bbeta_{-v}^{*})\bigr| \leq \sum_{t=1}^{5}\bigl|\Delta_{t}\bigr|,
\end{equation}
where the terms $\Delta_{1}$ - $\Delta_{5}$ are given by $\Delta_{1}=\nabla_{v,-v}^{2}\ell_{n}(\bbeta_{\alpha}^{*}) - \nabla_{v,-v}^{2}\ell_{n}(\bbeta^{*})$,
\begin{eqnarray*}
\Delta_{2}&=\nabla_{v,-v}^{2}\ell_{n}(\bbeta^{*})-\bw^{*T}J_{-v,-v}^{*}, \quad  \;\; \quad  \quad \quad  \quad   \quad  \Delta_{3}&=\bw^{*T}\bigl(J_{-v,-v}^{*}-\nabla_{-v,-v}^{2}\ell_{n}(\bbeta^{*})\bigr), \\ \Delta_{4}&=\bw^{*T}\bigl(\nabla_{-v,-v}^{2}\ell_{n}(\bbeta^{*})-\nabla_{-v,-v}^{2}\ell_{n}(\bbeta^{*}_{\alpha})\bigr), \quad \quad \Delta_{5}&=(\bw^{*T}-\widehat{\bw}^{T})\nabla_{-v,-v}^{2}\ell_{n}(\bbeta^{*}_{\alpha}).
\end{eqnarray*}
We have the following bounds
\begin{eqnarray*}
|\Delta_{1}| &=& \Bigl|\frac{1}{n}\sum_{i=1}^{n}\bZ_{i}\bZ_{i}^{T}(\widehat{\bbeta}_{-v}^{\lambda}-\bbeta_{-v}^{*})\bigl(\ell_{i}''(\bX_{i}^{T}\bbeta_{\alpha}^{*})-\ell_{i}''(\bX_{i}^{T}\bbeta^{*})\bigr)\Bigr| \\
&\leq & \max_{1\leq i\leq n}K_{i} \max_{1\leq i \leq n} \|\bX_{i}\|_{\infty}\bigl\|\frac{1}{n}\bZ(\widehat{\bbeta}_{-v} - \bbeta_{-v}^{*})\bigr\|_{2}^{2},
\end{eqnarray*}
$|\Delta_{2}| \leq \bigl\|\nabla_{v,-v}^{2}\ell_{n}(\bbeta^{*}) - J^{*}_{v,-v}\bigr\|_{\infty}\|\widehat{\bbeta}_{-v}^{\lambda} - \bbeta_{-v}^{*}\|_{1}$,
$|\Delta_{3}| \leq \|\bw\|_{1} \bigl\| J^{*}_{-v,-v} - \nabla_{-v,-v}^{2}\ell_{n}(\bbeta^{*})\bigr\|_{\max}\|\widehat{\bbeta}_{-v}^{\lambda} - \bbeta_{v}^{*}\|_{1}$,
\begin{eqnarray*}
|\Delta_{4}| &= & \bigl|\bw^{*T}\bigl(\nabla_{-v,-v}^{2}\ell_{n}(\bbeta^{*}) - \nabla_{-v,-v}^{2}\ell_{n}(\bbeta^{*}_{v})\bigr)(\widehat{\bgamma}^{\lambda}-\blambda^{*})\bigr|\\
				 &\leq &  \max_{1\leq i\leq n}K_{i}\|\bw^{*}\|_{1} \bigl\|\frac{1}{n}\bZ(\widehat{\bbeta}_{-v}^{\lambda} - \bbeta_{-v}^{*})\bigr\|_{2}^{2},
\end{eqnarray*}
and $|\Delta_{5}|\leq \|\bw^{*}-\widehat{\bw}\|_{1}\bigl\|\nabla_{-v,-v}\ell_{n}(\bbeta_{v}^{*})\bigr\|_{\max}\|\widehat{\bbeta}_{-v}^{\lambda}-\bbeta_{-v}^{*}\|_{1}$. Let $\epsilon=\delta/5$. Then by Condition \ref{con:glm} and Lemma \ref{lemmaA3}
\[
\PP\Bigl(|\Delta_{1}| \gtrsim s\frac{\log(d/\epsilon)}{n}\Bigr) < \epsilon \quad \text{ and } \quad \PP\Bigl(|\Delta_{4}| \gtrsim s s_{1}\frac{\log(d/\epsilon)}{n}\Bigr) < \epsilon.
\]
Noting the $\bbeta^{*}$ itself satisfies the requirements on $\widetilde{\bbeta}$ in (B5), Lemma \ref{lemmaCheckB5B6} and Condition \ref{con:estimation} together give
\[
\PP\Bigl(|\Delta_{2}| \gtrsim s_{1}\frac{\log(d/\epsilon)}{n}\Bigr) < \epsilon \quad \text{ and } \quad \PP\Bigl(|\Delta_{3}| \gtrsim s_{1}s\frac{\log(d/\epsilon)}{n}\Bigr) < \epsilon.
\]
By (B1) verified above and noting that
\[
\bigl\|\nabla_{-v,-v}\ell_{n}(\bbeta_{v}^{*})\bigr\|_{\max} \leq \bigl\|\nabla_{-v,-v}\ell_{n}(\bbeta_{v}^{*}) - \nabla_{-v,-v}\ell_{n}(\bbeta^{*})\bigr\|_{\max} + \bigl\|\nabla_{-v,-v}\ell_{n}(\bbeta^{*})\bigr\|_{\max},
\]
the proof of Lemma \ref{lemmaCheckB5B6} delivers $\PP\Bigl(|\Delta_{5}| \gtrsim s_{1} s\log(d/\epsilon)/n\Bigr) < \epsilon$. Combining the bounds, we finally have
\[
\PP\Bigl(\sup_{\alpha\in [0,1]} \Bigl| \bigl(\nabla^2_{v,-v} \ell_{n}(\beta_{v}^{*}, \bbeta_{-v,\alpha}) - \hat{\bw}^T \nabla^2_{-v,-v} \ell_{n}(\beta_{v}^{*}, \bbeta_{-v,\alpha})\bigr)(\widehat{\bbeta}_{-v}^{\lambda}-\bbeta_{-v}^{*})\Bigr| \gtrsim s_{1} s \frac{\log (d/\delta)}{n}\Bigr) < \delta
\]

\noindent \textbf{Verification of (B4)}. See \citet{NingLiu2014b}, proof of Lemma C.2.
\end{proof}

In the following lemma, we verify (B5) under the same conditions.

\begin{lemma}\label{lemmaCheckB5B6}
Under Conditions \ref{con:glm} and \ref{con:estimation}, (B5) of Condition \ref{conditionScoreLR} is fulfilled.
\end{lemma}

\begin{proof}
We obtain a tail probability bound for $\Delta_{1}$ and $\Delta_{2}$ in the decomposition
\[
\|\nab^{2}\ell_{n}(\widetilde{\bbeta}) - J^{*}\|_{\max} \leq \|\nab^{2}\ell_{n}(\widetilde{\bbeta}) - \nab^{2}\ell_{n}(\bbeta^{*})\|_{\max}+ \|\nab^{2}\ell_{n}(\bbeta^{*}) - J^{*}\|_{\max}=\Delta_{1} + \Delta_{2}.
\]
For the control over $\Delta_{1}$, note that by Condition \ref{con:glm} (ii) and (iii),
\[
\bigl|[\nab^{2}\ell_{n}(\bbeta^{*})]_{jk}\bigr| \leq \bigl|b''(\bX_{i}^{T}\bbeta^{*})\bigr|\bigl|X_{ij}X_{ik}\bigr| \leq U_2M^{2}.
\]
Hence Hoeffding's inequality and the union bound deliver
\begin{equation}\label{eqB5Delta2}
\PP(\Delta_{2}>t)=\PP\Bigl(\|\nab^{2}\ell_{n}(\bbeta^{*}) - J^{*}\|_{\max}>t\Bigr) \leq 2d^{2} \exp\Bigl\{-\frac{nt^{2}}{8U_2^{2}M^{4}}\Bigr\}.
\end{equation}
For the control over $\Delta_{1}$, we have by Lemma \ref{lem:lipshitz},
\begin{eqnarray*}
\bigl|[\nab^{2}\ell_{n}(\widetilde{\bbeta})-\nab^{2}\ell_{n}(\bbeta^{*})]_{jk}\bigr| &=& \bigl| \bigl(b''(\bX_{i}^{T}\widetilde{\bbeta}) - b''(\bX_{i}^{T}\bbeta^{*})\bigr)X_{ij}X_{ik} \bigr| \\
&\leq & M^3 U_3\|\widetilde{\bbeta} - \bbeta^{*}\|_{1} \leq M^{3} U_3 s\sqrt{n^{-1}\log(d/\delta)}
\end{eqnarray*}
with probability $>1-\delta$. Hoeffding's inequality and the union bound again deliver
\begin{equation}\label{eqB5Delta1}
\PP(\Delta_{1}>t) = \PP\Bigl(\|\nab_{\eta\eta}^{2}\ell_{n}(\widetilde{\bbeta}) - \nab_{\eta\eta}^{2}\ell_{n}(\bbeta^{*})\|_{\max}>t\Bigr) \leq 2d^{2}\exp\Bigl\{-\frac{n^{2}t^{2}}{8U_3^{2}M^{6}s^{2}\log(d/\delta)}\Bigr\}.
\end{equation}
Combining the bounds from equations \eqref{eqB5Delta2} and \eqref{eqB5Delta1} we have
\[
\PP\Bigl(\|\nab^{2}\ell(\widetilde{\bbeta}) - J^{*}\|_{\max} >t\Bigr) \leq 2d^{2}\Bigl(\exp\Bigl\{-\frac{nt^{2}}{8U_3^{2}M^{4}}\Bigr\} +\exp\Bigl\{-\frac{n^{2}t^{2}}{8U_3^{2}M^{6}s^{2}\log(d/\delta)}\Bigr\}\Bigr).
\]
Setting each term equal to $\delta/2$, solving for $t$ and ignoring the relative magnitude of constants, we have $t=U_3\max\bigl\{n^{-1}s\log(d/\delta),n^{-1/2}\sqrt{\log(d/\delta)}\bigr\}=U_3n^{-1/2}\log(d/\delta)$, thus verifying (B5).
\end{proof}

\begin{lemma}\label{lemmaLogLikelihoodDerivConvRates}
For each $j\in\{1,\ldots,k\}$, let $\bbeta_{-v, \alpha_{j}}=\alpha_{j}\hat{\bbeta}_{-v}^{\lambda}(\cD_{j})+(1-\alpha_{j})\bbeta_{-v}^{*}$, for some $\alpha_{j}\in[0,1]$, where $\hat{\bbeta}_{-v}^{\lambda}(\cD_{j})$ is defined in equation \eqref{eqMEstimator}. Define
\begin{eqnarray*}
\Delta_{1}\sj &=&(\hat{\bw}(\cD_{j})-\bw^{*})^{T}\nab_{-v}\ell_{n_{k}}\sj (\beta_{v}^{*},\bbeta_{-v}^{*})\\
\Delta_{2}\sj &=& \bigl(\nabla_{v,-v}^{2}\ell_{n_{k}}\sj(\beta_{v}^{*},\bbeta_{-v,\alpha_{j}})-\hat{\bw}^{T}\nabla_{-v,-v}\ell_{n_{k}}\sj (\beta_{v}^{*},\bbeta_{-v,\alpha_{j}})\bigr)(\hat{\bbeta}_{-v}^{\lambda}-\bbeta_{-v}^{*}).
\end{eqnarray*}
Under (B1) - (B3) of Condition \ref{conditionScoreLR}, $\Bigl|k^{-1}\sum_{j=1}^{k} \Delta_{1}\sj\Bigr|=o_{\PP}\bigl(n^{-1/2}\bigr)$ and $\Bigl|k^{-1}\sum_{j=1}^{k} \Delta_{2}\sj\Bigr|=o_{\PP}\bigl(n^{-1/2}\bigr)$ whenever $k \ll d$ is chosen to satisfy $k=o\bigl((s_{1}\log d)^{-1}\sqrt{n}\bigr)$.
\end{lemma}

\begin{proof}
By H\"older's inequality,
\[
\bigl|\Delta_{1}\sj\bigr| = \bigl|\bigl(\bw^*- \hat{\bw}(\cD_{j})\bigr)^T \nabla_{-v} \ell_{n_{k}}\sj(\beta_{v}^{*}, \bbeta_{-v}^*)\bigr|\le \|\hat{\bw}(\cD_{j})-\bw^*\|_{1} \|\nabla_{-v} \ell_{n_{k}}\sj(\beta_{v}^{*}, \bbeta_{-v}^*)\|_{\infty},
\]
hence, for any $t$,
\[
\bigl\{\bigl|\Delta_{1}\sj\bigr|>t\bigr\}\subseteq\bigl\{\|\hat{\bw}(\cD_{j})-\bw^*\|_{1} \|\nabla_{-v} \ell_{n_{k}}\sj(\beta_{v}^{*}, \bbeta_{-v}^*)\|_{\infty}>t\bigr\}.
\]
Taking $t=v q$ where $v=C n^{-1/2}s_{1}\sqrt{k\log(d/\delta)}$ and $q=C n^{-1/2}\sqrt{k\log(d/\delta)}$, we have
\begin{eqnarray*}
& & \PP\Bigl(\bigl\{\|\hat{\bw}(\cD_{j})-\bw^*\|_{1} \|\nabla_{-v} \ell_{n_{k}}\sj(\beta_{v}^{*}, \bbeta_{-v}^*)\|_{\infty}>v q \bigr\} \Bigr) \\
&=& \PP\Bigl(\bigl\{\|\hat{\bw}(\cD_{j})-\bw^*\|_{1} \|\nabla_{-v} \ell_{n_{k}}\sj(\beta_{v}^{*}, \bbeta_{-v}^*)\|_{\infty}>v q \bigr\} \cap \Bigl\{\frac{\|\hat{\bw}(\cD_{j})-\bw^*\|_{1}}{v}\leq 1\Bigr\}\Bigr) \\
& +  & \PP\Bigl(\bigl\{\|\hat{\bw}(\cD_{j})-\bw^*\|_{1} \|\nabla_{-v} \ell_{n_{k}}\sj(\beta_{v}^{*}, \bbeta_{-v}^*)\|_{\infty}>v q\bigr\} \cap \Bigl\{\frac{\|\hat{\bw}(\cD_{j})-\bw^*\|_{1}}{v}> 1\Bigr\}  \Bigr) \leq  2\delta
\end{eqnarray*}
by (B1) and (B2) of Condition \ref{conditionScoreLR}. Hence the union bound delivers
\[
\PP\Bigl(\bigl|\sum_{j=1}^{k}\Delta_{1}\sj\bigr|> k v q\Bigr) \leq   \PP\Bigl(\cup_{j=1}^{k}\bigl\{\bigl|\Delta_{1}\sj\bigr|> v q\bigr\}\Bigr) \leq \sum_{j=1}^{k} \PP\Bigl(\bigl|\Delta_{1}\sj\bigr|> v q \Bigr)  \leq  2k\delta=o(1)
\]
for $\delta=o(k^{-1})$. Taking $\delta=k^{-1}$ for $\alpha>0$ arbitrarily small in the definition of $v$ and $q$, the requirement is $k s_{1}\log d = o\bigl(\sqrt{n}\bigr)$ and $k s_{1} \log k =o(\sqrt{n})$ for $\alpha>0$ arbitrarily small. Since $k\ll d$, $k^{-1}\sum_{j=1}^{k}  \Delta_{1}\sj=o_{\PP}\bigl(n^{-1/2}\bigr)$ with $k=o\bigl((s_{1}\log d)^{-1}\sqrt{n}\bigr)$. Next, consider
\[
\bigl|\Delta_{2}\sj\bigr| \leq  \sup_{\alpha\in [0,1]} \Bigl| \bigl(\nabla^2_{v,-v} \ell_{n_{k}}^{(j)}(\beta_{v}^{*}, \bbeta_{-v,\alpha}) - \hat{\bw}^T \nabla^2_{-v,-v} \ell_{n_{k}}^{(j)}(\beta_{v}^{*}, \bbeta_{-v,\alpha})\bigr)(\widehat{\bbeta}_{-v}^{\lambda}(\cD_{j})-\bbeta_{-v}^{*})\Bigr|.
\]
By (B3) of Condition \ref{conditionScoreLR}, $\PP\bigl(\bigl|\Delta_{2}\sj\bigr|\geq t\bigr)<\delta$ for $t\asymp s_{1} s n^{-1}k\log(d/\delta)$, hence, proceeding in an analogous fashion to in the control over $k^{-1} \sum_{j=1}^{k}  \Delta_{1}\sj$, we obtain
\[
\PP\Bigl(\Bigl|\sum_{j=1}^{k}\Delta_{2}\sj\Bigr|> k t\Bigr) \leq  \PP\left(\cup_{j=1}^{k}\bigl|\Delta_{2}\sj\bigr|> t\right)\leq \sum_{j=1}^{k} \PP\left(\bigl|\Delta_{2}\sj\bigr|> t\right) \leq k\delta=o(1)
\]
for $\delta=o(k^{-1})$. Hence $k^{-1}\sum_{j=1}^{k}  \Delta_{2}\sj=o_{\PP}\bigl(n^{-1/2}\bigr)$ with $k=o\bigl((s_{1}s \log d)^{-1}n^{3/2}\bigr)$. Since $(s_{1}\log d)^{-1}\sqrt{n}=o\bigl((s_{1}s \log d)^{-1} n^{3/2}\bigr)$, $k^{-1} \sum_{j=1}^{k}  \bigl(\Delta_{1}\sj + \Delta_{2}\sj\bigr)=o_{\PP}\bigl(n^{-1/2}\bigr)$ requires $k=o\bigl((s_{1}\log d)^{-1}\sqrt{n}\bigr)$.
\end{proof}

\begin{lemma}\label{lemmaDistributionAggScore}
Under (B1) - (B4)  of Condition \ref{conditionScoreLR}, with $k \ll d$ chosen to satisfy the scaling $k=o\bigl(((s\vee s_{1})\log d)^{-1}\sqrt{n}\bigr)$,
\begin{align*}
& \frac{1}{k}\sum_{j=1}^{k} \hat{S}\sj (\beta_{v}^{*}, \gOj) = \frac{1}{k}\sum_{j=1}^{k} S\sj (\beta_{v}^{*}, \bbeta_{-v}^{*}) + o_{\PP}(n^{-1/2}) \; \text{ and } \;\\
&\lim_{n \rightarrow \infty}\sup_t |\PP((J_{v|-v}^{*})^{-1/2}\sqrt{n}\frac{1}{k}\sum_{j=1}^{k}S\sj(\beta_{v}^{*},\bbeta_{-v}^*)<t) - \Phi(t)| \rightarrow 0.
\end{align*}
\end{lemma}

\begin{proof}
Recall
\[
S\sj(\beta_{v}^{*},\bbeta_{-v}^*)=\nabla_{v} \ell_{n_{k}}\sj(\beta_{v}^{*}, \bbeta_{-v}^*) - \bw^{*T} \nabla_{-v} \ell_{n_{k}}\sj(\beta_{v}^{*}, \bbeta_{-v}^*).
\]
Through a mean value expansion of $\hat{S}\sj(\beta_{v}^{*}, \hat{\bbeta}_{-v}^{\lambda}(\cD_{j}))$ around $\bbeta_{-v}^{*}$, we have for each $j\in\{1,\ldots, k\}$,
\begin{align}
\nonumber \hat{S}\sj\bigl(\beta_{v}^{*}, \hat{\bbeta}_{-v}^{\lambda}(\cD_{j})\bigr) &=  \nabla_{v} \ell_{n_{k}}\sj\bigl(\beta_{v}^{*}, \hat{\bbeta}_{-v}^{\lambda}(\cD_{j})\bigr) - \hat{\bw}(\cD_{j})^T \nabla_{-v} \ell_{n_{k}}\sj\bigl(\beta_{v}^{*}, \hat{\bbeta}_{-v}^{\lambda}(\cD_{j})\bigr) \\
\nonumber&= S\sj(\beta_{v}^{*},\bbeta_{-v}^*) + \Delta_{1}\sj + \Delta_{2}\sj,
\end{align}
for some $\bbeta_{-v,\alpha}=\alpha\hat{\bbeta}_{-v}(\cD_{j})+(1-\alpha)\bbeta_{-v}^{*}$, where
\begin{eqnarray*}
\Delta_{1}\sj &=&\bigl(\bw^*- \hat{\bw}(\cD_{j})\bigr)^T \nabla_{-v} \ell_{n_{k}}\sj(\beta_{v}^{*}, \bbeta_{-v}^*) \\
	\Delta_{2}\sj &=&\Big[\nabla^2_{v,-v} \ell_{n_{k}}\sj(\beta_{v}^{*}, \bbeta_{-v,\alpha}) - \hat{\bw}(\cD_{j})^T \nabla^2_{-v,-v} \ell_{n_{k}}\sj(\beta_{v}^{*}, \bbeta_{-v,\alpha})\Big](\hat{\bbeta}_{-v}^{\lambda}(\cD_{j}) - \bbeta_{-v}^*).
\end{eqnarray*}
It follows that
\begin{equation}\label{scoreDecomp}
\frac{1}{k}\sum_{j=1}^{k} \hat{S}\sj\bigl(\beta_{v}^{*}, \hat{\bbeta}_{-v}^{\lambda}(\cD_{j})\bigr) = \frac{1}{k} \sum_{j = 1}^k S\sj(\beta_{v}^{*},\bbeta_{-v}^*) + \frac{1}{k} \sum_{j=1}^{k}  \bigl(\Delta_{1}\sj + \Delta_{2}\sj\bigr) = \frac{1}{k} \sum_{j = 1}^k S\sj(\theta^{*},\bgamma^*) + o_{\PP}(n^{-1/2})
\end{equation}
by Lemma \ref{lemmaLogLikelihoodDerivConvRates} whenever $k=o\bigl((s_{1}\log d)^{-1}\sqrt{n}\bigr)$. Observe
\[
\sqrt{n}\Bigl( k^{-1} \sum_{j = 1}^k S\sj(\beta_{v}^{*},\bbeta_{-v}^*)\Bigr)=\sqrt{n}(1,-\bw^{*T})\Bigl(\frac{1}{k}\sum_{j=1}^{k} \nab\ell_{n_{k}}\sj(\beta_{v}^{*},\bbeta_{-v}^*) \Bigr) \quad \text{ and }
\]
\[
J^{*}_{v|-v}=(1,-\bw^{*T})J^{*}(1,-\bw^{*T})^{T}.
\]
So $\sqrt{n}\frac{1}{k}\sum_{j=1}^{k}S\sj(\beta_{v}^{*},\bbeta_{-v}^*)\leadsto N(0,J^{*}_{v|-v})$
by Condition (B4). Similar to Corollary \ref{corollWaldvdGDist}, we apply the Berry-Essen inequality to show that $\sup_t |\PP(\sqrt{n}\frac{1}{k}\sum_{j=1}^{k}S\sj(\beta_{v}^{*},\bbeta_{-v}^*)<t) - \Phi(t)| \rightarrow 0$.
\end{proof}

\begin{lemma}\label{lemmaConvWBarGammaBar}
Under Condition (B1), for any $\delta\in (0,1)$,
\[
\PP\Bigl(\|\overline{\bw} - \bw^{*}\|_{1} >C n^{-1/2}s_{1}\sqrt{k \log(d/\delta)} \Bigr) < k\delta \;\; \text{and} \;\; \PP\Bigl(\|\overline{\bbeta}_{-v} - \bbeta_{-v}^{*}\|_{1} >C n^{-1/2}s\sqrt{k \log(d/\delta)} \Bigr) < k\delta.
\]
\end{lemma}

\begin{proof} Set $t=C s_{1}\sqrt{n^{-1}(k \log(d/\delta))}$ and note
\[
\PP\bigl(\|\sum_{j=1}^{k}(\widehat{\bw}(\cD_{j}) - \bw^{*})\|_{1} > k t \bigr)\leq \PP\bigl(\cup_{j=1}^{k}\|\widehat{\bw}(\cD_{j}) - \bw^{*}\|_{1} > t \bigr) \leq \sum_{j=1}^{k} \PP\bigl(\|\overline{\bw} - \bw^{*}\|_{1} > t \bigr)
\]
by the union bound. Then by Condition (B1), $\PP\Bigl(\|\overline{\bw} - \bw^{*}\|_{1} >C n^{-1/2}s_{1}\sqrt{k \log(d/\delta)} \Bigr) < k\delta$. The proof of the second bound is analogous, setting $t=C s \sqrt{n^{-1}(k \log(d/\delta))}$.
\end{proof}

\begin{lemma}\label{lemmaAvgNabNab}
Suppose (B5) of Condition \ref{conditionScoreLR} is satisfied. For any $\delta$, if there exists an estimator $\widetilde{\bbeta}=(\widetilde{\beta}_{v}^{T},\widetilde{\bbeta}_{-v}^{T})^{T}$ satisfying $\|\widetilde{\bbeta}-\bbeta^{*}\|_{1}\leq Cs\sqrt{n^{-1}\log (d/\delta)}$ with probability $1-\delta$, then
\[
\PP\Bigl(\Bigl\|\frac{1}{k}\sum_{j=1}^{k} \nab^{2}\ell_{n_{k}}\sj(\widetilde{\bbeta})-J^{*}\Bigr\|_{\max} >Cn^{-1/2}\sqrt{k\log (d/\delta)} \Bigr) < k\delta.
\]
\end{lemma}

\begin{proof} The proof follows from (B5) in Condition \ref{conditionScoreLR} via an analogous argument to that of Lemma \ref{lemmaConvWBarGammaBar}, taking $t=C \sqrt{n^{-1}(k \log(d/\delta))}$.
\end{proof}

\begin{lemma}\label{lemmaPartialInfoEstimator}
Suppose (B1)-(B5) of Condition \ref{conditionScoreLR} are fulfilled. Then for any $k\ll d$ satisfying $k=o\bigl(((s\vee s_{1})\log d)^{-1}\sqrt{n}\bigr)$, $| \overline{J}_{\theta|\bgamma}-J^{*}_{v|-v}|=o_{\PP}(1)$.
\end{lemma}

\begin{proof}
Recall that $J^{*}_{v|-v}=J_{v,v}^{*}-\bJ^{*}_{v,-v}J^{*-1}_{-v,-v}\bJ_{-v,v}^{*}$ and
\[
 \overline{J}_{v|-v} = \frac{1}{k}\sum_{j=1}^{k}\bigl(\nab_{v,v}\ell_{n_{k}}\sj(\overline{\beta}_{v}^{d},\overline{\bbeta}_{-v}) - \overline{w}^{T} \nabla_{-v,v}^{2}\ell_{n_{k}}\sj (\overline{\beta}_{v}^{d},\overline{\bbeta}_{-v}), \text{ so }
\]
\[
\bigl| \overline{J}_{v|-v} - J^{*}_{v|-v}\bigr| =\underbrace{\bigl| \frac{1}{k}\sum_{j=1}^{k} \nab_{v,v}\ell_{n_{k}}\sj(\overline{\beta}_{v}^{d},\overline{\bbeta}_{-v})-J_{v,v}^{*}\bigr|}_{\Delta_{1}} + \underbrace{\bigl|\overline{\bw}^{T}\bigl(\frac{1}{k}\sum_{j=1}^{k} \nabla_{-v,v}^{2}\ell_{n_{k}}\sj (\overline{\beta}_{v}^{d},\overline{\bbeta}_{-v}) - \bw^{*T}\bJ^{*}_{-v,v} \bigr)\bigr|}_{\Delta_{2}}.
\]
Let $\widetilde{\bbeta}=(\overline{\beta}_{v}^{d},\overline{\bbeta}_{-v})$ and note that $\|\widetilde{\bbeta}-\bbeta^{*}\|_{1}$ satisfies the clause in (B5) of Condition \ref{conditionScoreLR} by Lemma \ref{lemmaConvWBarGammaBar} when $k=o\bigl(((s\vee s_{1})\log d)^{-1}\sqrt{n}\bigr)$. Hence $\Delta_{1}=o_{\PP}(1)$ by Lemma \ref{lemmaAvgNabNab}.
\begin{eqnarray*}
\Delta_{2}&\leq &\underbrace{\Bigl|(\overline{\bw}-\bw^{*})^{T}\Bigl(\frac{1}{k}\sum_{j=1}^{k}\nabla_{-v,v}^{2}\ell_{n_{k}}\sj(\overline{\beta}_{v}^{d},\overline{\bbeta}_{-v}) -\bJ^{*}_{-v,v}\Bigr)\Bigr|}_{\Delta_{21}} \\
& & \quad \quad + \;\; \underbrace{\bigl|(\overline{\bw} - \bw^{*})^{T}\bJ_{-v,v}^{*}\bigr|}_{\Delta_{22}} + \underbrace{\Bigl|\bw^{*T}\Bigl(\frac{1}{k}\sum_{j=1}^{k} \nabla_{-v,v}^{2}\ell_{n_{k}}\sj (\overline{\beta}_{v}^{d},\overline{\bbeta}_{-v}) - \bJ^{*}_{-v,v} \Bigr)\Bigr|}_{\Delta_{23}}.
\end{eqnarray*}
By the fact that $\|J^{*}\|_{\max}<\infty$ and $\|\bw^{*}\|_{1}\leq C s_{1}$ by (B1) of Condition \ref{conditionScoreLR}, an application of Lemmas \ref{lemmaConvWBarGammaBar} and \ref{lemmaAvgNabNab} delivers
\begin{eqnarray*}
\Delta_{21}&\leq & \|\overline{\bw}-\bw^{*}\|_{1}\bigl\|\frac{1}{k}\sum_{j=1}^{k}\nabla_{-v,v}^{2}\ell_{n_{k}}\sj(\overline{\beta}_{v}^{d},\overline{\bbeta}_{-v}) -\bJ^{*}_{-v,v}\bigr\|_{\infty} = o_{\PP}(1),\\
\Delta_{22}&\leq & \|\overline{\bw} - \bw^{*}\|_{1}\|\bJ_{-v,v}^{*}\|_{\infty}=o_{\PP}(1),\\
\Delta_{23}&\leq & \bigl\|\frac{1}{k}\sum_{j=1}^{k} \nabla_{-v,v}^{2}\ell_{n_{k}}\sj (\overline{\beta}_{v}^{d},\overline{\bbeta}_{-v}) - \bJ^{*}_{-v,v}\bigr\|_{\infty} \|\bw^{*}\|_{1} =o_{\PP}(1)
\end{eqnarray*}
for $k=o\bigl((s_{1}\log d)^{-1}n\bigr)$, a fortiori for $k=o\bigl(((s\vee s_{1})\log d)^{-1}\sqrt{n}\bigr)$. Hence $\bigl| \overline{J}_{v|-v} - J^{*}_{v|-v}\bigr|=o_{\PP}(1)$.
\end{proof}

\section{Auxiliary Lemmas for Estimation}\label{sec:est-apen}

In this section, we provide the proofs of the technical lemmas for the divide and conquer estimation.

\begin{lemma}
		\label{lem:cctts}
		Suppose ${X}$ is a $n\times d$ matrix that has independent sub-gaussian rows $\{\bX_i\}_{i=1}^n$. Denote $\EE(\bX_i\bX_i^T)$ by ${\Sigma}$, then we have
	\[
		\label{lm2.1}
		\PP\left(\|\frac{1}{n}{X^TX}-{\Sigma_X}\|_2\ge (\delta \vee \delta^2)\right)\le \exp(-c_1t^2),
	\]
	where $t\ge 0$, $\delta=C_1\sqrt{d/n}+t/\sqrt n$ and $C_1$ and $c_1$ are both constants depending only on $\|\bX_i\|_{\psi_2}$.
\end{lemma}
\begin{proof}
	See \cite{vershynin2010introduction}.
\end{proof}
		
\begin{lemma}
	\label{lem:bsineq}
	(Bernstein-type inequality) Let $X_1,\ldots,X_n$ be independent centered sub-exponential random variables, and $M=\max\limits_{1\leq i\leq n} \|X_i\|_{\psi_1}$. Then for every $a=(a_1,\ldots,a_n)\in\RR^n$ and every $t\ge 0$, we have
	\[
		\label{lm2.2}
		\PP\left(\sum\limits_{i=1}^n a_iX_i\geq t\right)\leq \exp\left(-C_2\min\left(\frac{t^2}{M^2\|a\|_2^2},\frac{t}{M\|a\|_\infty}\right)\right).
	\]
\end{lemma}
\begin{proof}
	See \cite{vershynin2010introduction}.
\end{proof}
		
\begin{lemma}
		\label{lem:maxeg}
		Suppose ${X}$ is a $n\times d$ matrix that has independent sub-gaussian rows $\{\xb_i\}_{i=1}^n$. If $\lambda_{\max}({\Sigma})\le C_{\max}$ and $d\ll n$, then for all $M>C_{\max}$, there exists a constant $c>0$ such that when $n$ and $d$ are sufficiently large,
		\[
			\PP\left(\Bigl\|\frac{1}{n}X^TX\Bigr\|_2\ge M\right)\le \exp(-cn).
		\]
\end{lemma}
\begin{proof}
		Apply Lemma \ref{lem:cctts} with $t=\sqrt{cn/ c_1}$, where $(\sqrt{c/c_1} \vee  c/c_1)<M-C_{\max}$, and it follows that
		\[
			\PP\left(\Bigl\|\frac{1}{n}X^TX-\Sigma\Bigr\|_2\ge (\delta \vee \delta^2)\right)\le \exp(-cn).
		\]
		 Since $d\ll n$, we obtain $(\delta\vee \delta^2)\to \sqrt{c/c_1}$, which completes the proof.
\end{proof}
	
	\begin{lemma}
		\label{lem:mineg}
		Suppose ${X}$ is a $n\times d$ matrix that has independent sub-gaussian rows $\{\bX_i\}_{i=1}^n$. $\EE{\bX}_i={\bf 0}$, $\lambda_{\min}({\Sigma})\geq C_{\min}>0$ and $d\ll n$. For all $m<C_{\min}$, there exists a constant $c>0$ such that when $n$ and $d$ are sufficiently large,
	\[
		\PP\left(\Bigl\|\Bigl(\frac{1}{n}{X^TX}\Bigr)^{-1}\Bigr\|_2\ge \frac{1}{m}\right)=\PP\left(\lambda_{\min}\Bigl(\frac{1}{n}X^TX\Bigr)\le m\right)\le \exp(-cn).
	\]
	\end{lemma}
	\begin{proof}
		It is easy to check the following inequality. For any two symmetric and semi-definite $d\times d$ matrices ${A}$ and ${B}$, we have
	\[
		\lambda_{\min}({A})\geq\lambda_{\min}({B})-\|{A-B}\|_2\ ,
	\]
	because for any vector $\xb$ satisfying $\|\xb\|_2=1$, we have $\|{ A \bx}\|_2=\|{B\bx+(A-B)\bx}\|_2\geq \|{B\bx}\|_2-\|{(A-B)\bx}\|_2\ge \lambda_{\min}({B})-\|{A-B}\|_2$. Then it follows that
		\begin{eqnarray*}
	\PP\left(\Bigl\|\Bigl(\frac{1}{n}{X^TX}\Bigr)^{-1}\Bigr\|_2\ge \frac{1}{m}\right)&=&\PP\left(\lambda_{\min}\Bigl(\frac{1}{n}{X^TX}\Bigr)\le m\right) \le \PP\left(\lambda_{\min}({\Sigma})-\Bigl\|\frac{1}{n}{X^TX}-{\Sigma}\Bigr\|_2 \ge m\right)\\
		&  \le & \PP\left(\Bigl\|\frac{1}{n}{X^TX}-{\Sigma_X}\Bigr\|_2\ge C_{\min}-m\right)\le \exp(-cn),
		\end{eqnarray*}
	where $c$ satisfies $(\sqrt{c/c_1}\vee c/c_1)<C_{\min}-m$ and the last inequality is an application of Lemma \ref{lem:cctts} with $t=\sqrt{cn/c_1}$.
	\end{proof}
	
	\begin{lemma}
		\label{lem:hineq}
	(Hoeffding-type Inequality).\quad{Let $X_1$,\ldots,$X_n$ be independent centered sub-gaussian random variables, and let $K=\max\limits_i\|X_i\|_{\psi_2}$. Then for every $a=(a_1,\ldots,a_n)\in\RR^n$ and every $t>0$, we have}
	\[
		\PP\left(\left|\sum\limits_{i=1}^na_iX_i\right|\geq t\right)\leq e\cdot\exp\left(-\frac{ct^2}{K^2\|a\|_2^2}\right).
	\]
	\end{lemma}
	
	\begin{lemma}
		\label{lem:psi1}
	(Sub-exponential is sub-gaussian squared).\quad {A random variable $X$ is a sub-gaussian if and only if $X^2$ is sub-exponential. Moreover,}
	\[
		\|X\|_{\psi_2}^2\leq\|X^2\|_{\psi_1}\leq2\|X\|_{\psi_2}^2.
	\]
	\end{lemma}
	
	\begin{lemma}
		\label{lem:bd2m}
		Let $X_1$,\ldots,$X_n$ be independent centered sub-gaussian random variables. Let $\kappa=\max_i\|X_i\|_{\psi_2}$ and $\sigma^2=\max_i \EE X_i^2$. Suppose $\sigma^2>1$, then we have
		\[
			\PP\left(\frac{1}{n}\sum\limits_{i=1}^n X_i^2>2\sigma^2\right)\le\exp\left(-C_2\frac{\sigma^2n}{\kappa^2}\right).
		\]
	\end{lemma}
	
	\begin{proof}
		Combining Lemma \ref{lem:bsineq} and Lemma \ref{lem:psi1} yields the result.
	\end{proof}
	
	\begin{lemma}
		\label{lem:olsmgf}
		Following the same notation as in the beginning of Proof of Theorem \ref{thm:ols2},
		\[
			\PP\left(\left\{\bigl\|\frac{1}{k}\sum\limits_{j=1}^k({X}^{(j)}{D}_1^{(j)})^T\bepsilon^{(j)}/n_k\bigr\|_2>t/2\right\}\cap \cE_0\right) \le  \exp\left(d\log(6)-\frac{t^2C_{\min}^3n}{32C_3s_1^2(\delta_1 \vee\delta_1^2)^2}\right)
		\]
		and
		\[	
			\PP\left(\left\{\|({X}{D}_2)^T\bepsilon/n\|_2>t/2\right\}\cap \cE\right)\le\exp\left(d\log(6)-\frac{t^2C_{\min}^3n}{32C_3s_1^2 (\delta_2 \vee \delta_2^2)^2}\right).
		\]
	\end{lemma}
	
	\begin{proof}
		\begin{equation}
			\label{eq:mgf}
			\begin{aligned}
				\EE\left(\exp\left(\lambda({{D}_1^{(j)}\bf v})^T({X^{(j)T}}{\bepsilon}^{(j)}/n_k)\right)\given{X}^{(j)}\right)&=\prod\limits_{i=1}^{n_k}\EE\left(\exp\left((\lambda\bX^{(j)}_i/n_k)^T({D}^{(j)}\vb)\epsilon_i\right)\given{X}^{(j)}\right)\\
				&\leq\exp\left(C_3\lambda^2s_1^2\sum\limits_{i=1}^{n} (A^{(j)}_i)^2/n_k^2\right),
			\end{aligned}
		\end{equation}
		\begin{equation}
			\label{eq:mgf2}
			\begin{aligned}
				\EE\left(\exp\left(\lambda({{D}_2\bf v})^T({X}^T{\bepsilon}/n)\right)\given{X}\right)&=\prod\limits_{i=1}^{N}\EE\left(\exp\left((\lambda\bX_i/N)^T({D}_2\vb)\epsilon_i\right)\mid{X}\right)\\
				&\leq\exp\left(C_3\lambda^2s_1^2\sum\limits_{i=1}^{N} A_i^2/n^2\right),
			\end{aligned}
		\end{equation}
		where we write $A^{(j)}_i$ and $A_i$ in place of $({\bX}^{(j)}_i)^T{D}_1^{(j)}\vb$ and $({\bX}_i)^T{D}_2\vb$ respectively $C_3$ is an absolute constant, and the last inequality holds because $\epsilon_i$ are sub-gaussian. Next we provide an upper bound on $\sum\limits_{i=1}^{n_k}(A^{(j)}_i)^2$ and $\sum\limits_{i=1}^n A_i^2$. Note that
		\[
			\begin{aligned}
				\sum\limits_{i=1}^n(A^{(j)}_i)^2&=\vb^T{D}_1^{(j)}{X^TX}{D}_1^{(j)}\vb =\vb^T((S^{(j)}_{X})^{-1}-({\Sigma})^{-1})n_kS^{(j)}_{X}((S^{(j)}_{X})^{-1}-({\Sigma})^{-1})\vb \\
				&=n_k\vb^T{\Sigma}^{-1}({\Sigma}-S^{(j)}_{X})(S^{(j)}_{X})^{-1}({\Sigma}-S^{(j)}_{X}){\Sigma}^{-1}\vb,
			\end{aligned}
		\]
		and similarly,
		\[
			\sum\limits_{i=1}^n A_i^2=n\vb^T{\Sigma}^{-1}({\Sigma}-{S}_{X})({S}_{X})^{-1}({\Sigma}-{S}_{X}){\Sigma}^{-1}\vb.
		\]
		For any $\tau\in\RR$, define the event $\cE^{(j)}=\{\|(S^{(j)}_{X})^{-1}\|_2\leq2/C_{\min}\}\cap\{\|S^{(j)}_{X}-{\Sigma}\|_2\leq (\delta_1 \vee \delta_1^2)\}$ for all $j=1,\ldots,k$, where $\delta_1=C_1\sqrt{d/n_k}+\tau/\sqrt{n_k}$, and the event $\cE=\{\|({S_{X}})^{-1}\|_2\leq2/C_{\min}\}\cap\{\|{S_{X}}-{\Sigma}\|_2<(\delta_2 \vee \delta_2^2)\}$, where $\delta_2=C_1\sqrt{d/n}+\tau/\sqrt n$. On $\cE^{(j)}$ and $\cE$, we have respectively
		\[
			\sum\limits_{i=1}^{n_k} (A^{(j)}_i)^2 \le\frac{2n_k}{C_{\min}^3} (\delta_1 \vee \delta_1^2)^2\ \text{and}\ \sum\limits_{i=1}^n A_i^2 \le\frac{2n}{C_{\min}^3} (\delta_2 \vee \delta_2^2)^2.
		\]
		Therefore from Equation (\ref{eq:mgf}) and (\ref{eq:mgf2}) we obtain
		\[
			\EE\left(\exp(\lambda({{D}_1^{(j)}\bf v})^T({X^{(j)T}}{\bepsilon}^{(j)}/n_k))\ind\{\cE^{(j)}\}\right)\leq\exp\left(\frac{2C_3\lambda^2s_1^2}{C_{\min}^3n_k} (\delta_1 \vee\delta_1^2)^2\right)
		\]
		and
		\[
			\EE\left(\exp(\lambda({{D}_2\bf v})^T({X}^T{\bepsilon}/n))\ind\{\cE\}\right)\leq\exp\left(\frac{2C_3\lambda^2s_1^2}{C_{\min}^3N}(\delta_2\vee \delta_2^2)^2\right).
		\]
		
		In addition, according to Lemma \ref{lem:cctts} and \ref{lem:mineg}, the probability of both $(\cE^{(j)})^{c}$ and $\cE^c$ are very small. More specifically,
		\[
			\PP(\cE^c)\le \exp(-cn)+\exp(-c_1\tau^2)\ \text{and}\ \PP((\cE^{(j)})^c)\le \exp(-cn/k)+\exp(-c_1\tau^2).
		\]
	Let $\cE_0:=\bigcap\limits_{j=1}^k \cE^{(j)}$. An application of the Chernoff bound trick leads us to the following inequality.
	\[
		\begin{aligned}
	 	&\PP\left(\left\{\frac{1}{k}\sum\limits_{j=1}^k({{D}_1^{(j)}\bf v})^T({X^{(j)T}}\bepsilon^{(j)})/n_k>t/2\right\}\cap \cE_0\right)\\
		&\le\exp(-\lambda t/2)\prod\limits_{j=1}^k\EE\left(\exp\left(\frac{\lambda}{k}({{D}_1^{(j)}\bv})^T({X^{(j)T}}{\bepsilon}^{(j)}/n_k)\right)\ind\{\cE^{(j)}\}\right)\\
		&\leq\exp\left(-\lambda t/2+\frac{2C_3\lambda^2s_1^2}{C_{\min}^3n}(\delta_1\vee \delta_1^2)^2\right).
		\end{aligned}
	\]
	Minimize the right hand side by $\lambda$, then we have
	\[
		\PP\left(\left\{\frac{1}{k}\sum\limits_{j=1}^k({{D}_1^{(j)}\bv})^T({X^{(j)T}}\bepsilon^{(j)})/n_k>t/2\right\}\cap \cE_0\right)\leq\exp\left(-\frac{t^2C_{\min}^3n}{32C_3s_1^2 (\delta_1\vee \delta_1^2)^2}\right).
	\]
	Consider the $1/2-$net of $\RR^p$, denoted by $\mathcal{N}(1/2)$. Again it is known that $|\cN(1/2)|<6^p$. Using the maximal inequality, we have
	\[
		\begin{aligned}
			&\PP\left(\left\{\bigl\|\frac{1}{k}\sum\limits_{j=1}^k({X}^{(j)}{D}_1^{(j)})^T\bepsilon^{(j)}/n_k\bigr\|_2>t/2\right\}\cap \cE_0\right)\\
			&=\sup\limits_{\|\vb\|_2=1}\PP\left(\left\{\frac{1}{k}\sum\limits_{j=1}^k({{D}_1^{(j)}\bv})^T({X^{(j)T}}\bepsilon^{(j)})/n_k>t/2\right\}\cap\cE_0\right)\\
			&\le\sup\limits_{\vb\in\cN(1/2)}\PP\left(\left\{\frac{1}{k}\sum\limits_{j=1}^k({{D}_1^{(j)}\bv})^T({X^{(j)T}}\bepsilon^{(j)})/n_k>t/4\right\}\cap\cE_0\right)\\
			&\le \exp\left(d\log(6)-\frac{t^2C_{\min}^3n}{32C_3s_1^2 (\delta_1\vee \delta_1^2)^2}\right).
		\end{aligned}
	\]
	Proceeding in an analogous fashion, we obtain
	\[
		\PP\left(\left\{\|({X}{D}_2)^T\bepsilon/n\|_2>t/2\right\}\cap \cE\right)\le\exp\left(d\log(6)-\frac{t^2C_{\min}^3n}{32C_3s_1^2 (\delta_2 \vee \delta_2^2)^2}\right).
	\]

	\end{proof}
	
	\begin{lemma}
		\label{lem:glmb}
		Following the same notation as in the proof of Theorem \ref{thm:glm2},
		\[
			\PP(\left\{\|\bB\|_2>t_1\right\}\cap \cA)\le  2\exp\left(d\log(6)-\frac{C_{\min}^4L_{\min}^2nt_1^2}{128\phi U_2C_{\max} (\delta_1\vee \delta_1^2)^2}\right).
		\]
	\end{lemma}
	
	\begin{proof}
		By Lemma \ref{lem:mgf}, for any $\lambda \in \RR$ and $\vb$ such that $\|\vb\|_2=1$, we have
		\[
			\begin{aligned}
				\EE\left(\exp(\lambda({{D}_1^{(j)}\bf v})^T({X^{(j)T}}{\bepsilon}^{(j)}/n_k))\mid{X}^{(j)}\right)&=\prod\limits_{i=1}^{n_k}\EE\left(\exp((\lambda\bX^{(j)}_i/n_k)^T({D}^{(j)}\vb)\epsilon_i)\mid{X}^{(j)}\right)\\
				&\le \exp\left(\phi U\lambda^2\sum\limits_{i=1}^{n_k} (A^{(j)}_i)^2/n_k^2\right)
			\end{aligned}
		\]
		and
		\[
			\label{mgf2}
			\begin{aligned}
				\EE\left(\exp(\lambda({{D}_2\bf v})^T({X}^T{\bepsilon}/n))\mid{X}\right)&=\prod\limits_{i=1}^{n}\EE\left(\exp((\lambda\bX_i/n)^T({D}_2\vb)\epsilon_i)\mid{X}\right)\\
				&\le\exp\left(\phi U\lambda^2\sum\limits_{i=1}^{n} A_i^2/n^2\right),
			\end{aligned}
		\]
		where we write $A^{(j)}_i$ and $A_i$ in place of $({X}^{(j)}_i)^T{D}_1^{(j)}\vb$ and $({X}_i)^T{D}_2\vb$ respectively. Next we give a upper bound on $\sum\limits_{i=1}^{n_k} (A_i^{(j)})^2$ and $\sum\limits_{i=1}^n A_i^2$. Note that
		\[
			\begin{aligned}
				\sum\limits_{i=1}^{n_k}(A^{(j)}_i)^2&=\vb^T{D}_1^{(j)}{X^TX}D_1^{(j)}\vb\\
				&=\vb^T(({S}^{(j)})^{-1}-\Sigma^{-1})nS_{X}(({S}^{(j)})^{-1}-\Sigma^{-1})\vb \\
				&=n\vb^T\Sigma^{-1}(\Sigma-{S}^{(j)})({S}^{(j)})^{-1}S_{X}^{(j)}({S}^{(j)})^{-1}(\Sigma-{S}^{(j)}){\Sigma}^{-1}\vb.
			\end{aligned}
		\]
		Similarly,
		\[
			\sum\limits_{i=1}^n A_i^2=n{\vb}^T\Sigma^{-1}(\Sigma-S){S}^{-1}S_{X}{S}^{-1}(\Sigma-S)\Sigma^{-1}\vb.
		\]
		On $\cE^{(j)}$ and $\cE$, we have respectively
		\[
			\sum\limits_{i=1}^{n_k} (A^{(j)}_i)^2 \leq\frac{8C_{\max}n_k}{C_{\min}^4L_{\min}^2}(\delta_1 \vee \delta_1^2)^2\ \text{and}\ \sum\limits_{i=1}^n A_i^2 \leq\frac{8C_{\max}n}{C_{\min}^4L_{\min}^2}(\delta_2\vee \delta_2^2)^2.
		\]	
		Then it follows that
		\[
			\EE\left(\exp(\lambda({{D}_1^{(j)}\bv})^T({X^{(j)T}}{\bepsilon}^{(j)}/n_k))\ind\{\cE^{(j)}\}\right)\leq\exp\left(\frac{8\phi UC_{\max}\lambda^2}{C_{\min}^4L_{\min}^2n_k}(\delta_1\vee \delta_1^2)^2\right)
		\]
		and
		\[
			\EE\left(\exp(\lambda({{D}_2\bf v})^T({X}^T{\bepsilon}/n))\ind\{\cE\}\right)\leq\exp\left(\frac{8\phi UC_{\max}\lambda^2}{C_{\min}^4L_{\min}^2n}(\delta_2 \vee\delta_2^2)^2\right).
		\]
		Now we follow exactly the same steps as in the OLS part. Denote $\cap_{j=1}^k \cE_j$ by $\cE_0$. An application of the Chernoff bound technique and the maximal inequality leads us to the following inequality.
		\[
			\PP\left(\left\{\|\frac{1}{k}\sum\limits_{j=1}^k({X}^{(j)}{D}_1^{(j)})^T\bepsilon^{(j)}/n_k\|_2>t/2\right\}\cap \cE_0\right)\le\exp\left(d\log(6)-\frac{C_{\min}^4L_{\min}^2nt^2}{128\phi U_2C_{\max}(\delta_1 \vee \delta_1^2)^2}\right)
		\]
		and
		\[
			\PP\left(\left\{\|({X}{D}_2)^T\bepsilon/n\|_2>t/2\right\}\cap \cE\right)\leq\exp\left(d\log(6)-\frac{C_{\min}^4L_{\min}^2nt^2}{128\phi U_2C_{\max}(\delta_2 \vee \delta_2^2)^2}\right).
		\]
	 	We have thus derived an upper bound for $\|\bB\|_2$ that holds with high probability. Specifically,
		\[
			\begin{aligned}
				\PP(\left\{\|\bB\|_2>t_1\right\}\cap \cA)&\le\PP\left(\left\{\|\frac{1}{k}\sum\limits_{j=1}^k({X}^{(j)}{D}_1^{(j)})^T\bepsilon^{(j)}/n_k\|_2>\frac{t_1}{2}\right\}\cap\cE_0\right)\\
				&+\PP\left(\left\{\|({X}{D}_2)^T\bepsilon/n\|_2>\frac{t_1}{2}\right\}\cap \cE\right) \le 2\exp\left(d\log(6)-\frac{C_{\min}^4L_{\min}^2nt_1^2}{128\phi U_2C_{\max} (\delta_1 \vee \delta_1^2)^2}\right).
			\end{aligned}
		\]
	\end{proof}

	\begin{lemma}
		\label{lem:glm1}
		Under Condition \ref{con:glm}, for $\tau\le L_{\min}/(8M C_{\max}U_3\sqrt{d})$ and sufficiently large $n$ and $d$ we have
		\[
			\begin{aligned}
				\PP(\|\hat\bbeta-\bbeta^*\|_2>\tau) &\le \exp\left(d\log 6-\frac{nC_{\min}^2L_{\min}^2\tau^2}{2^{11}C_{\max}U_2\phi}\right)+2\exp(-cn).
			\end{aligned}
		\]
	\end{lemma}
	\begin{proof}
The notation is that introduced in the proof of Theorem \ref{thm:glm2}. We further define $\Sigma(\bbeta):=\EE(b''({X^T}\bbeta)X{X^T})$ as well as the event $\cH:=\{\ell_n(\bbeta^*)>\max_{\substack{\bbeta\in\partial \cB_\tau}} \ell_n(\bbeta)\}$, where $\cB_\tau=\{\bbeta: \|\bbeta-\bbeta^*\|_2\le \tau\}$. Note that as long as the event $\cH$ holds, the MLE falls in $\cB_\tau$, therefore the proof strategy involves showing that $\PP(\cH)$ approaches 1 at certain rate. By the Taylor expansion,
		\[
			\begin{aligned}
				\ell_n(\bbeta)-\ell_n(\bbeta^*)&=(\bbeta-\bbeta^*)^T\bv-\frac{1}{2}(\bbeta-\bbeta^*)^TS(\tilde\bbeta)(\bbeta-\bbeta^*)\\
				&=(\bbeta-\bbeta^*)^T\bv-\frac{1}{2}(\bbeta-\bbeta^*)^TS(\bbeta^*)(\bbeta-\bbeta^*)-\frac{1}{2}(\bbeta-\bbeta^*)^T(S(\tilde\bbeta)-S(\bbeta^*))(\bbeta-\bbeta^*)\\
				&=A_1+A_2,
			\end{aligned}
		\]
		where $S(\bbeta)=(1/n) {X^T}D(X\bbeta)X$, $\tilde\bbeta$ is some vector between $\bbeta$ and $\bbeta^*$, $\bv=(1/n){X^T}(\bY-\bmu(X\bbeta^*))$, $A_1=(\bbeta-\bbeta^*)^T\bv-(1/2)(\bbeta-\bbeta^*)^TS(\bbeta^*)(\bbeta-\bbeta^*)$ and $A_2=-(1/2)(\bbeta-\bbeta^*)^T(S(\tilde\bbeta)-S(\bbeta^*))(\bbeta-\bbeta^*)$.
		
		Define the event $\cE:=\{\lambda_{\min}\left[S(\bbeta^*)\right]\ge L_{\min}/2\}$, where $L_{\min}$ is the same constant in Condition \ref{con:glm}. Note that by Condition \ref{con:glm} (ii), $\sqrt{b''(\bX_i^T\bbeta)}\bX_i$ is a sub-gaussian random vector. Then by Condition \ref{con:glm} (iii) and Lemma \ref{lem:mineg}, for sufficiently large $n$ and $d$ we have $\PP\left(\cE^c\right)\le \exp(-cn)$. Therefore on the event $\cE$,
		\[
			A_1 \le \tau(\|\bv\|_2-\frac{L_{\min}}{4}\tau).
		\]

		We next show that, under an appropriate choice of $\tau$, $|A_2|<L_{\min}\tau^2/8$ with high probability. We first consider Condition \ref{con:glm} (ii). Define $\cF:=\{\|X^TX/n\|_2\le 2C_{\max}\}$. By Lemma \ref{lem:maxeg}, we have $\PP(\cF^c)\le \exp(-cn)$. By Lemma \ref{lem:lipshitz}, on the event $\cF$, we have
		\[
			\begin{aligned}
				A_2 & \le \max_{\substack{1\le i\le n}} |b''(\bX_i^T\tilde\bbeta)-b''(\bX_i^T\bbeta^*)|C_{\max}\tau^2 \\
				& \le M U_3\sqrt{d}\|\tilde\bbeta-\bbeta^*\|_2\cdot C_{\max}\tau^2\\
				& \le M C_{\max}U_3\sqrt{d}\tau^3 \le \frac{L_{\min}\tau^2}{8},
			\end{aligned}
		\]
		where the last inequality holds if we choose $\tau\le L_{\min}/(8M C_{\max}U_3\sqrt{d})$. Now we obtain the following probabilistic upper bound on $\cH^c$, which we later prove to be negligible.
		\begin{equation}
			\label{eq:h}
			\begin{aligned}
				\PP(\cH^c)&\le \PP(\cH^c \cap \cE \cap \cF)+\PP(\cE^c)+\PP(\cF^c) \\
				& \le \PP\left(\left\{\|\bv\|_2\ge\frac{L_{\min}\tau}{8}\right\}\cap \cE \cap \cF\right)+\PP(\cE^c)+\PP(\cF^c).
			\end{aligned}
		\end{equation}
		Since each component of $\vb$ is a weighted average of i.i.d. random variables, the effect of concentration tends to make $\|\vb\|_2$ very small with large probability, which inspires us to study the moment generating function and apply the Chernoff bound technique. By Lemma \ref{lem:mgf}, for any constant $\ub\in\RR^d$, $\|\ub\|_2=1$ and let $a_i=\ub^T\bX_i$, then we have for any $t\in \RR$,
	\[
		\begin{aligned}
			\EE\left(\exp(t\langle\ub, \vb\rangle) \given X\right)&=\prod_{i=1}^n \EE\left(\exp\left(\frac{ta_i}{n}(Y_i-\mu(\bX_i^T\bbeta))\right)\given X\right)\\
			&\le \exp\left(\frac{\phi U_2t^2}{2n^2}\sum\limits_{i=1}^n a_i^2\right) \\
			&=\exp\left(\frac{\phi U_2t^2}{2n}\cdot\frac{\ub^TX^TX\ub}{n}\right).				\end{aligned}
	\]
		It follows that
		\[
			\EE\exp(t\langle\ub, \vb\rangle\ind\{\cE \cap \cF\})\le\exp\left(\frac{\phi C_{\max}U_2t^2}{2n}\right).
		\]
		By the Chernoff bound technique, we obtain
		\[
			\PP(\{\langle\ub, \vb\rangle>\epsilon\}\cap \cE\cap \cF)\le \exp\left(-\frac{n\epsilon^2}{8C_{\max}U_2\phi}\right).
		\]
		Consider a $1/2-$net of $\RR^d$, denoted by $\cN(1/2)$. Since
		\[
			\|\vb\|_2=\max_{\|\ub\|_2=1}\langle\ub, \vb\rangle\le 2\max_{\ub\in\cN(1/2)}\langle\ub, \vb\rangle,
		\]
		it follows that
		\[
			\begin{aligned}
				\PP(\{\|\vb\|_2>\frac{L_{\min}\tau}{8}\}\cap \cE\cap \cF)&\le\PP\left(\left\{\max_{\substack{\ub\in\cN(1/2)}} \langle\ub, \vb\rangle>\frac{L_{\min}\tau}{16}\right\}\cap \cE\cap \cF\right)\\
				&\le 6^d\exp\left(-\frac{nL_{\min}^2\tau^2}{2^{10}\phi C_{\max}U_2}\right)\\
				&=\exp\left(d\log 6-\frac{nC_{\min}^2L_{\min}^2\tau^2}{2^{11}C_{\max}U_2\phi}\right).
			\end{aligned}
		\]
		Finally combining the result above with Equation (\ref{eq:h}) delivers the conclusion.
	\end{proof}	

	\begin{remark}
		\label{rem:glm1}
	 	Simple calculation shows that when $d=o(\sqrt{n})$, $\|\hat\bbeta-\bbeta^*\|_2=O_{\PP}(\sqrt{d/n})$. When $d$ is a fixed constant, $\|\hat\bbeta-\bbeta^*\|_2=O_{\PP}(\sqrt{1/n})$.
	\end{remark}

\end{appendix}

\end{document}